\documentclass[reqno]{amsart}[12pt]
\usepackage[paper=letterpaper,margin=.9in]{geometry} 
\usepackage[utf8]{inputenc}
\usepackage{lmodern}
\usepackage{array}
\usepackage{adjustbox}
\usepackage{graphicx}
\usepackage{amssymb, amsfonts, amsthm, amsmath}
\usepackage{bbm}
\usepackage[T1]{fontenc}
\usepackage{enumitem}
\usepackage[english]{babel}
\usepackage[pdfborder={0 0 0}]{hyperref}
\hypersetup{colorlinks,linkcolor={blue},citecolor={blue},urlcolor={red}}  
\numberwithin{equation}{section}
\usepackage{float}
\usepackage{comment}
\usepackage{tikz}
\usepackage{bbm}
\usepackage{caption,subcaption}
\usepackage{mathtools}
\definecolor{aoe}{rgb}{0.0, 0.5, 0.0}
\usetikzlibrary{decorations.markings}
\usetikzlibrary{calc,angles,positioning}
\usetikzlibrary{decorations.pathreplacing}
\newcommand{\N}{\mathsf{N}}
\newcommand{\Q}{Q_{\ell, r}}
\newcommand{\e}{\varepsilon}
\newcommand{\mfI}{\mathfrak{I}}
\newcommand{\mfIHS}{\mathfrak{I}^{\text{HS}}}
\newcommand{\mfIuf}{\overline{\mathfrak{I}}}
\newcommand{\blk}[1]{\mu_{(n)}}
\newcommand{\blkHS}[1]{\mu_{(n)}^{\text{HS}}}

\newcommand{\Muf}{\overline{\mathcal{M}}_1 (\SS)}
\newcommand{\Ber}{\text{Ber}}
\newcommand{\mbP}{\mathbb{P}}
\newcommand{\weta}{\wtil{\eta}}
\newcommand{\wxi}{\wtil{\xi}}
\newcommand{\ZZ}{\mathbb{Z}}
\newcommand{\PP}{\mathcal{P}}

\newcommand{\II}{\mathcal{I}}

\newcommand{\IIt}{\widetilde{\mathcal{I}}}
\newcommand{\IIf}{\mathcal{I}^{\HS}}

\newcommand{\MM}{\mathcal{M}}
\newcommand{\Id}{\mathsf{Id}}
\newcommand{\qNB}{q\text{NB}}
\renewcommand{\L}{\mathsf{L}}
\newcommand{\GG}{\mathcal{G}}
\renewcommand{\SS}{\mathcal{S}}
\renewcommand{\mod}[2]{#1_{[#2]}}

\newcommand{\HS}{\text{HS}}

\newcommand{\bLambda}{\boldsymbol{\Lambda}}
\newcommand{\Sbar}{\overline{S}}
\newcommand{\Ibar}{\overline{\II}}
\newcommand{\mubar}{\overline{\mu}}
\newcommand{\Pbar}{\overline{\mathcal{P}}}
\newcommand{\wtil}[1]{\widetilde{#1}}
\newtheorem{theorem}{Theorem}[section]

\newtheorem{defin}[theorem]{Definition}
\newtheorem{prop}[theorem]{Proposition}
\newtheorem{cor}[theorem]{Corollary}
\newtheorem{lemma}[theorem]{Lemma}

\newtheorem{remark}[theorem]{Remark}

\newcommand{\red}[1]{\textcolor{black}{#1}}

\begin{document}
\title{Classification of Stationary distributions for the stochastic vertex models}
\author{Yier Lin}
\address[Yier Lin]{Department of Statistics, University of Chicago}
\email{ylin10@uchicago.edu}
\begin{abstract}
	In this paper, we study the stationary distributions for the stochastic vertex models. Our main focus is the
	stochastic six vertex (S6V) model. We show that the extremal stationary distributions of the S6V model are given by product Bernoulli measures. Moreover, for the S6V model under a moving frame \red{of} speed $1$, we show that the extremal stationary distributions are given by product Bernoulli measures and blocking measures. 
	\red{Finally}, we generalize our results to the stochastic higher spin six vertex model.  
	Our proof relies on the coupling of the S6V models introduced in \cite{aggarwal2020limit}, the analysis of current and the method of fusion. 
\end{abstract}
\maketitle
\section{Introduction} 
Stationary distribution is an important object for Markov processes and particularly for interacting particle systems. \red{A stationary distribution of a Markov process is a probability distribution that remains unchanged in the Markov process as time progresses.} When the state space is countable and the Markov process is irreducible, it is known that the existence of a stationary distribution is equivalent to the property of positive recurrence. Moreover, the stationary distribution is unique. However, in many examples of interacting particle systems, the state space is uncountable and there are an infinite number of stationary distributions. This paper intends to study and classify the set of stationary distributions for a family of interacting particle systems -- the stochastic six vertex (S6V) model and its higher spin generalizations. 
\subsection{The S6V model}
\label{sec:s6v}
The S6V model is a classical model in two dimension statistical physics. The model was introduced by \cite{gwa1992six} as a special case of the six vertex model \cite{Lieb74, Bax16}. For each vertex in $\mathbb{Z}^2$, 
\red{we tile it with one of the six vertex configurations  
in Figure \ref{fig:vertexconfig}.} 
\red{For a vertex configuration, we view the lines from the left and bottom of the vertex as \emph{input lines} and view the lines to the right and above as \emph{output lines}. The vertex configurations among different vertices need to be compatible in the sense that the lines keep flowing. More precisely, the vertex configuration at $(x, y)$ has a horizontal output line iff the vertex configuration at $(x+1, y)$ has a horizontal input line. Similarly, the vertex configuration at $(x, y)$ has a vertical output line iff the vertex configuration at $(x, y+1)$ has a vertical input line.} 
The weight of each \red{vertex} configuration is parameterized by two parameters $b_1 , b_2 \in (0, 1)$.  We say that the vertex configuration is \emph{conservative} since the total number of input lines that flow into the vertex is always equal to the number of output lines. We call the six vertex model 
\emph{stochastic} since given the number of input lines from the left and bottom, the sum of the weights of all possible \red{vertex} configurations equals $1$.  
\begin{figure}[ht]
	\centering
	\begin{adjustbox}{valign=t}
		\begin{tabular}{|c|c|c|c|c|c|c|}
			\hline
			Type & I & II & III & IV & V & VI \\
			\hline
			\begin{tikzpicture}[scale = 1.5]
			\draw[fill][white] (0.5, 0) circle (0.05);
			\draw[thick][white] (0, 0) -- (1,0);
			\draw[thick][white] (0.5, -0.5) -- 
			(0.5,0.5);
			\node at (0.5, 0) {Configuration};
			\end{tikzpicture}
			&
			\begin{tikzpicture}[scale = 1.2]
			\draw[fill] (0.5, 0) circle (0.05);
			\draw[thick] (0, 0) -- (1,0);
			\draw[thick] (0.5, -0.5) -- (0.5,0.5);
			\end{tikzpicture}
			&
			\begin{tikzpicture}[scale = 1.2]
			\draw[thick][white] (0, 0) -- (1,0);
			\draw[thick][white] (0.5, -0.5) -- (0.5,0.5);
			\draw[fill] (0.5, 0) circle (0.05);
			\end{tikzpicture}
			&
			\begin{tikzpicture}[scale = 1.2]
			\draw[thick] (0, 0) -- (1,0);
			\draw[thick][white] (0.5, -0.5) -- (0.5,0.5);
			\draw[fill] (0.5, 0) circle (0.05);
			\end{tikzpicture}
			&
			\begin{tikzpicture}[scale = 1.2]
			\draw[thick] (0, 0) -- (0.5,0);
			\draw[thick] (0.5, 0) -- (0.5, 0.5);
			\draw[thick][white] (0.5, 0) -- (1, 0);
			\draw[thick][white] (0.5, -0.5) -- (0.5, 0);
			(0.5,0.5);
			\draw[fill] (0.5, 0) circle (0.05);
			\end{tikzpicture}
			&
			\begin{tikzpicture}[scale = 1.2]
			\draw[thick][white] (0, 0) -- (1,0);
			\draw[thick] (0.5, -0.5) -- (0.5,0.5);
			\draw[fill] (0.5, 0) circle (0.05);
			\end{tikzpicture}
			&
			\begin{tikzpicture}[scale = 1.2]
			\draw[thick][white] (0, 0) -- (0.5,0);
			\draw[thick][white] (0.5, 0) -- (0.5, 0.5);
			\draw[thick] (0.5, 0) -- (1, 0);
			\draw[thick] (0.5, -0.5) -- (0.5, 0);
			(0.5,0.5);
			\draw[fill] (0.5, 0) circle (0.05);
			\end{tikzpicture}
			\\
			\hline
			Weight 
			& 1 & 1 & $b_2$ & $1- b_2$ & $b_1$ & $1-b_1$\\
			\hline
\end{tabular}
\end{adjustbox}
\caption{Six types of configurations for a vertex in $\mathbb{Z}^2$.}
	\label{fig:vertexconfig}
\end{figure}

By \cite{borodin2016stochastic}, it is known that  the S6V model  belongs to the Kardar-Parisi-Zhang (KPZ) universality class  -- a class of models that manifest the universal statistical behavior (such as Tracy-Widom fluctuation) in their long time / large-scale limit, see \cite{KPZ86, corwin2012kardar} for survey. 
It has been shown that under weakly asymmetric scaling, the S6V model 
converges  to the KPZ equation \cite{corwin2020stochastic}\red{.} 
On the other hand, it is shown that under a different scaling regime, the S6V model converges to the stochastic telegraph equation \cite{borodin2019stochastic, shen2019stochastic}. Other recent works on the S6V model include \cite{borodin2016stochastic, aggarwal2017convergence, aggarwal2018current, lin2019markov, aggarwal2020limit, dimitrov2020two, kuan2021short}.

In this paper, we interpret the S6V model as an interacting particle system. 
We consider the vertices on the integer lattice points of the upper half \red{plane $\mathbb{Z} \times 
\mathbb{Z}_{\geq 0}$.} We assume that there is no input line entering from the left and 
there are a finite number of input lines entering from the bottom. 
We first sample the output lines of a vertex whose input line has been specified, according to the stochastic weights given in Figure \ref{fig:vertexconfig}. Proceeding with this sampling, we get a collection of \red{upright} paths. 
By viewing the horizontal axis as space and vertical axis as time, these lines can be viewed as trajectories of particles. We cut the upper half plane by the dotted lines $y = t - \frac{1}{2}$ \red{for $t \in \mathbb{Z}_{\geq 0}$}, the intersection of these lines with the upright paths represents the location of the particles at time $t$, see Figure \ref{fig:S6Vparticle}.

To state our result, let us begin with a \red{concrete} definition of the S6V model. Our definition follows \cite[Definition 2.1]{aggarwal2020limit}, which intrinsically goes back to \cite[Section 2.2]{borodin2016stochastic}. 
\begin{figure}[ht]
\centering
\begin{tikzpicture}
\draw[thick][aoe][dashed][->] (-1, 0) -- (12, 0);
\draw[thick][aoe][dashed][->] (1, -1) -- (1, 4);
\draw[thick](2, -0.5) -- (2, 0);
\draw[thick](4, -0.5) -- (4, 0);
\draw[thick] (6, -0.5) -- (6, 0.5);
\draw[thick](8, -0.5) -- (8, 0);
\draw[thick] (2, 0) -- (4, 0);
\draw[thick] (4, 0) -- (4, 0.5);
\draw[thick] (4,0) -- (5,0);
\draw[thick] (5,0) -- (5, 0.5);
\draw[thick] (8, 0) -- (9, 0);
\draw[thick] (9, 0) -- (9, 0.5);
\node at (1, 4.3) {$y$};
\draw[thick] (4, 0.5) -- (4, 1.5);
\draw[thick] (5, 0.5) -- (5,1);
\draw[thick] (5, 1) -- (6, 1);
\draw[thick] (6, 1) -- (6, 1.5);
\draw[thick] (6, 0.5) -- (6, 1);
\draw[thick](6, 1) -- (9, 1);
\draw[thick] (9, 0.5) -- (9, 1);
\draw[thick] (9, 1) -- (10,1);
\draw[thick] (9, 1) -- (9, 1.5);
\draw[thick](10, 1) -- (10, 1.5);
\draw[thick] (4, 1.5) -- (4, 2);
\draw[thick](4, 2) -- (5, 2);
\draw[thick] (5, 2) -- (5, 2.5);
\draw[thick] (6,1.5) -- (6, 2);
\draw[thick] (6, 2) -- (7,2);
\draw[thick](7, 2) -- (7, 2.5);
\draw[thick] (10, 1.5) -- (10, 2);
\draw[thick] (10, 2) -- (11, 2);
\draw[thick] (11, 2) -- (11,2.5);
\draw[thick] (9, 1.5) -- (9, 2.5);
\draw[thick] (5, 2.5) -- (5, 3);
\draw[thick] (5, 3) -- (7, 3);
\draw[thick] (7, 2.5) -- (7, 3);
\draw[thick] (7 ,3) -- (7,4);
\draw[thick] (7, 3) -- (9,3);
\draw[thick] (9,3) -- (9, 4);
\draw[thick] (9, 2.5) -- (9, 3);
\draw[thick] (9, 3) -- (10, 3);
\draw[thick] (10, 3) -- (10, 4);
\draw[thick] (11, 2.5) -- (11, 3);
\draw[thick](11, 3) -- (11, 4);
\draw[dotted, thick] (0.5, -0.5) -- (12, -0.5);
\draw[dotted, thick] (0.5, 0.5) -- (12, 0.5);
\draw[dotted, thick] (0.5, 1.5) -- (12, 1.5);
\draw[dotted, thick] (0.5, 2.5) -- (12, 2.5);
\draw[dotted, thick] (0.5, 3.5) -- (12, 3.5);
\node at (0, -0.5) {$t = 0$};
\node at (0, 0.5) {$t = 1$};
\node at (0, 1.5) {$t = 2$};
\node at (0, 2.5) {$t = 3$};
\node at (0, 3.5) {$t = 4$};
\node at (5.5, -1.5) {\textbf{space}};
\node at (-1, 1.75) {\textbf{time}};
\draw[fill][red] (2, -0.5) circle (0.1); 
\draw[fill][red] (4, -0.5) circle (0.1); 
\draw[fill][red] (6, -0.5) circle (0.1); 
\draw[fill][red] (8, -0.5) circle (0.1); 
\draw[fill][red] (4, 0.5) circle (0.1); 
\draw[fill][red] (5, 0.5) circle (0.1); 
\draw[fill][red] (6, 0.5) circle (0.1); 
\draw[fill][red] (9, 0.5) circle (0.1); 
\draw[fill][red] (4, 1.5) circle (0.1); 
\draw[fill][red] (6, 1.5) circle (0.1); 
\draw[fill][red] (9, 1.5) circle (0.1); 
\draw[fill][red] (10, 1.5) circle (0.1); 
\draw[fill][red] (5, 2.5) circle (0.1); 
\draw[fill][red] (7, 2.5) circle (0.1); 
\draw[fill][red] (9, 2.5) circle (0.1); 
\draw[fill][red] (11, 2.5) circle (0.1); 
\draw[fill][red] (7, 3.5) circle (0.1); 
\draw[fill][red] (9, 3.5) circle (0.1); 
\draw[fill][red] (10, 3.5) circle (0.1); \draw[fill][red] (11, 3.5) circle (0.1);
\node at (12.3, 0) {$x$};
\end{tikzpicture}
\caption{The S6V model can be viewed as an interacting particle system.}
\label{fig:S6Vparticle}
\end{figure}
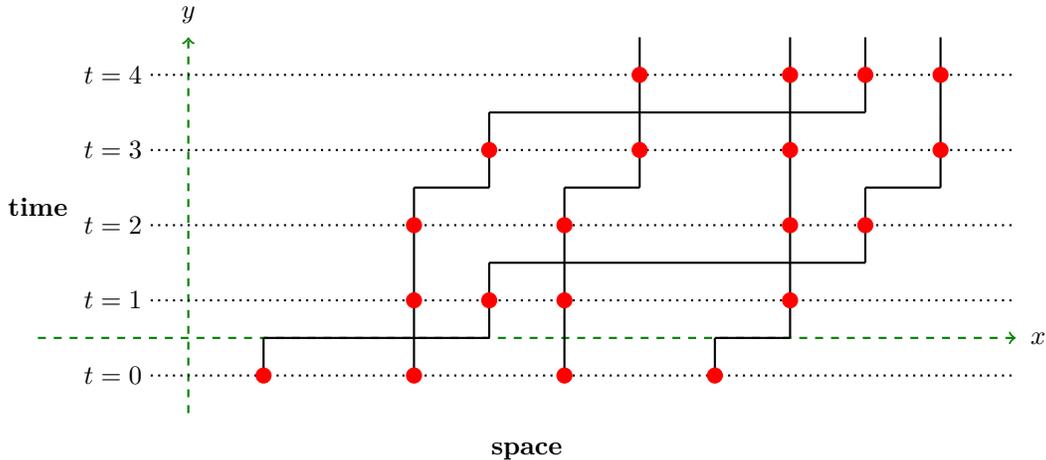 
\begin{defin}[S6V model with \red{finitely many}  particles]\label{def:s6v}
Fix $M, N \in \mathbb{Z}_{\geq 0}$ and $b_1, b_2 \in (0, 1)$. We define the S6V model 
as a discrete \red{time-homogeneous} Markov process $\{\mathbf{p}_t: t \in \mathbb{Z}_{\geq 0}\}$. Here, $\mathbf{p}_t = (p_t (-M) < \dots < p_t (N))$, where the integer $p_t (i)$ denotes the location of the $i$-th particle at time $t$. We sample independent random variables $\{\chi_t (x), j_t (x): x\in \mathbb{Z}, t \in \mathbb{Z}_{\geq 0}\}$ such that for all $t, x$,
\begin{equation}\label{eq:updaterule}
\begin{split}
&\mathbb{P}(\chi_t(x) = 1) = b_1, \qquad \mathbb{P}(\chi_t (x) = 0) = 1 - b_1, \\ &\mathbb{P}(j_t (x) = n) = (1 - b_2) b_2^{n-1}, \qquad \forall \, n \in \mathbb{Z}_{\geq 1}.
\end{split}
\end{equation}
It suffices to define the update rule from $\mathbf{p}_t$ to $\mathbf{p}_{t+1}$. For the convenience of notation, we set $p_{t+1}(-M-1) := -\infty$ and $p_{t+1} (N+1) := +\infty$. Given $\mathbf{p}_t$, we define $p_{t+1} (i), i = -M, \dots, N$ recursively: 
Suppose that we have defined 
$p_{t+1} (i)$, 
let us define $p_{t+1} (i+1)$ as follows: 
\begin{enumerate}[leftmargin = 2em]
\item 
If $p_{t+1} (i) < p_{t} (i+1)$ and $\chi_t (p_{t} (i+1)) = 1$, we set $p_{t+1} (i+1) := p_{t} (i+1)$. 
\item If $p_{t+1} (i) = p_{t} (i+1)$ or $\chi_t (p_{t} (i+1)) = 0$,  we set $p_{t+1} (i+1) :=  \min\big(p_{t} (i+1) + j_t (p_{t} (i+1)), p_{t} (i+2)\big)$.
\end{enumerate}
In other words, for the update from time $t$ to $t+1$, we run the update procedure from left to right. Each particle wakes up with probability $1 - b_1$. Once a particle has woken, it keeps moving one step to the right with probability $b_2$. If the $i$-th
particle counting from left hits the $(i+1)$-th particle,
then it stops immediately and the $(i+1)$-th 
particle 
moves one step to the right and keeps moving with probability $b_2$.
\end{defin}

We can also define the S6V model with \red{infinitely many} particles. \red{In this situation, at least one of $M, N$ is infinite.} The definition appears in \cite[Section 2]
{aggarwal2020limit}, 
see Section \ref{sec:infinite} for detail. 

Define the state space $\SS:= \{0, 1\}^{\mathbb{Z}}$. We introduce the S6V model occupation process $\eta_t = \{\eta_t (x)\}_{x \in \mathbb{Z}} \in \SS$ such that for every $t \in \mathbb{Z}_{\geq 0}$ and $x \in \mathbb{Z}_{\geq 0}$, 
\begin{equation}\label{eq:particleocc}
\eta_t (x) := \mathbbm{1}_{\{\exists i, p_t(i) = x\}}. 
\end{equation}
In other words, $\eta_t (x)$ describes whether the location $x$ is occupied at time $t$.
Note that the Markov processes $\eta_t$ and $\mathbf{p}_t$ are two equivalent ways of characterizing the S6V model.

The S6V model can be \red{viewed} as a discrete time version of the asymmetric simple exclusion process (ASEP). The ASEP is an interacting particle system on $\mathbb{Z}$ where each lattice \red{point} has either a particle or no particle. Fix $r, \ell > 0$, each particle jumps one step to the right with rate $r$ and jumps one step to the left with rate $\ell$.  The jump is blocked if the destination of the jump has already been occupied. The ASEP is known to be a continuous time limit of the S6V model \cite{borodin2016stochastic, aggarwal2017convergence}. More precisely, if we scale $b_1 = \e \ell, b_2 = \e r$, $t \to \e^{-1} t$ and impose a moving frame \red{of} speed $1$ to the right, under this moving frame, the S6V model converges to the ASEP as $\e \to 0$.

The above paragraph motivates us to define the S6V model under a moving frame of speed $1$.  Let us define a class of shift operators $\{\tau_n\}_{n \in \mathbb{Z}}$ such that $\tau_n: \mathcal{S} \to \mathcal{S}$ and $\tau_n(\eta) (i) := \eta(i+n)$ for arbitrary $n, i \in \mathbb{Z}$. We abbreviate $\tau_1$ as $\tau$. 
\begin{defin}[Shifted S6V model]
\label{def:shifted}
Recall the S6V model $(\eta_t)_{t \in \mathbb{Z}_{\geq 0}}$ from \eqref{eq:particleocc}. We define the shifted S6V model to be a Markov process $(\eta'_t)_{t \in \mathbb{Z}_{\geq 0}}$ such that $\eta'_t = \tau_t(\eta_t)$. It is the S6V model under a moving frame to the right \red{of} speed $1$.
\end{defin}


We seek to find all the stationary distributions for the S6V model and the shifted S6V model. 
We state our results in Section \ref{sec:mainresult}. In Section \ref{sec:extenHS}, we extend our results to the stochastic higher spin six vertex (SHS6V) model.
\subsection{Main results}\label{sec:mainresult}
We identify all the stationary distributions of the S6V model $(\eta_t)_{t \in \mathbb{Z}_{\geq 0}}$.
Let $\mathcal{M}_1 (\mathcal{S})$ denote the space of probability measure\red{s} on $\mathcal{S}$.  
We endow $\mathcal{M}_1(\mathcal{S})$ with the \red{topology of weak convergence of probability measures}.
Let $\II \subseteq \mathcal{M}_1 (\mathcal{S})$ denote the set of stationary distributions of $(\eta_t)_{t \in \mathbb{Z}_{\geq 0}}$. It is straightforward that 
$\II$ is compact and convex. Define the set of extreme point\red{s} of $\mathcal{I}$ to be 
\begin{equation*}
\II_e := \Big\{\nu \in \II: \text{ there does not exist } \nu_1 \neq \nu_2 \in \II \text{ and } \alpha \in (0, 1) \text{ such that } \nu = \alpha \nu_1 + (1-\alpha) \nu_2\Big\}.
\end{equation*}
By Krein–Milman theorem \cite{Krein1940},
 $\II$ equals the closed convex hull of $\II_e$. 
 
Let $\Ber(\rho)$ denote the Bernoulli distribution with mean $\rho$. Define $\mu_\rho := \bigotimes_{\mathbb{Z}} \Ber(\rho)$. We refer $\mu_\rho$ to be the product Bernoulli measures with density $\rho$.
\begin{theorem}\label{thm:st}
$\II_e = \{\mu_\rho\}_{\rho \in [0, 1]}$. 
\end{theorem}
\begin{remark}
In \cite[Theorem 3.6]{aggarwal2020limit}, the author proved that the set of translation invariant stationary distribution\red{s} of the S6V model is given by the closed convex hull of product Bernoulli measures. Our result says that there does not exist a stationary distribution for the S6V model which is not translation invariant. 
\end{remark}
We proceed to find the stationary distribution\red{s} of the shifted S6V model.
Define $q := b_1/b_2$ and assume $q \neq 1$. 
Define the inhomogeneous product measure  
\begin{equation}\label{eq:prodinhomo}
\mu_{*} := \bigotimes_{k \in \mathbb{Z}} \text{Ber}\Big(\frac{q^{-k}}{1 + q^{-k}}\Big). 
\end{equation}
Since $q \neq 1$, we have either $q \in (0, 1)$ or $q > 1$. Assume without loss of generality that $q > 1$. Using the Borel-Cantelli lemma, it is straightforward to verify  that under probability measure $\mu_{*}$, almost surely there exist two random numbers $-M <  0 <  N$ such that $\eta(i) = 1$ for $i \leq -M$ and $\eta(i) = 0$ for $i \geq N$. Let 
\begin{equation}\label{eq:setA}
A := \Big\{\eta: \sum_{i = -\infty}^{0} \big(1 - \eta(i)\big) + \sum_{i = 1}^\infty \eta(i) < \infty\Big\}.
\end{equation}
By the previous argument, we have $\mu_{*}(A) = 1.$
Consider a partition of the set $A = \cup_{n = -\infty}^\infty A_n$ where 
\begin{equation}\label{eq:setAn}
A_n  := \Big\{\eta: \sum_{i = -\infty}^{0} \big(1 - \eta(i)\big) + n=  \sum_{i = 1}^\infty \eta(i) < \infty\Big\}.
\end{equation}
Define the probability measure $\mu_{(n)}$ to be the projection of 
$\red{\mu_*}$
onto $A_n$, i.e. $\mu_{(n)}$ is supported on $A_n$ and for any measurable subset of $B \subseteq A_n$, $\mu_{(n)} (B ) := \frac{\mu_{*} (B)}{\mu_{*} (A_n)}.$ 
\red{It is straightforward that the sets $A_n$ for different $n$ are related by translation $\tau A_n =  A_{n-1}$. We also know that (see equation \eqref{eq:shifting}) for every $n \in \mathbb{Z}$, $\mu_{(n)} \circ \tau^{-1} = \mu_{(n-1)}$.}
We call the probability measures that are supported on $A$ \emph{blocking measures}. In particular, $\red{\{\mu_{(n)}\}_{n \in \mathbb{Z}}}$ are   example\red{s} of  blocking measure\red{s}.

Let $\mathfrak{I}$ be the set of stationary distributions for the shifted S6V model $(\eta_t')_{t \in \mathbb{Z}_{\geq 0}}$. The following theorem gives a complete characterization of $\mathfrak{I}$.

\begin{theorem}\label{thm:s6vmoving}
If $q \neq 1$, we have $\mfI_e = \{\mu_\rho\}_{\rho \in [0, 1]} \cup \{\mu_{(n)}\}_{n \in \mathbb{Z}}$.
\end{theorem}
\begin{remark}
Note that $\mfI_e$ is equal to the \red{set of} extremal stationary distribution\red{s} of ASEP with left jump rate $\ell$ and right jump rate $r$ such that $\frac{\ell}{r} = q$, see 
\cite[Theorem 1.4]{liggett1976coupling}. 
\end{remark}
\begin{remark}
If $q = 1$, $\mu_*$ is \red{equal to $\mu_{\frac{1}{2}}$.}  
The method of our proof \red{of Theorem \ref{thm:s6vmoving}} would imply that $\mfI_e = \{\mu_\rho\}_{\rho \in [0, 1]}$.
\end{remark}

\subsection{Extension to the stochastic higher spin six vertex (SHS6V) model} 
\label{sec:extenHS}
For the S6V model, as shown in Figure \ref{fig:vertexconfig}, 
there can be at most one line flowing in the horizontal and vertical direction. The SHS6V model \red{introduced in} \cite{corwin2016stochastic, BP18} removes this restriction and allows the vertex configurations to have up to $I$ lines in the vertical direction and up to $J$ lines in the horizontal directions, where $I, J$ are fixed positive integers.  

The SHS6V model is a four-parameter family of quantum integrable system \red{and} has been intensely studied in recent years\red{. See} \cite{corwin2016stochastic, kuniba2016stochastic, corwin2017kpz, Agg18, BP18, borodin2018inhomogeneous, borodin2018coloured, kuan2018algebraic,    borodin2020observables, lin2020kpz, lin2020stochastic, imamura2020stationary,  aggarwal2021deformed} \red{for reference.} To define the model, we \red{fix $I, J \in \mathbb{Z}_{\geq 1}$, $q, \alpha \in \mathbb{R}$}  and introduce the matrix $\L_\alpha^{I, J}$ with row and column both indexed by $(i_1, j_1), (i_2, j_2) \in \{0, 1, \dots, I\} \times \{0, 1, \dots, J\}$. The $i_1, j_1, i_2, j_2$ indicate the number of lines on the bottom, left, top and right of a vertex. Similar to the vertex configurations of the S6V model, we view the lines on the left and bottom of the vertex as input lines and view the lines to the right and above as output lines, see Figure \ref{fig:higherspin}. We assign the vertex configuration indexed by $i_1, j_1, i_2, j_2$ with weight $\L_{\alpha}^{I, J} (i_1, j_1; i_2, j_2)$.
\begin{figure}[ht]
\centering
\includegraphics[scale = 0.7]{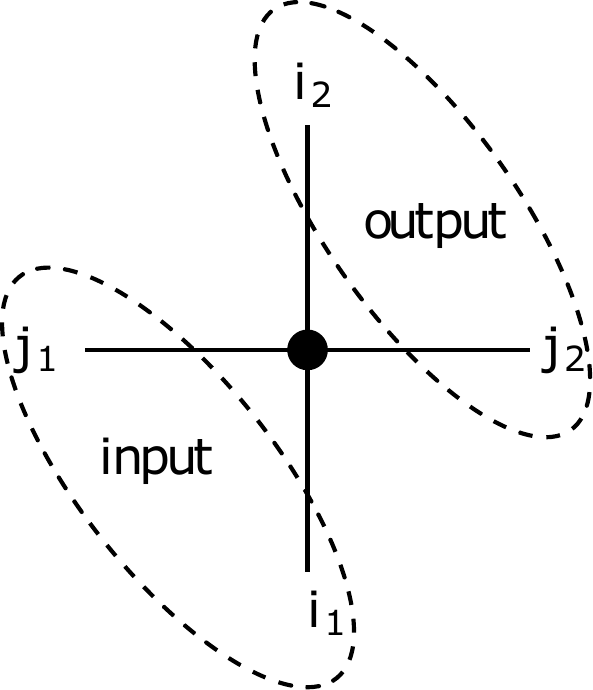}
\caption{The vertex configuration is indexed by four numbers $i_1, j_1, i_2, j_2$ and has weight $\L^{I, J}_\alpha (i_1, j_1; i_2, j_2)$. The vertex absorbs $i_1$ lines from the bottom, $j_1$ lines from the left and produces $i_2$ lines to above and $j_2$ lines to the right.}
\label{fig:higherspin}
\end{figure}
We proceed to specify the entries of the matrix $\L_{\alpha}^{I, J}$. \red{We d}efine the $q$-Pochhammer symbol. \red{For all $b, q \in \mathbb{R}$,}
\begin{equation*}
(b; q)_n := 
\begin{cases}
\prod_{k = 0}^{n-1} (1 - bq^k),  &\qquad n > 0,\\
1  &\qquad n = 0,\\
\prod_{k = 0}^{-n-1} (1-\red{b} q^{n+k})^{-1} &\qquad n < 0.
\end{cases}
\end{equation*}
\red{By \cite[Theorem 3.15]{corwin2016stochastic}},  \red{ t}he entries of $\L_\alpha^{I, J}$ are given by  
\begin{align*}
\L_{\alpha}^{I, J}(i_1, j_1; i_2, j_2) := &\mathbbm{1}_{\{i_1 + j_1 = i_2 + j_2\}} q^{\frac{2j_1 - j_1^2}{4} - \frac{2j_2 - j_2^2}{4} + \frac{i_2^2 + i_1^2}{4} + \frac{i_2 (j_2 - 1) + i_1 j_1}{2}}\\
&\times \frac{\nu^{j_1 - i_2} \alpha^{j_2 - j_1 + i_2} (-\alpha\nu^{-1}; q)_{j_2 - i_1}}{(q;q)_{i_2} (-\alpha; q)_{i_2 + j_2} (q^{J+1 - j_1 }; q)_{j_1 - j_2}}  { }_{4} \overline{\phi}_3 \bigg(\begin{matrix}
q^{-i_2};q^{-i_1}, -\alpha q^J, -q\nu \alpha^{-1}\\ \nu, q^{1 + j_2 - i_1}, q^{J + 1 - i_2 - j_2} 
\end{matrix} \bigg| q, q\bigg),
\end{align*}
where $\nu = q^{-I}$ and ${}_4 \bar{\phi}_3$ is the regularized terminating basic hyper-geometric series defined by 
\begin{align*}
_{r+1} \overline{\phi}_r \bigg(\begin{matrix}
q^{-n}, a_1, \dots, a_r\\ b_{\red{1}}, \dots, b_r 
\end{matrix}
\bigg|q, z \bigg) &= \sum_{k=0}^n z^k \frac{(q^{-n}; q)_k}{(q; q)_k} \prod_{i=1}^r (a_i; q)_k (b_i q^k; q)_{n-k}.
\end{align*}  
\red{Note that the dependence of $\mathsf{L}_{\alpha}^{I, J}$ on $q$ is hidden for brevity.} It turns out the matrix $\mathsf{L}_{\alpha}^{I, J}$ is \emph{stochastic} \red{if} either one of the following is satisfied. 
\begin{enumerate}[leftmargin = 2em]
\item\label{item:condition1} $q \in [0, 1)$ and $\alpha < - q^{-I- J + 1}$.
\item \label{item:condition2} $q > 1$ and $-q^{-I-J+1} < \alpha < 0$. 
\end{enumerate}
\red{When \eqref{item:condition1} or \eqref{item:condition2} holds, the stochasticity of the matrix is proved in \cite[Propsition 2.3]{corwin2016stochastic} and \cite[Corollary 1.4]{lin2020stochastic}. \red{Moreover, we have that the matrix $\L_{\alpha}^{I, J}$ is conservative in the sense that the entry $\L_{\alpha}^{I, J} (i_1, j_1; i_2, j_2)$ is zero if $i_1 + j_1 \neq i_2 + j_2$.} Although the following will not be the focus of this paper, we remark that \cite{corwin2016stochastic, BP18} also introduced the SHS6V model with $I, J$ being non-integers by analytic continuation. In this case, the row and column of the matrix $\L^{I, J}_{\alpha}$ is indexed by $(i_1, j_1), (i_2, j_2) \in \mathbb{Z}_{\geq 0} \times \mathbb{Z}_{\geq 0}$.}

For the rest of the section, we assume that $\alpha$ and $q$ satisfy either \eqref{item:condition1} or \eqref{item:condition2}.
Define $\GG := \{0, 1, \dots, I\}^{\mathbb{Z}}$. In the following, we define the SHS6V model as an interacting particle system \red{on $\mathbb{Z}$.} \red{The system} allows at most $I$ \red{number of} particles \red{to occupy} each \red{integer} lattice \red{point}, with the assumption that the total number of particles is finite. 
\begin{defin}[SHS6V model with \red{finitely many particles}]\label{def:fused}
The SHS6V model \red{is} 
a discrete time-homogeneous Markov process $g_t = \{g_t (x), x \in \mathbb{Z}\} \in \GG$, where $g_t (x)$ denotes the number of particles located at $x$ at time $t$. It suffices to specify the update rule from $g_t$ to $g_{t+1}$.
Assume that the leftmost particle in the configuration $g$ is at $x$, i.e. $g_t(x) > 0$ and $g_t(z) = 0$ for all $z < x$. Starting from $x$, we update $g_t(x)$ to $g_{t+1}(x)$ by setting $h_t(x) = 0$ and randomly choosing $g_{t+1}(x)$ according to the probability measure $\L^{I, J}_\alpha (g_t (x), h_t (x) = 0; g_{t+1} (x), h_t (x+1))$ where $h_t (x+1) := g_t (x) - g_{t+1} (x)$. Proceeding sequentially, we update $g_t (x+1)$ to $g_{t+1} (x+1)$ according to the probability $\L^{I, J}_\alpha (g_{t} (x+1), h_t (x+1); g_{t+1}(x+1), h_{t} (x+2))$ where $h_{t} (x+2) := g_{t} (x+1) + h_t (x+1) - g_{t+1} (x+1)$. Continuing for \red{the update of} $g_t (x+2), g_t (x+3), \dots$, we have defined the update rule from $g_t$ to $g_{t+1}$, see Figure \ref{fig:jmodel} for visualization of the update procedure. We call the discrete \textbf{time-homogeneous} Markov process  $g_t \in \mathbb{G}$ with the update rule defined above
	\textbf{the 
SHS6V model}.
\end{defin}
Similar to the S6V model, we can extend the definition of the SHS6V model to allow infinite number of particles, see the discussion after Proposition \ref{prop:fusion}. 
\begin{remark}
\red{One readily} checks that when $I = J = 1$, the vertex configurations that have non-zero weight reduce to the vertex configurations in Figure \ref{fig:vertexconfig}, with $b_1  = \red{\L_{\alpha}^{1, 1} (1, 0; 1, 0) = } \frac{1+ \alpha q}{1 + \alpha}$ and $b_2 = \red{\L_{\alpha}^{1, 1} (0, 1; 0, 1) = } \frac{\alpha + q^{-1}}{1 + \alpha}$. Hence, 
the SHS6V model $(g_t)_{t \in \mathbb{Z}_{\geq 0}}$ reduces to the S6V model $(\eta_t)_{t \in \mathbb{Z}_{\geq 0}}$.
\end{remark}
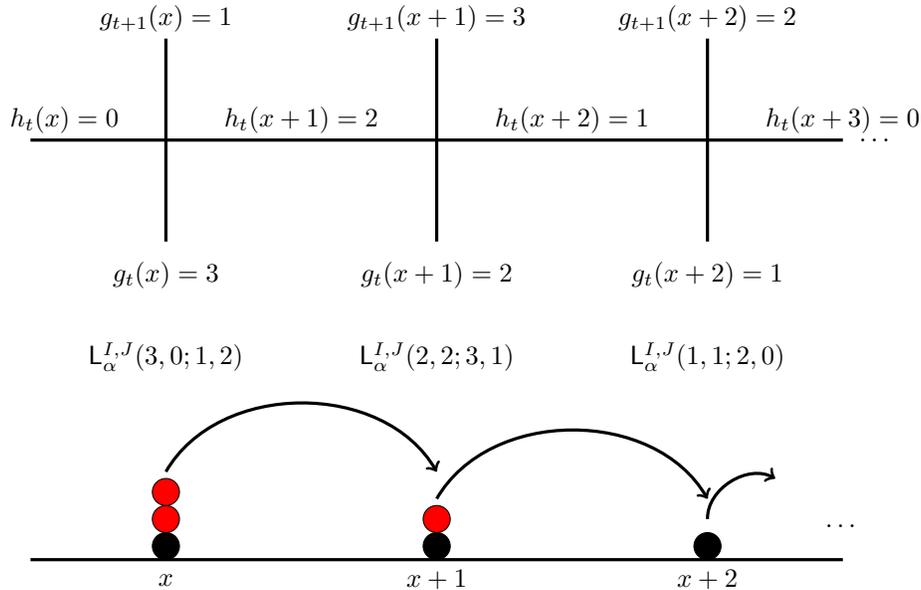
\begin{figure}[ht]
	\begin{tikzpicture}[scale = 0.9]
	\draw[very thick] (0, 4.5) -- (12, 4.5);
	\draw[very thick] (2, 3) -- (2, 6);
	\draw[very thick] (6, 3) -- (6, 6);
	\draw[very thick] (10, 3) -- (10, 6);
	\node at (2, 2.5)  {$g_t (x) = 3$};
	\node at (6, 2.5)  {$g_t (x+1) = 2$};
	\node at (10, 2.5)  {$g_t (x+2) = 1$};
	\node at (2, 6.3)  {$g_{t+1} (x) = 1$};
	\node at (6, 6.3)  {$g_{t+1} (x+1) = 3$};
	\node at (10, 6.3)  {$g_{t+1} (x+2) =2 $};
	\node at (0.5, 4.8) {$h_t (x) = 0$};
	\node at (4, 4.8) {$h_t (x+1) = 2$};
	\node at (8, 4.8) {$h_t (x+2) = 1$};
	\node at (12, 4.8) {$h_t (x+3) = 0$};
	\node at (12.5, 4.5) {\dots};
	\begin{scope}[yshift = -2em]
	\draw[very thick] (0, -1) -- (12, -1);
	
	\draw[fill]  (2, -0.8) circle (0.2);
	\draw[fill = red]  (2, -0.4) circle (0.2);
	\draw[fill = red]  (2, 0) circle (0.2);
	\draw[fill]  (6, -0.8) circle (0.2);
	
	\draw[fill = red]  (6, -0.4) circle (0.2);
	\draw[fill]  (10, -0.8) circle (0.2);
	\node at (2, -1.3) {$x$};
	\node at (6, -1.3) {$x+1$};
	\node at (10, -1.3) {$x+2$};
	
	\node at (12, -0.5) {\dots};
	\draw[bend left = 60, very thick, ->]  (2, 0.3) to (6, 0.3);
	\draw[bend left = 60, very thick, ->]  (6, -0.1) to (10, -0.1);
	\draw[bend left = 60, very thick, ->]  (10, -0.4) to (11, 0.2);
	\end{scope}
	\node at (2, 1.3) {$\L_\alpha^{I, J}(3, 0; 1, 2)$};
	\node at (6, 1.3) {$\L_\alpha^{I, J}(2, 2; 3, 1)$};
	\node at (10, 1.3) {$\L_\alpha^{I, J}(1, 1; 2, 0)$};
	
	\end{tikzpicture}
	\caption{A visualization of the sequential update rule for the SHS6V model in Definition \ref{def:fused}.  Assuming that $x$ is the location of the leftmost particle, we update sequentially for positions $x, x+1, x+2, \dots$ according to the stochastic matrix $\L^{I, J}_\alpha$. The red particles in the picture above will move one step to the right.}
	\label{fig:jmodel}
\end{figure}
We proceed to study the stationary distributions of the SHS6V model. 
We say that a random variable $X$ follows the \emph{$q$-negative Binomial distribution} if
\begin{equation}\label{eq:qNB}
\mathbb{P}(X = n) = p^n \frac{(b; q)_n (p; q)_\infty}{(q; q)_n (pb; q)_\infty}, \qquad \text{for all } n \in \mathbb{Z}_{\geq 0}. 
\end{equation}
for some parameters $p, b$. We use the notation $X \sim \qNB(b, p)$ for the above probability distribution.
When $p < 0$ and $b = q^{-K}$ for some positive integer $K$, $X$ only takes the value in the set $\{0, \dots, K\}$ since $\mathbb{P}(X = n) = 0$ for any $n > K$.

 For $\rho \in [0, 1]$, let $\mu^\HS_\rho$ denote the product $q$-negative binomial distribution
\begin{equation}\label{eq:qnegative}
\mu^\HS_\rho := \bigotimes_{x \in \mathbb{Z}} q\text{NB}\Big(q^{-I}, \frac{-q^{I} \rho}{1 - \rho}\Big).
\end{equation}  
\red{Note that $\mu_{\rho}^{\text{HS}}$ is equal to $\mu_\rho$ when $I = 1$.}
Let $\IIf$ denote the set of stationary distributions for the SHS6V model $(g_t)_{t \in \mathbb{Z}_{\geq 0}}$.
\begin{theorem}\label{thm:st1}
We have $\IIf_{e} = \{\mu^{\text{HS}}_\rho\}_{\rho \in [0, 1]}$.
\end{theorem}
\red{We slightly abuse the notation and let $\tau_n$ now denote the map from $\mathcal{G}$ to $\mathcal{G}$ such that we have $\tau_n g(i) = g(i+n)$ for any $g \in \GG$ and $i \in \mathbb{Z}$. We set $\tau = \tau_1$ for abbreviation.}
When $I = J$, we consider \red{the shifted SHS6V model $(g'_t)_{t \in \mathbb{Z}_{\geq 0}} = (\tau_t (g_t))_{t \in \mathbb{Z}_{\geq 0}}$, which is obtained by observing} the SHS6V model under a moving frame \red{of} speed $1$. Let $\mfIHS$ denote the set of stationary distribution\red{s} for shifted SHS6V model. 
We define the blocking measures for the shifted SHS6V model. Set
\begin{equation}\label{eq:mu*HS}
\mu_{*}^{\text{HS}} := \bigotimes_{k \in \mathbb{Z}} q\text{NB}(q^{-I}, -q^{\red{-kI+1}}).
\end{equation}
\red{Note that $\mu_*^{\text{HS}}$ is equal to $\mu_*$ when $I = 1$.} When $q >1$, let 
\begin{equation*}
G := \Big\{\eta: \sum_{i = -\infty}^{0} \big(I - g(i)\big) + \sum_{i = 1}^\infty g(i) < \infty\Big\}.
\end{equation*}
It is clear that $\mu_*^\HS$ is supported on $G$. Consider a partition of the set $G = \cup_{n = -\infty}^\infty G_n$, where 
\begin{equation*}
G_n  := \Big\{\eta: n + \sum_{i = -\infty}^{0} \big(I - g(i)\big) = \sum_{i = 1}^\infty g(i) < \infty\Big\}.
\end{equation*}
Note that $G_n$ are irreducible subspaces for the SHS6V model, we define $\mu^{\HS}_{(n)}$ to be the projection of $\mu_{*}^{\HS}$ onto $G_n$. \red{Note that $\tau G_n = G_{n - I}$. One can also verify that $\mu^{\text{HS}}_{(n)} \circ \tau^{-1} = \mu^{\text{HS}}_{(n - I)}$ for every $n \in \mathbb{Z}$.} 
\begin{theorem}\label{thm:blockingHS}
$\mfIHS$ contains \text{the closed convex hull of } $\big(\{\mu_{\rho}^{\text{HS}}\}_{\rho \in [0, 1]} \cup \{\mu_{(n)}^{\text{HS}}\}_{n \in \mathbb{Z}}\big)$. 
\end{theorem}

\begin{remark}\label{rmk:whycant}
Unlike Theorem \ref{thm:s6vmoving}, we \red{do not} know whether $\mfI^{\text{HS}}$ is the closed convex hull of  $\big(\{\mu_{\rho}^{\text{HS}}\}_{\rho \in [0, 1]} \cup \{\mu_{(n)}^{\text{HS}}\}_{n \in \mathbb{Z}}\big)$. \red{The major difficulty is that in the situation of the S6V model, each location allows either zero or one particle, and we can access the asymptotic density of the extremal non-translation invariant stationary distribution at $\pm \infty$ with the help of Lemma \ref{lem:temp3}, which helps us obtain Lemma \ref{lem:ordering}. 
For the SHS6V model, each location allows more than one particle, and we can not obtain the asymptotic density of the extremal non-translation invariant stationary distribution at $\pm \infty$ using the same approach.} We leave this problem to future work. 
\end{remark}
\red{
\begin{remark}
It would also be interesting to study the set of stationary distributions for the SHS6V model and shifted SHS6V model, with $I, J$ being non-integer. We believe that the product measures $\{\mu_\rho^{\text{HS}}\}_{\rho \in [0, 1]}$ are stationary for the SHS6V model. However, We are not sure if there are additional extremal stationary distributions for the SHS6V model. When $I = J$, the same result as in Theorem \ref{thm:blockingHS} should hold. 
\end{remark}
} 
\subsection{The proof idea} 

The proof for Theorem \ref{thm:st} and \ref{thm:s6vmoving} rely on two tools -- \emph{coupling} and \emph{current analysis}. The coupling method allows us to measure the difference between two stationary distributions. For example, via computing the generator of coupled ASEP, \cite[Theorem 1.4]{liggett1976coupling} classified all the stationary distributions of ASEP. 
By \cite[Section 2]{aggarwal2020limit}, there exists a coupling of two S6V models obtained via looking at the two-class S6V model. Unlike ASEP, the semigroup of the (coupled) S6V model is  complicated. Instead of \red{carrying out} a direct computation as \cite{liggett1976coupling}, we study 
the \emph{discrepancies} \red{of the coupled S6V model}.

The discrepancies refer to the locations where one of the S6V models in the coupled process has a particle and the other does not. There are two types of discrepancies, depending on whether we have a (particle, hole) or (hole, particle) in the coupled S6V model. The discrepancies play the role \red{of} second class particles in the two-class S6V model, so they \red{cannot} be created from nowhere and their behavior can be studied. \red{An important observation will be that} a pair of discrepancies with different types will annihilate each other if they end up at the same location. 

The current refers to the number of particles that move across a location at a given time. For the 
S6V model \red{and its multi-class generalization}, a nice property that will be used several times in our proof is that the current is upper bounded. In particular, it is either zero or one. \red{Note that this is a special property for the S6V model since for general interacting particle systems, one might allow multiple particles to move across a location at a given time. The reason behind is that the current of the S6V model at time $t$ and location $x$ is indicated by the number of horizontal edge of $(x, t)$ -- $(x+1, t)$ in the tiling interpretation of the S6V model, which is either zero or one.}

The starting point of our proof is Proposition \ref{prop:coramol}, which shows that for any stationary distribution of the coupled S6V model, we 
\red{cannot}
have both types of discrepancies. This would imply that all extremal translation invariant stationary distributions are product Bernoulli measures. Proposition \ref{prop:coramol} was proved in \cite[Corollary 3.9]{aggarwal2020limit}, we provide a self-contained proof which is similar to \cite{aggarwal2020limit}. 
The idea is via counting the number of discrepancies in an interval. On one hand, since the current is either zero or one, there can be at most two discrepancies entering the interval during each time of update for the coupled S6V model. On the other hand, if with positive probability there are discrepancies of different types, then by taking a long enough interval and translation invariance, the expected number of discrepancies that are annihilated each time can be very large. Hence, we can find an interval such that the expectation fo the number of discrepancies inside is \red{strictly} decreasing. This contradicts 
the stationarity, which implies that the expectation of the number of discrepancies in any interval must be the same as time evolves.

To prove Theorem \ref{thm:st}, we need to show that any extremal stationary distribution must be translation invariant. We adopt the idea of \cite{bramson2002characterization} and couple the S6V model with its horizontal shift. By tracking the discrepancies in the coupled process, we can show that the number of discrepancies inside $[-n, n\red{]}$ at time $\lfloor \sqrt{n} \rfloor$ is $o(n)$ as $n \to \infty$, see Proposition \ref{prop:liggettineq}.
Using this we can show that any extremal stationary distribution is either larger or equal or smaller than its horizontal shift \red{(there is a partial order on the space of probability measures, see Definition \ref{def:ordering})}. We rule out the first and last situation by noting that the expectation of the current of the S6V model at location $x$ is a monotone function of the particle configuration on the left of $x$. Hence, if the stationary distribution is larger or smaller than its horizontal shift, the expectations of the currents at neighboring positions are not equal. This shows that all stationary distributions of the S6V model have to be translation invariant. 

To prove Theorem \ref{thm:s6vmoving}, we first construct a stationary blocking measure by observing a local property of pseudo-stationarity, see Lemma \ref{lem:bs}. We prove \red{that} the blocking measure $\mu_*$ defined in \eqref{eq:prodinhomo} is stationary for the shifted S6V model by iterating this local property. This \red{will be} done in Section \ref{sec:blk}. 
In Section \ref{sec:classify}, we apply coupling and current analysis to show that the extremal stationary distributions of the shifted S6V model are given by $\{\mu_\rho\}_{\rho \in [0, 1]}$ and the projection of $\mu_*$ onto the irreducible subspace $A_n$.

In the end, we generalize our results to the SHS6V model. By the method of fusion \cite{kirillov1987exact, corwin2016stochastic}, it turns out that we can unfuse the SHS6V model into the unfused SHS6V model defined in Definition \ref{def:unfused}, see Proposition \ref{prop:fusion}. The unfused SHS6V model is an inhomogeneous version of the S6V model whose $b_1$ and $b_2$ parameters have period $I$ in space and $J$ in time. 
Using a similar argument used for proving Theorem \ref{thm:st}, we obtain all the extremal stationary distributions for the unfused SHS6V model. By the method of fusion, we then obtain all the extremal stationary distributions for the SHS6V model, see Theorem \ref{thm:st1}. For the shifted SHS6V model with $I = J$, one can also obtain the stationary blocking measure $\mu_*^{\HS}$ by using the method of fusion and this shows Theorem \ref{thm:blockingHS}. 
\subsection{Related literature}
The stationary distribution\red{s} of various interacting particle systems \red{have} been studied intensively over the past 50 years. For the symmetric exclusion process, the stationary distributions are well-understood by the work of \cite{liggett1973characterization, liggett1974characterization}. The question is more challenging in the asymmetric case.
For ASEP, the \red{set} of stationary distribution\red{s} is completely understood in the work of \cite{liggett1976coupling}, in particular there exists a stationary blocking measure which is non-translation invariant. \cite{bramson2002stationary, bramson2002characterization} studied the non-nearest jump particle system and provided sufficient conditions for the existence/non-existence of the blocking measures. \cite{bramson2005exclusion} studied the existence of the blocking measure for exclusion processes in higher dimensions. Recently, \cite{amir2021invariant} studied a multi-lane generalization of the exclusion process, which interpolates the exclusion process in one and two dimensions. 
For the multi-class interacting particle systems, the stationary distributions of the multi-class totally asymmetric exclusion process and multi-class ASEP were studied in \cite{angel2006stationary, ferrari2007stationary, ferrari2009multiclass, cantini2015matrix, martin2020stationary, corteel2022multiline}. Other results concerning the stationary distributions of different interacting particle systems include \cite{andjel1982invariant, derrida1993exact, frometa2019boundary, floreani2020orthogonal}.

The previous results are for continuous time interacting particle systems, while the stochastic vertex models are discrete time particle systems which have intricate dynamics. \cite{aggarwal2018current} showed that the product Bernoulli measures are stationary for the S6V model. In addition, \cite{aggarwal2020limit} showed that these are the only extremal translation invariant stationary distribution. The product $q$-negative binomial distribution was shown to be stationary for the SHS6V model in \cite{borodin2018inhomogeneous, imamura2020stationary, lin2020kpz}.
Our result provides a more complete characterization of the stationary distributions for the S6V and SHS6V model and their shifted versions.  

\subsection*{Outline}
\red{In Section \ref{sec:pre}, we explain the definition of the S6V model with infinitely many particles and recall the coupling for the S6V model introduced in \cite[Section 2]{aggarwal2020limit}.} In Section \ref{sec:translationinvariant}, we use the coupling method to show that any extremal translation invariant stationary distribution must be the product Bernoulli distribution. In Section \ref{sec:coupling}, we show that there is a partial order between any extremal stationary distribution and its horizontal shift. In Section \ref{sec:current}, we analyze the current of the S6V model and conclude Theorem \ref{thm:st}. In Section \ref{sec:blk}, we show that the probability measure $\mu_*$ is stationary for the shifted S6V model by proving a local property called pseudo-stationarity. In Section \ref{sec:classify}, we complete the proof of Theorem \ref{thm:s6vmoving}. In Section \ref{sec:sketch}, we extend our result for the S6V model to the SHS6V model and prove Theorem \ref{thm:st1} and \ref{thm:blockingHS}. In Appendix \ref{sec:productstat}, we explain how to show that the product Bernoulli measures are stationary for the S6V model. 

\subsection*{Notations}
For $\eta, \xi \in \mathcal{S}$, we say that $\eta \geq \xi$ if $\eta(x) \geq \xi(x)$ for any $x \in \mathbb{Z}$. We say that $\eta > \xi$ if $\eta \geq \xi$ and there exists $x \in \mathbb{Z}$ such that $\eta(x) > \xi(x)$. 
\subsection*{Acknowledgment}
The author thanks Amol Aggarwal, Ivan Corwin and Pablo Ferrari for helpful discussion. We also thank  Amol Aggarwal and Ivan Corwin for helpful comments on the paper. \red{We are grateful to three anonymous referees, who provide plenty of useful suggestions which improve the presentation of the paper significantly.} 
This paper is based upon work supported by the National Science Foundation
under Grant No. DMS-1928930 while the author participated in a program hosted
by the Mathematical Sciences Research Institute in Berkeley, California, during
the Fall 2021 semester.
\section{\red{Some Preliminaries}}
\label{sec:pre}
\subsection{\red{The S6V model with infinitely many particles}}
\label{sec:infinite}
\red{Let us explain how to extend the definition of the S6V model in Definition \ref{def:s6v} to that with an infinite number of particles. Without loss of generality, we assume that $M = N = \infty$ in Definition \ref{def:s6v}. In other words, the S6V model has neither a leftmost particle nor a rightmost particle. The situation where only one of $M, N$ is infinite can be tacked in a similar way.}

\red{We adopt the method in \cite[Section 2]{aggarwal2020limit}, which is also related to Harris's construction of ASEP \cite{harris1978additive}. The idea is to divide the lattice points on $\mathbb{Z}$ into a countable union of (random) intervals such that the update of the particles in these intervals are independent.} 

\red{It suffices to explain how we define the update of the S6V model in Definition \ref{def:s6v} from time $t$ to $t+1$ in the situation of infinitely many particles. Since $M = N = \infty$, we suppose that the location of the particles at time $t$ is given by the vector}
\begin{equation*}
\red{\mathbf{p}_t = (\dots < p_t (-1) < p_t (0) < p_t(1) < \cdots).}
\end{equation*}
\red{We recall the i.i.d. random variables $\{\chi_t (x)\}_{x \in \ZZ, t \in \ZZ_{\geq 0}}$ 
from Definition \ref{def:s6v} and know that $\mathbb{P}(\chi_t (x) = 0) = 1 - b_1 > 0$ for every $t \in \mathbb{Z}_{\geq 0}$ and $x \in \mathbb{Z}$. \red{B}y Borel-Cantelli lemma, with probability one we can find a random sequence of integers $\dots <d_{-1} < d_0 < d_1 < \cdots$  such that $\chi_t (p_t (d_i)) = 0$ for every $i \in \mathbb{Z}$. According to Definition \ref{def:s6v}, this means that the $d_i$-th particle will jump one step to its right no matter what happens to its left. As a consequence, we can update the particles with labels that belong to $[d_i, d_{i+1} - 1]$ independently, conditioned on the $d_i$-th particle would jump and $(d_{i+1} - 1)$-th particle \red{cannot} move across $p_t (d_{i+1})$. This reduces the question of understanding the update of the S6V model on the full line to understanding the update of the S6V model on finite length intervals $[p_t(d_i), p_t (d_{i+1})], i \in \mathbb{Z}$, which has been defined in Definition \ref{def:s6v}.} 
\label{sec:two-class}
\subsection{Two-class S6V model}
We want to introduce a coupling for the S6V models. To do that, we need to first introduce the two-class S6V model.

The multi-class S6V model appears as the spin-$\frac{1}{2}$ case of the stochastic $U_q (\widehat{\mathfrak{sl}}_{n+1})$ vertex model, which was introduced in the work of \cite{bazhanov1985trigonometric, jimbo1986quantum}. The model was also studied by the recent works of \cite{ kuniba2016stochastic, kuan2018algebraic, borodin2018coloured}.

Utilizing the multi-class S6V model, \cite[\red{Section 2.4}]{aggarwal2020limit} introduced a coupling between the S6V models. 
This coupling is called the \emph{higher rank coupling}. 
In this paper, we only need to couple two S6V models. Hence, we only 
need to
introduce the two-class S6V model 
in order to define the coupling.

The vertex configurations of the two-class S6V model are indexed by four numbers $(i_1, j_1; i_2, j_2) \in \{1, 2, \infty\}^4$. The numbers should be interpreted as the classes of lines on the 
\red{bottom, left} and top, right of a vertex, 
see Figure \ref{fig:twocolor}. 
We assign stochastic weights to the vertex configurations. 
\begin{figure}[ht]
\centering
\begin{tabular}{|c | c | c| c | c | c| c|}
\hline
 &
\begin{tikzpicture}[scale = 0.8]
\draw[dashed, thick] (-0.5, 0) -- (0.5, 0);
\draw[thick] (0, -0.5) -- (0, 0.5);
\node at (-0.8, 0) {$j$};
\node at (0.8, 0) {$j$};
\node at (0, -0.8) {$i$};
\node at (0, 0.8) {$i$};
\end{tikzpicture}
& 
\begin{tikzpicture}[scale = 0.8]
\draw[thick, dashed] (-0.5, 0) -- (0, 0) -- (0, 0.5);
\draw[thick] (0, -0.5) -- (0, 0) -- (0.5, 0);
\node at (-0.8, 0) {$j$};
\node at (0.8, 0) {$i$};
\node at (0, -0.8) {$i$};
\node at (0, 0.8) {$j$};
\end{tikzpicture}
&
\begin{tikzpicture}[scale = 0.8]
\draw[thick] (-0.5, 0) -- (0.5, 0);
\draw[dashed, thick] (0, -0.5) -- (0, 0.5);
\node at (-0.8, 0) {$i$};
\node at (0.8, 0) {$i$};
\node at (0, -0.8) {$j$};
\node at (0, 0.8) {$j$};
\end{tikzpicture}
&
\begin{tikzpicture}[scale = 0.8]
\draw[thick] (-0.5, 0) -- (0, 0) -- (0, 0.5);
\draw[dashed, thick] (0, -0.5) -- (0, 0) -- (0.5, 0);
\node at (-0.8, 0) {$i$};
\node at (0.8, 0) {$j$};
\node at (0, -0.8) {$j$};
\node at (0, 0.8) {$i$};
\end{tikzpicture}
&
\begin{tikzpicture}[scale = 0.8]
\draw[thick] (-0.5, 0) -- (0.5, 0);
\draw[thick] (0, -0.5) -- (0, 0.5);
\node at (-0.8, 0) {$i$};
\node at (0.8, 0) {$i$};
\node at (0, -0.8) {$i$};
\node at (0, 0.8) {$i$};
\end{tikzpicture}
&
\begin{tikzpicture}[scale = 0.8]
\draw[thick, dashed] (-0.5, 0) -- (0.5, 0);
\draw[thick, dashed] (0, -0.5) -- (0, 0.5);
\node at (-0.8, 0) {$j$};
\node at (0.8, 0) {$j$};
\node at (0, -0.8) {$j$};
\node at (0, 0.8) {$j$};
\end{tikzpicture}
\\
\hline
weight & $b_1$ & $1 - b_1$ & $b_2$ & $1 - b_2$ & $1$ & $1$
\\
\hline 
\end{tabular}
\caption{When $i < j$, the stochastic weights of each possible configuration is given as above.}
\label{fig:twocolor}
\end{figure}
We set $w(i, i; i, i) = 1, i \in \{1, 2, \infty\}$. In addition, for $i, j \in \{1, 2, \infty\}$ satisfying $i < j$, define
\begin{align*}
w(i, j; i, j) &= b_1,\qquad w(i, j; j, i) = 1- b_1, \\
w(j, i; j, i) &= b_2, \qquad w(j, i; i, j) = 1- b_2.
\end{align*}
For all the other vertex configurations, we set the weight to be zero. 

\red{For the each vertex configuration in Figure \ref{fig:twocolor},} we view the lines from left and bottom as the input lines, and the lines to the right and above as the output lines. 
The vertex is conservative in lines since the number of class $r$ input lines equals the number of class $r$ output lines. 
The weights are also stochastic since it defines a probability measure on the output lines when we fix the input lines.

\red{When we interpret the one-class S6V model as an interacting particle system, we sample the vertex configurations for the vertices on the upper half plane as in Figure \ref{fig:S6Vparticle}, the resulting upright paths in Figure \ref{fig:S6Vparticle} are the trajectories of different particles in the S6V model. 
We can also interpret the two-class S6V model as an interacting particle system 
by sampling the vertex configurations in Figure \ref{fig:twocolor} for the vertices on the upper half plane. 
We visualize the class $1$ lines as red lines, class $2$ lines as blue lines and class $\infty$ lines as empty lines. The outcome of the sampling gives a collection of upright paths. Some of the upright paths are given by red lines and some of the upright path are given by blue lines. They respectively denote the trajectories of the first-class and second-class particles in the two-class S6V model, see Figure \ref{fig:two-class}.}
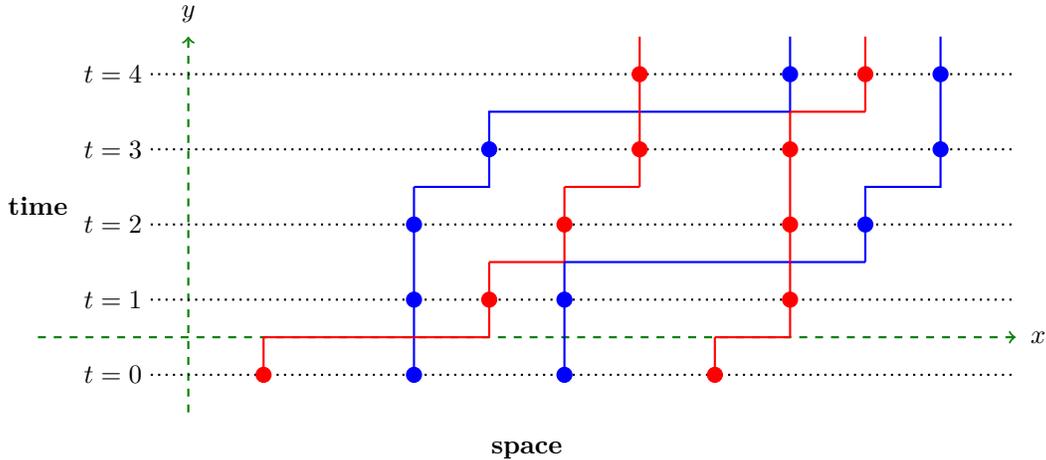
\begin{figure}[ht]
	\centering
	\begin{tikzpicture}
	\draw[thick][aoe][dashed][->] (-1, 0) -- (12, 0);
	\draw[thick][aoe][dashed][->] (1, -1) -- (1, 4);
	\draw[thick, red](2, -0.5) -- (2, 0);
	\draw[thick, blue](4, -0.5) -- (4, 0);
	\draw[thick, blue] (6, -0.5) -- (6, 0.5);
	\draw[thick, red](8, -0.5) -- (8, 0);
	\draw[thick, red] (2, 0) -- (4, 0);
	\draw[thick, blue] (4, 0) -- (4, 0.5);
	\draw[thick, red] (4,0) -- (5,0);
	\draw[thick, red] (5,0) -- (5, 0.5);
	\draw[thick, red] (8, 0) -- (9, 0);
	\draw[thick, red] (9, 0) -- (9, 0.5);
	\node at (1, 4.3) {$y$};
	\node at (12.3, 0) {$x$};
	\draw[thick, blue] (4, 0.5) -- (4, 1.5);
	\draw[thick, red] (5, 0.5) -- (5,1);
	\draw[thick, red] (5, 1) -- (6, 1);
	\draw[thick, red] (6, 1) -- (6, 1.5);
	\draw[thick, blue] (6, 0.5) -- (6, 1);
	\draw[thick, blue](6, 1) -- (9, 1);
	\draw[thick, red] (9, 0.5) -- (9, 1);
	\draw[thick, blue] (9, 1) -- (10,1);
	\draw[thick, red] (9, 1) -- (9, 1.5);
	\draw[thick, blue](10, 1) -- (10, 1.5);
	\draw[thick, blue] (4, 1.5) -- (4, 2);
	\draw[thick, blue](4, 2) -- (5, 2) -- (5, 3) -- (7, 3) -- (9, 3) -- (9, 4);
	\draw[thick, red] (6,1.5) -- (6, 2);
	\draw[thick, red] (6, 2) -- (7,2);
	\draw[thick, red](7, 2) -- (7, 2.5);
	\draw[thick, red] (9, 1.5) -- (9, 2.5);
	\draw[thick, red] (7, 2.5) -- (7, 4);

	\draw[thick, red] (9, 2.5) -- (9, 3);
	\draw[thick, red] (9, 3) -- (10, 3);
	\draw[thick, red] (10, 3) -- (10, 4);
	\draw[thick, blue] (10, 1.5) -- (10, 2) -- (11, 2) -- (11, 4);
	\draw[dotted, thick] (0.5, -0.5) -- (12, -0.5);
	\draw[dotted, thick] (0.5, 0.5) -- (12, 0.5);
	\draw[dotted, thick] (0.5, 1.5) -- (12, 1.5);
	\draw[dotted, thick] (0.5, 2.5) -- (12, 2.5);
	\draw[dotted, thick] (0.5, 3.5) -- (12, 3.5);
	\node at (0, -0.5) {$t = 0$};
	\node at (0, 0.5) {$t = 1$};
	\node at (0, 1.5) {$t = 2$};
	\node at (0, 2.5) {$t = 3$};
	\node at (0, 3.5) {$t = 4$};
	\node at (5.5, -1.5) {\textbf{space}};
	\node at (-1, 1.75) {\textbf{time}};
	\draw[fill][red] (2, -0.5) circle (0.1); 
	\draw[fill][blue] (4, -0.5) circle (0.1); 
	\draw[fill][blue] (6, -0.5) circle (0.1); 
	\draw[fill][red] (8, -0.5) circle (0.1); 
	\draw[fill][blue] (4, 0.5) circle (0.1); 
	\draw[fill][red] (5, 0.5) circle (0.1); 
	\draw[fill][blue] (6, 0.5) circle (0.1); 
	\draw[fill][red] (9, 0.5) circle (0.1); 
	\draw[fill][blue] (4, 1.5) circle (0.1); 
	\draw[fill][red] (6, 1.5) circle (0.1); 
	\draw[fill][red] (9, 1.5) circle (0.1); 
	\draw[fill][blue] (10, 1.5) circle (0.1); 
	\draw[fill][blue] (5, 2.5) circle (0.1); 
	\draw[fill][red] (7, 2.5) circle (0.1); 
	\draw[fill][red] (9, 2.5) circle (0.1); 
	\draw[fill][blue] (11, 2.5) circle (0.1); 
	\draw[fill][red] (7, 3.5) circle (0.1); 
	\draw[fill][blue] (9, 3.5) circle (0.1); 
	\draw[fill][red] (10, 3.5) circle (0.1); \draw[fill][blue] (11, 3.5) circle (0.1);
	\end{tikzpicture}
	\caption{The two-class S6V model viewed as an interacting particle system. The red and blue particles in the picture represent respectively the first and second-class particles. The trajectories of the first and second-class particles are respectively given by red upright paths and blue upright paths.}
	\label{fig:two-class}
\end{figure}

We proceed to provide \red{a concrete} definition for \red{this} interacting particle system. Fix finite integers $M_r, N_r$ such that $M_r \leq 0 \leq N_r$ for $r = 1, 2$. 
Let $\mathbf{p}^r_t := (p^r_t (-M_r) < \dots < p^r_{t} (N_r))$ denote the location 
of first-class (resp. second-class) particles at time $t$ for $r = 1$ (resp. $r = 2$). Note that each lattice \red{point} on $\mathbb{Z}$ allows at most one particle, i.e. we require $\mathbf{p}^1_t \cap \mathbf{p}^2_t = \emptyset$ for any $t \in \mathbb{Z}_{\geq 0}$. \red{The following definition originally appears in \cite[Section 2.3]{aggarwal2020limit}.}
\begin{defin}[Two-class S6V model with \red{finitely many} particles]
\label{def:2classs6v}
We define the two-class S6V model $\red{\mathfrak{p}}_t = (\mathbf{p}^1_t, \mathbf{p}^2_t)$ as a discrete time-homogeneous Markov process. We sample independent random variables $\{\chi^i_t (x)\}_{t \in \mathbb{Z}_{\geq 0}, x \in \mathbb{Z}, i = 1, 2}$, $\{j^i_t (x)\}_{t \in \mathbb{Z}_{\geq 0}, x \in \mathbb{Z}, i = 1, 2}$ such that for every $t \in \mathbb{Z}_{\geq 0}$, $x \in \mathbb{Z}$ and $i \in \{1, 2\}$, 
\begin{align*}
&\mathbb{P}(\chi_t^i (x)  = 1) = b_1, \qquad \mathbb{P}(\chi_t^i (x) = 0) = 1 - b_1, \\
&\mathbb{P}(j_t^i (x) = n) = (1-b_2) b_2^{n-1}, \qquad \forall\, n \in \mathbb{Z}_{\geq 1}. 
\end{align*}
Fix arbitrary $t \in \mathbb{Z}_{\geq 0}$ and set $p_{t+1}^{r} (-M_r -1 ): = -\infty$ and $p_{t+1}^r (N_r+1) = +\infty$ for $r = 1, 2$.
We define the update rule from $\red{\mathfrak{p}}_t$ to $\red{\mathfrak{p}}_{t+1}$ as follow: 
\begin{enumerate}[leftmargin = 2em]
\item  
We first define the update rule from $\mathbf{p}^1_t$ to $\mathbf{p}^1_{t+1}$. 
We define $p^1_{t+1}(i), i = -M_1, \dots, N_1$ sequentially. 
Suppose that we have defined 
$p^1_{t+1} (i)$ for fixed $i \in \{-M_1 - 1, \dots, N_1 \red{- 1}\}$, we proceed to define $p^1_{t+1} (i+1)$.
\begin{enumerate}[leftmargin = 2em, label = (\roman*)]
	\item 
	If $p^1_{t+1} (i) < p^1_{t} (i+1)$ and $\chi^{\red{1}}_t (p^1_{t} (i+1)) = 1$, we set $p_{t+1}^1 (i+1) := p^1_{t} (i+1)$. 
	\item If $p^{\red{1}}_{t+1} (i) = p^{\red{1}}_{t} (i+1)$ or $\chi^{\red{1}}_t (p^1_{t} (i+1)) = 0$, we set $p^1_{t+1} (i+1) :=  \min(p^1_{t} (i+1) + j^{\red{1}}_t (p^1_{t} (i+1)), p^1_{t} (i+2))$.
\end{enumerate}
In other words, \red{t}he first-class particles update in the same way as the particles in the S6V model defined in Definition \ref{def:s6v} and they ignore the second-class particles. 
\item We proceed to define the update from $\mathbf{p}^2_t$ to $\mathbf{p}^2_{t+1}$. Suppose that we have defined $p^2_{t+1} (i)$ for fixed $i \in \{-M_2-1, \dots, N_2 - 1\}$, we proceed to define $p^2_{t+1}(i+1)$. 
\begin{enumerate}[leftmargin = 2em, label = (\roman*)]
\item If either (a) there exists $k \in \{-M_1, \dots, N_1 - 1\}$ such that $p^1_t (k) < p^2_t (\red{i+1}) < p^1_{\red{t+1}} (\red{k})$ or (b) $p_{t+1}^2(i) < p_t^2(i+1)$, \red{$p_t^2 (i+1) \notin \{p_{t+1}^1 (k)\}_{k \in [-M_1, N_1]}$} and $\chi_t^{2} (p_t^2 (i\red{+1})) = 1$ holds, then the second-class particle at $p_t^2(i+1)$ stays and we define $p^2_{t+1} (i+1) := p^2_t (i+1)$. Else it moves. \red{In other words, the second-class particle at $p_t^2 (i+1)$ will stay iff either (a) a first-class particle has jumped across the location $p_t^2 (i+1)$ or (b) $\chi_t^2 (p_t^2(i+1)) = 1$ and no first/second-class particles that have completed their update stay at $p_t^2(i+1)$.} 
\item When the second-class particle at $p_t^2(i+1)$ moves, 
we define   
$$U:= \inf \Big\{p^1_{t} (k): \red{k \in \{-M_1, \dots, N_1\} \text{ which satisifes }} p^1_t (k) > p^2_t(i+1) \text{ and } p^1_t (k) < p^1_{t+1} (k)\Big\}.$$ 
Stated alternatively, $U$ is the leftmost position of a first-class particle that is on the right of $p_t^2(i+1\red{)}$ and that moves \red{during the update from time $t$ to $t+1$.} 

We let $j:= j_t^2(p_t^2(i+1))$ and define $V$ to be the minimal integer such that there exists $j$ integers $x_1, \dots, x_j$ satisfying (a) $p_t^2(i+1) < x_1 < \dots < x_j = V$ and (b) there does not exist $i \in \{-M_{\red{1}}, \dots, N_{\red{1}}\}$ such that $p_t^1 \red{(i)} = p_{t+1}^{1} \red{(i)} \in \{x_1, \dots, x_j\}$. In other words, $V$ is the location of the second-class particle at $p_t^2(i+1)$ after jumping $j$ spaces on its right, skipping the first-class particles that decided to stay.  

Finally, we set $p_{t+1}^2(i+1) := \min(U, V, p_t^2(i+2))$.
\end{enumerate} 
\end{enumerate}
\end{defin}
Following an argument which is similar to  Section \ref{sec:infinite}, the definition above extends to the situation when there are infinitely many particles.

The particle configuration of the two-class S6V model can also be described using the two-class S6V model occupation process $\widetilde{\eta}_t = \{\widetilde{\eta}_t (x): x \in \mathbb{Z}\}$ such that for all $x \in \mathbb{Z}$, we have 
\begin{equation*}
\widetilde{\eta}_t (x) := \sum_{r = 1}^2 r \mathbbm{1}_{\{\exists i \text{ s.t. } p^r_t (i) = x\}}. 
\end{equation*}
In other words, we set $\widetilde{\eta}_t (x) = 1$ (resp. 2) if at time $t$ there is a first (resp. second)-class particle staying at $x$. We set $\widetilde{\eta}_t (x) = 0$ if at time $t$ there is no particle at $x$. \red{Note that the Markov processes $\widetilde{\eta}_t$ and $\mathfrak{p}_t$ are two equivalent ways of characterizing the two-class S6V model.}

The following lemma is immediate once we observe that the weight of vertex configuration for the two-class S6V model in Figure \ref{fig:twocolor} reduces to the (one-class) the S6V model  in Figure \ref{fig:vertexconfig} if we identify both class $1$ and $2$ lines as bold lines and class $\infty$ lines as empty lines. 
\begin{lemma}\label{lem:unclass}
If we unclass the particles in the two-class S6V model, then the two-class S6V model reduces to the (one-class) S6V model. \red{More precisely, we define $\mathbf{p}_t := (p_t (-M) < \dots < p_t(N))$ such that for all $t \in \mathbb{Z}_{\geq 0}$, the set $\{p_t(i): i = -M, \dots N\}$ is equal to  $\cup_{r = 1}^2 \{p_t^{r}(i): i = -M_r, \dots, N_r\}$. Then we have that $(\mathbf{p}_t)_{t \in \mathbb{Z}_{\geq 0}}$ is exactly the S6V model defined in Definition \ref{def:s6v}}.
\end{lemma}
\subsection{Higher rank coupling}
We couple two S6V models $(\eta_t)_{t \in \mathbb{Z}_{\geq 0}}$ and $(\xi_t)_{t \in \mathbb{Z}_{\geq 0}}$, using the two-class S6V model introduced above. Let us recall the higher rank coupling in \cite[Section 2.4]{aggarwal2020limit}. 
\begin{defin}[Higher rank coupling]\label{def:coupling} 
We define the coupled S6V model $(\eta_t, \xi_t)_{t \in \ZZ_{\geq 0}}$, which is a discrete time-homogeneous Markov process on $\mathcal{S} \times \mathcal{S}.$ 
It suffices to state the update rule from $(\eta_t, \xi_t)$ to $(\eta_{t+1}, \xi_{t+1})$ for every fixed $t \in  \mathbb{Z}_{\geq 0}$. 
We define two auxiliary particle configurations $\widetilde{\eta}_t$ and $\widetilde{\xi}_t$ as 
\begin{equation}\label{eq:hkcouple}
\widetilde{\eta}_t (x) := 
\begin{cases}
1, & \text{ if } \eta_t (x) = \xi_t (x) = 1, \\
2, & \text{ if } \eta_t (x) = 1, \xi_t (x) = 0,
\\
0 &\text{ else},
\end{cases}
\qquad \widetilde{\xi}_t(x) := 
\begin{cases}
1, & \text{ if } \xi_t (x) = \eta_t (x) = 1, \\
2, & \text{ if } \xi_t (x) = 1, \eta_t (x) = 0,
\\
0 &\text{ else}.
\end{cases}
\end{equation}
In other words, $\widetilde{\eta}_t$ has a first-class particle at $x$ iff this location is occupied in both $\eta_t$ and $\xi_t$. $\widetilde{\eta}_t$ has a second-class particle \red{at $x$} iff $x$ is occupied in $\eta_t$ but not in $\xi_t$. The interpretation for $\widetilde{\xi}_t$ is similar.  

We update the two-class S6V models from $(\weta_t, \wxi_t)$ to $(\weta_{t+1}, \wxi_{t+1})$. 
The update for both $\weta_t \to \weta_{t+1}$ and $\wxi_t \to \wxi_{t+1}$ follows Definition \ref{def:2classs6v} and uses the \textbf{same} family of random variables $\{\chi^i_t (x), j^i_t (x)\}_{x \in \mathbb{Z}, i = 1, 2}$. 

\red{Then we redefine $\widetilde{\eta}_{t+1}$ and $\widetilde{\xi}_{t+1}$ such that for all $x \in \ZZ$ satisfying $\widetilde{\eta}_{t+1} (x) = \widetilde{\xi}_{t+1} (x) = 2$, we set $\widetilde{\eta}_{t+1} (x) = \widetilde{\xi}_{t+1} (x) = 1$. We keep the value of $\widetilde{\eta}_{t+1} (x), \widetilde{\xi}_{t+1} (x)$ unchanged for all other $x \in \mathbb{Z}$.} We finally set 
\begin{equation}\label{eq:temp4}
\eta_{t+1} (x) := 
\begin{cases}
1, &\qquad \text{ if } \widetilde{\eta}_{t+1} (x) = 1 \text{ or } 2,\\
0 &\qquad \text{ if } \widetilde{\eta}_{t=1} (x) = 0.
\end{cases}
\qquad
\xi_{t+1} (x) := 
\begin{cases}
1, &\qquad \text{ if } \widetilde{\xi}_{t+1} (x) = 1 \text{ or } 2,\\
0 &\qquad \text{ if } \widetilde{\xi}_{t=1} (x) = 0.
\end{cases}
\end{equation}
The above procedure defines the update from $(\eta_t, \xi_t)$ to $(\eta_{t+1}, \xi_{t+1})$.   
\end{defin}
\begin{remark}\label{rmk:seccoupled}
\red{In Definition \ref{def:coupling}, the redefinition of $\widetilde{\eta}_{t+1}$ and $\widetilde{\xi}_{t+1}$ is related to procedure that if two second-class particles respectively in $\widetilde{\eta}_t$ and $\widetilde{\xi}_t$ arrive at the same location at time $t+1$, then they both become the first-class particles. The redefinition of $\widetilde{\eta}_{t+1}$ and $\widetilde{\xi}_{t+1}$  will not change the value of $\eta_{t+1}$ and $\xi_{t+1}$ in \eqref{eq:temp4}. However, it is important since it ensures that $(\widetilde{\eta}_{t+1}, \widetilde{\xi}_{t+1})$ and $(\eta_{t+1},\xi_{t+1})$ satisfy the relation in \eqref{eq:hkcouple}.}
\end{remark}
\begin{remark}
By Lemma \ref{lem:unclass}, we know that the marginals of the coupled S6V model are indeed the S6V models. 
\end{remark}


As observed in \cite[Section 2.4]{aggarwal2020limit}, the coupling in Definition \ref{def:coupling} is \emph{attractive}.
\begin{lemma}[Attractivity]\label{lem:mono}
Under higher rank coupling, if $\eta_t \geq \xi_t$ for some $t  \in \mathbb{Z}_{\geq 0}$, then we have $\eta_T \geq \xi_T$ for any $T \geq t$. If $\eta_t > \xi_t$, then we have $\eta_T > \xi_T$ for any $T \geq t$. 
\end{lemma}
\begin{proof}
\red{
It suffices to show that $\eta_t \geq \xi_t$ implies $\eta_{t+1} \geq \xi_{t+1}$, and $\eta_{t} > \xi_{t}$ implies $\eta_{t+1} > \xi_{t+1}$. 
\bigskip
\\
Let us first show the former part. By Definition \ref{def:coupling}, we know that  $\eta_t \geq \xi_t$ implies $\widetilde{\xi}_t = \xi_t$. This yields that $\widetilde{\xi}_t$ does not have a second-class particle, i.e. $\widetilde{\xi}_t \in \{0, 1\}^{\mathbb{Z}}$. As a consequence, $\widetilde{\xi}_{t+1} \in \{0, 1\}^{\mathbb{Z}}$, which implies that $\widetilde{\xi}_{t+1} = \xi_{t+1}$. In addition, we know that the sets of first-class particles of $\widetilde{\eta}_t$ and $\widetilde{\xi}_t$ are the same, i.e.
$\{x: \widetilde{\eta}_t (x) = 1\} = \{x: \widetilde{\xi}_t (x) = 1\}$. Since the first-class particles in $\widetilde{\eta}_t$ and $\widetilde{\xi}_t$ have the same dynamics, we know that $\{x: \widetilde{\eta}_{t+1} (x) = 1\} = \{x: \widetilde{\xi}_{t+1} (x) = 1\}$. Combining this with $\widetilde{\xi}_{t+1} = \xi_{t+1}$, we know that $\{x: \xi_{t+1}(x) = 1\} = \{x: \widetilde{\xi}_{t+1} (x) =  1\} = \{x: \widetilde{\eta}_{t+1} (x) = 1\} \subseteq \{x: \eta_{t+1} (x) = 1\}$. This implies that $\eta_{t+1} \geq \xi_{t+1}$. 
\bigskip
\\
We proceed to show the latter part. Since $\eta_{t} > \xi_t$, we know that $\widetilde{\eta}_t$ has at least a second-class particle, while $\widetilde{\xi}_t$ does not have a second class particle. As a consequence, $\widetilde{\eta}_{t+1}$ has at least a second-class particle, i.e. the set $\{x: \widetilde{\eta}_{t+1} (x) = 2\}$ is  non-empty. This implies the existence of $y \in \mathbb{Z}$ with $\widetilde{\eta}_{t+1} (y)  = 2$. By Definition \ref{def:coupling}, this yields that $\eta_{t+1} (y) = 1$ and $\xi_{t+1} (y) = 0$. By the previous paragraph, 
we already know that $\eta_{t} > \xi_t$ implies that $\eta_{t+1} \geq \xi_{t+1}$. The existence of aforementioned $y$ concludes that $\eta_{t+1} > \xi_{t+1}$. 
}
\end{proof}

\section{Translation invariant stationary distributions}
\label{sec:translationinvariant}

In this section, we show that the \red{extremal} translation invariant stationary distributions for the S6V model are given by the 
product Bernoulli measures. 

Let $\mu$ be a probability measure for a measurable space $(\Omega, \mathcal{F})$ and let $h: (\Omega, \mathcal{F}) \to (Y, \mathcal{G})$ be a measurable function. As a convention, we use $\mu \circ h^{-1}$ to denote the probability measure on $(Y, \mathcal{G})$ such that $(\mu \circ h^{-1}) (A) := \mu\big(h^{-1} (A)\big)$ for every measurable set $A \in \mathcal{G}$. 

\red{Recall the definition of $\tau$ from the paragraph before Definition \ref{def:shifted},} we define the set of translation invariant probability measures to be 
$$\mathcal{P} := \Big\{\nu \in \MM_1 (\mathcal{S}): \nu = \nu \circ \tau^{-1}\Big\}.$$
\noindent 
 
Let $\IIt \subseteq \MM_1 (\mathcal{S} \times \mathcal{S})$ denote the set of stationary distributions for the coupled S6V model $(\eta_{t}, \xi_{t})_{t \in \mathbb{Z}_{\geq 0}}$. 
Define a map  $\widetilde{\tau}: \SS \times \SS \to \SS \times \SS$ such that $\widetilde{\tau}(\eta, \xi) = (\tau(\eta), \tau(\xi))$. We set
\begin{equation*}
\widetilde{\mathcal{P}} := \Big\{\lambda \in \MM_1 (\mathcal{S} \times \mathcal{S}): \lambda = \lambda \circ \wtil{\tau}^{-1}\Big\}.
\end{equation*}

The following proposition has been proved in \cite[Corollary 3.9]{aggarwal2020limit}. We provide a self-contained argument here. The idea used in the proof is similar to \cite{aggarwal2020limit}.
\begin{prop}
	\label{prop:coramol}
	Let $\lambda \in \red{\IIt  \cap \widetilde{\mathcal{P}}}$. We have 
	\begin{equation*}
	\lambda\Big(\{(\eta, \xi): \eta > \xi\} \cup \{(\eta, \xi): \eta = \xi\} \cup \{(\eta, \xi): \eta < \xi\} \Big) = 1.
	\end{equation*}
\end{prop}
\red{We make some preparation before proving Proposition \ref{prop:coramol}.} 

\red{For the coupled S6V model $(\eta_t, \xi_t)_{t \in \mathbb{Z}_{\geq 0}}$, if $\eta_t(x) \neq \xi_t (x)$ for some $x \in \mathbb{Z}$ and $t \in \mathbb{Z}_{\geq 0}$, we say that there is a \emph{discrepancy} at location $x$ at time $t$. An important observation is that the discrepancies travel as second-class particles in the  the auxiliary processes $(\weta_t)_{t \in \mathbb{Z}_{\geq 0}}$ and $(\wxi_t)_{t \in \mathbb{Z}_{\geq 0}}$ defined in \eqref{eq:hkcouple}. More precisely, if $\eta_t (x) \neq \xi_t (x)$, $(\eta_t (x), \xi_t (x))$ is equal to either $(1, 0)$ or $(0, 1)$.
Note that $(\eta_t (x), \xi_t (x)) = (1, 0)$ iff $\weta_t$ has a second-class particle at $x$. In this situation, the second-class particle is called \emph{$\eta$-type}. Similarly, $(\eta_t (x), \xi_t (x)) = (0, 1)$ iff $\wxi_t$ has a second-class particle at $x$. In this situation, the second-class particle is called \emph{$\xi$-type}. 
It is clear that $(\widetilde{\eta}_t)_{t \in \mathbb{Z}_{\geq 0}}$ cannot have a $\xi$-type second-class particle and $(\widetilde{\xi}_t)_{t \in \mathbb{Z}_{\geq 0}}$ cannot have an $\eta$-type second-class particle. By Definition \ref{def:coupling} and Remark \ref{rmk:seccoupled}, when an $\eta$-type second-class particle and a $\xi$-type second-class particle reach the same location at some time $T$, they both become first-class particles and will travel together afterwards. In this situation, we say that these two second-class particles (of different types) become \emph{coupled}.} 
 
\red{
Fix two arbitrary integers $x < y$, we set}
\begin{equation}\label{eq:Axy}
\red{A_{x, y} := \Big\{(\eta_{\red{0}}, \xi_{\red{0}}): \eta_{\red{0}} (x) = \xi_{\red{0}}(y) = 1,  \xi_{\red{0}} (x) =  \eta_{\red{0}}(y) = 0\Big\} \cup \Big\{(\eta_{\red{0}}, \xi_{\red{0}}):  \eta_{\red{0}} (x) = \xi_{\red{0}} (y) = 0,  \xi_{\red{0}} (x) =  \eta_{\red{0}} (y) = 1\Big\}.}
\end{equation}
\red{Note that for $(\eta_0, \xi_0) \in A_{x, y}$, $(\widetilde{\eta}_0, \widetilde{\xi}_0)$ has second-class particles of different types located at $x$ and $y$. We define the set} 
\begin{equation*}
\red{B_{x, y} := \Big\{\text{an $\eta$-type and a $\xi$-type second-class particle between $[x, y]$ become coupled at time $1$}\Big\}}
\end{equation*}
\red{and }
\begin{equation}\label{eq:axy}
\red{a_{x, y} := \inf_{(\eta_0, \xi_0) \in A_{x, y}} \mathbb{P}^{(\eta_0, \xi_0)}(B_{x, y}).}
\end{equation}
The following lemma is analogous to \cite[Lemma 3.1]{bramson2002characterization}. 
\begin{lemma}\label{lem:axy}
Fix two arbitrary integers $x < y$, we have $a_{x, y} > 0$.
\end{lemma}
\begin{proof}
We prove the lemma by proving a uniform lower bound for $\mathbb{P}^{\eta_0, \xi_0} (A_{x, y} > 0)$ for all $(\eta_0, \xi_0) \in A_{x, y}$. Without loss of generality, we assume that $(\eta_0, \xi_0) \in A'_{x, y}$ with
\begin{equation*}
\red{A'_{x,y} = \Big\{(\eta_0, \xi_0 ) :\eta_0 (x) = \xi_0 (y) = 1, \xi_0 (x) = \eta_0 (y) = 0\Big\}.}
\end{equation*}
Under the higher rank coupling \red{defined in Definition \ref{def:coupling}}, this event says that there is \red{an} $\eta$-type second-class particle at $x$ and a $\xi$-type second-class particle at $y$. 	

Since $(\eta_0, \xi_0) \in A_{x, y}'$, we can define 
\begin{equation*}
\ell := \max\Big\{z < y: \eta_0 (z) = 1, \xi_0 (z) = 0\Big\}, \qquad  r := \min\Big\{z > \ell: \eta_0(z) = 0, \xi_0(z) = 1\Big\}.
\end{equation*}
We know that when $t = 0$, at location $\ell$, there is \red{an} $\eta$-type second-class particle. At location $r$, there is $\xi$-type second-class particle. We label their locations at time $t$ by $X(t)$ and $Y(t)$ respectively. \red{We have} $X(0) = \ell$ and $Y(0) = r$.
\red{Note that there are no second-class particles lying in the open interval $(\ell, r)$, since $(\eta_0(k), \xi_0 (k)) \notin \{(1, 0), (0, 1)\}$ for all $k \in 
(\ell, r)$}.

We consider the evolution of the coupled two-class S6V model $(\widetilde{\eta}_t, \widetilde{\xi}_t)$ from time $0$ to time $1$. The dynamics of the model was defined in Definition \ref{def:coupling}. Let $Q_{\ell, r}^{(i)}$, \red{$i = 1, 2, 3$ denote the events}  
\begin{align*}
&Q_{\ell, r}^{(1)} := \Big\{\text{The particles on the left of $\ell$ which belong to either $\widetilde{\eta}_0$ or $\widetilde{\xi}_0$ \red{do not} \red{move across} $\ell$ \red{during the update}}\Big\}\red{,}\\
&Q_{\ell, r}^{(2)} := \Big\{\text{All the first-class particles in $(\ell, r)$ stay}\Big\}\red{,}\\
&Q_{\ell, r}^{(3)} := \Big\{X(1) = Y(1) = r\Big\}\red{.}
\end{align*}
\red{Let $Q_{\ell, r} = \cap_{i = 1}^3 Q_{\ell, r}^{(i)}$. The event $Q_{\ell, r}$ is illustrated in Figure \ref{fig:discrepancy}. It is clear that $Q_{\ell, r} \subseteq Q_{\ell, r}^{(3)} \subseteq B_{x, y}$, hence $\mathbb{P}^{(\eta_0, \xi_0)} (B_{x, y}) \geq  \mathbb{P}^{(\eta_0, \xi_0)}(Q_{\ell, r})$.  
Moreover, we have}
\begin{align}
\notag
\mbP^{(\eta_0, \xi_0)}(\Q) &= \mbP^{(\eta_0, \xi_0)} \big(\Q^{(1)}\big) \mbP^{(\eta_0, \xi_0)}\big(\Q^{(2)} \,|\, \Q^{(1)}\big) \mbP^{(\eta_0, \xi_0)} \big(\Q^{(3)} \,|\, \Q^{(1)} \cap \Q^{(2)}\big)\\
\notag
&	\geq (1-b_{\red{2}}) \cdot b_1^n \cdot (1-b_1) b_2^{r-\ell - n-1} (1-b_2) \red{b_1} \qquad 
\\
&= 
\label{eq:lowerbd}
C(b_1, b_2) \big(\min(b_1, b_2)\big)^{y-x},	
\end{align}
where $C(b_1, b_2)$ is a \red{positive} constant that only depends on $b_1$ and $b_2$.
The inequality above follows from lower bounding each of the probabilities in the product. \red{Let us explain the reason with more detail. For the first probability, we have $\mathbb{P}^{(\eta_0, \xi_0)}(Q_{\ell, r}^{(1)}) \geq 1 - b_2$ because given that a particle arrives at $\ell$, the probability that it will stop is at least $1 - b_2$. For the second probability, let the location of the $n$ first-class particles be $x_1, \dots, x_n$. We have 
\begin{equation*}
\mathbb{P}^{(\eta_0, \xi_0)}(Q_{\ell, r}^{(2)} \,|\, Q_{\ell, r}^{(1)}) = b_1^n,
\end{equation*} 
because given that the event $Q_{\ell, r}^{(1)}$ happens, the event $Q_{\ell, r}^{(2)}$ happens iff $\chi_0^1 (x_i) = 1$ (recall its definition from Definition \ref{def:coupling}) for all $i = 1, \dots, n$. 
Finally, given that $Q_{\ell, r}^{(1)} \cap Q_{\ell, r}^{(2)}$ happens, if $\chi_t^2(\ell) = 0, j_t^2(\ell) = r - \ell -n$ and $\chi_t^2(r) = 1$ holds, then $Q_{\ell, r}^{(3)}$ happens. Therefore,}  
\begin{equation*}
\red{\mathbb{P}^{(\eta_0, \xi_0)}(Q_{\ell, r}^{(3)} \,|\, Q_{\ell, r}^{(1)} \cap Q_{\ell, r}^{(2)}) \geq \mathbb{P}^{(\eta_0, \xi_0)} (\chi_t^2(\ell) = 0, j_t^2(\ell) = r - \ell -n, \chi_t^2(r) = 1) = (1-b_1) b_2^{r - \ell -  n- 1} (1-b_2) b_1.}
\end{equation*}
We have explained the reason behind \eqref{eq:lowerbd} and thus we know that \begin{equation*}
\mathbb{P}^{(\eta_0, \xi_0)}(B_{x, y}) \geq  \mathbb{P}^{(\eta_0, \xi_0)}(Q_{\ell, r}) \geq C(b_1, b_2) \big(\min(b_1, b_2)\big)^{y-x}
\end{equation*} 
for every $(\eta_0, \xi_0) \in A_{x, y}'$. This concludes the lemma.
\end{proof}
	\begin{figure}[ht]
	\centering
	\begin{tikzpicture}
	\begin{scope}
	\draw[thick] (-5, 0) -- (5, 0); 
	\foreach \x in {-5, -4, -3, -2, -1, 0, 1, 2, 3, 4}
	{
		\draw (\x+ 0.5, -0.1) -- (\x + 0.5, 0.1);
	}
	\draw[fill, blue] (-4.5, 0) circle (0.1);
	\draw[fill, red] (-2.5, 0) circle (0.1);
	\draw[fill, red] (-1.5, 0) circle (0.1);
	\draw[fill, red] (0.5, 0) circle (0.1);
	\draw[fill, red] (2.5, 0) circle (0.1);
	\node at(-4.5, -0.5) {$\ell$};
	\node at (4.5, -0.5) {$r$};
	\node at (5.5, 0) {$\widetilde{\eta}_0$};
	\draw [thick, decorate,
	decoration = {brace, mirror}] (-2.5,-0.3) --  (2.5,-0.3);
	\node at (0, -0.7) {$n$ \text{ first-class particles stay}};
	\node at (-4, 0.8) {\text{the particle moves to $r$ and no particles on its left move across $\ell$}}; 
	\draw[->] (-5, 0.5)  -- (-4.7, 0.2); 
	\end{scope}
	\begin{scope}[yshift = -2.5cm]
	\draw[thick] (-5, 0) -- (5, 0); 
	\foreach \x in {-5, -4, -3, -2, -1, 0, 1, 2, 3, 4}
	{
		\draw (\x+ 0.5, -0.1) -- (\x + 0.5, 0.1);}
	\draw[fill, blue] (4.5, 0) circle (0.1);
	\draw[fill, red] (-2.5, 0) circle (0.1);
	\draw[fill, red] (-1.5, 0) circle (0.1);
	\draw[fill, red] (0.5, 0) circle (0.1);
	\draw[fill, red] (2.5, 0) circle (0.1);
	\node at(-4.5, -0.5) {$\ell$};
	\node at (4.5, -0.5) {$r$};
	\draw [thick, decorate,
	decoration = {brace, mirror}] (-2.5,-0.3) --  (2.5,-0.3);
	\node at (0, -0.7) {$n$ \text{ first-class particles stay}};
	\node at (5.5, 0) {$\widetilde{\xi}_0$};
	\draw[->] (5, 0.5)  -- (4.7, 0.2);
	\node at (5, 0.7) {\text{the particle stays}};
	\draw[->] (-5, 0.5)  -- (-4.7, 0.2); 
	\node at (-6, 0.8) {\text{no particles move across $\ell$}}; 
	\end{scope}
	\end{tikzpicture}
	\caption{An illustration for the event $Q_{\ell, r} = \cap_{i = 1}^3 Q_{\ell, r}^{(i)}$. The red particles represent the first-class particles in the coupling while the blue particles represent the second-class particles.}
	\label{fig:discrepancy}
\end{figure}
\begin{proof}[Proof of Proposition \ref{prop:coramol}]
It suffices to show that $\lambda(A_{x, y}) = 0$ for arbitrary $x < y$.  We assume that $\lambda(A_{x, y}) > 0$ for some $x < y$ and argue by contradiction. 
We use $N_{x, y}$ to denote the number of pairs of second-class particles (of different types) in the interval $[x, y]$ at time $0$ that become coupled  at time $1$.
By Lemma \ref{lem:axy}, we have $\mathbb{P}^{\lambda}(N_{x, y} \geq 1) \geq \lambda(A_{x, y}) a_{x, y}.$	Therefore, we have $\mathbb{E}^{\red{\lambda}} [N_{x, y}] \geq \mathbb{P}^{\lambda} (N_{x, y} \geq 1) \geq \lambda(A_{x, y}) a_{x, y}.$ 
	Since $\lambda$ is translation invariant, we know that $\mathbb{E}^{\red{\lambda}} [N_{x, y}] = \mathbb{E}^{\red{\lambda}} [N_{x+k, y+k}]$. We take $k = y-x+1$ \red{so that} $x+k > y$.
	Hence, \red{we know that $[x, y+mk] \supseteq \cup_{i = 1}^m [x + ik, y+ik]$} and the preceding union is taken over disjoint intervals. Therefore, we have 
	\begin{equation*}
	\mathbb{E}^{\red{\lambda}}[N_{x, y+mk}] \geq \sum_{\red{i} = 1}^{\red{m}} \mathbb{E}^{\red{\lambda}} [N_{x + \red{i}k, y + \red{i}k}] \red{\ =\ } m \mathbb{E}^{\red{\lambda}} [N_{x,y }] \geq m \lambda(A_{x, y}) a_{x, y}. 
	\end{equation*}
	
\red{By our assumption and Lemma \ref{lem:axy}, we have $\lambda(A_{x, y}) > 0$ and $a_{x, y} > 0$}. \red{W}e \red{can} take a large $m$ such that the right hand side is larger than $\red{2}$. Hence, we have $\mathbb{E}^{\red{\lambda}}[N_{x, y+ mk}] > \red{2}$. 
\red{The previous inequality} implies that the expectation of the number of second-class particles that lie in $[x, y+mk]$ at time $0$ and \red{are} coupled 
at time $1$ is larger than $2$.
\red{Note that the change in the number of second-class particles in an interval is equal to the number of the second-class particles entering the interval subtracts the number of second-class particles that get coupled.} 
Since there are at most two second-class particle entering $[x, y+mk]$ (one $\eta$-type and one $\xi$-type) \red{when we update the coupled S6V model from time $0$ to $1$,} 
the expectation of the number of second-class particles in the $[x, y+mk]$ is strictly decreasing. 
Since the number of second-class particles of $(\wtil{\eta}_t, \wtil{\xi}_t)$ that lie in $[x, y+mk]$ equals exactly $\sum_{i = x}^{\red{y +} mk} |\eta_t(i) - \xi_t(i)|$, we get 
\begin{equation*}
\mathbb{E}^{\lambda} \Big[\sum_{i = x}^{\red{y} + mk} |\eta_0 (i) - \xi_0 (i)|\Big] > \mathbb{E}^{\lambda}\Big[\sum_{i = x}^{\red{y} + mk} |\eta_1 (i) - \xi_1 (i)|\Big].
\end{equation*}
This contradicts 
the stationarity of $\lambda$.
\end{proof}

\begin{defin}[\red{Partial order on $\mathcal{M}_1 (\SS)$}] \label{def:ordering}
\red{For two probability measures $\mu, \nu \in \MM_1 (\SS)$, we say that $\mu \leq \nu$ if there exists a $\lambda \in \MM_1 (\SS \times \SS)$ with marginals $\mu, \nu$ such that $\lambda(\{(\eta, \xi): \eta \leq \xi\}) = 1$. According to Definition 2.1 and Theorem 2.4  of \cite{Lig12} (see page 71-72 of the book), this is equivalent to saying that for any continuous function $f:\mathcal{S} \to \mathbb{R}$ satisfying $f(\eta) \leq f(\xi)$, we have $\int f d\mu \leq \int f d\nu$. Hence, ``$\leq$" is a partial order on $\mathcal{M}_1 (\SS)$. Similarly, we say that $\mu < \nu$ if there exists $\lambda \in \MM_1 (\SS \times \SS)$ with marginals $\mu, \nu$ such that $\lambda(\{(\eta, \xi): \eta < \xi\}) = 1$.}  
\end{defin}

\noindent \red{The following corollary is stated as Theorem 3.6 in \cite{aggarwal2020limit}.}
\begin{cor}\label{cor:agg}
	We have
	$(\mathcal{I} \cap \mathcal{P})_e = 
	\{\mu_\rho\}_{\rho \in [0, 1]}$.
	
\end{cor}
\begin{proof}
By Lemma \ref{lem:prodstat}, we know that $\mu_\rho \in \II$. In addition, $\mu_\rho$ is ergodic for the map $\tau$, we have $\mu \in \PP_e$ by \cite[Theorem 4.4]{einsiedler2013ergodic}. This implies that $\{\mu_\rho\}_{\rho \in [0, 1]} \subseteq (\II \cap \PP)_e$. 
We show that the reverse direction also holds. If $\nu \in (\II \cap \PP)_e$, 
for each $\rho \in [0, 1]$, by \cite[Chapter VIII Proposition 2.14]{Lig12}, there exists $\lambda \in (\IIt \cap \widetilde{\mathcal{P}})$ with marginal $\mu_\rho$ and $\nu$. Using Proposition \ref{prop:coramol}, Lemma \ref{lem:mono} and the \red{extremality} of $\lambda$, one gets that either $\nu > \mu_\rho$ or $\nu = \mu_\rho$ or $\nu < \mu_\rho$. \red{Since the product probability measures $\mu_0 \otimes \nu$ and $\nu \otimes \mu_1$ are both supported on the set $\{(\eta, \xi): \eta \leq \xi\}$, we know that $\mu_0 \leq \nu \leq \mu_1$. Using this and} the continuity of the map $\rho \to \mu_\rho$, we conclude that $\nu = \mu_\rho$ for some $\rho \in [0, 1]$. This implies that $(\mathcal{I} \cap \mathcal{P})_e \subseteq \{\mu_\rho\}_{\rho \in [0, 1]}$ and concludes $(\mathcal{I} \cap \mathcal{P})_e = \{\mu_\rho\}_{\rho \in [0, 1]}$.
\end{proof}

\section{Coupling the initial distribution with its shift}\label{sec:coupling}
In this section, we seek to prove Proposition \ref{prop:coupled}. The idea for our proof goes back to \cite[Proposition 3.2]{bramson2002characterization}. For a Markov chain $(X_t)_{\red{t \in \mathbb{Z}_{\geq 0}}}$ with initial distribution $\nu$ and semigroup $\red{(S_t)_{t \in \mathbb{Z}_{\geq 0}}}$, we use $\nu S_{\red{t}}$ to denote the probability distribution of $X_{\red{t}}$. 
\red{We denote} $\red{(\widetilde{S}_t)_{t \in \mathbb{Z}_{\geq 0}}}$ \red{ to be} the semigroup of the coupled S6V model $(\eta_t, \xi_t)_{t \in \mathbb{Z}_{\geq 0}}$. 

For a general probability measure $\mu \in \MM_1 (\SS)$, we define 
\begin{equation*}
\theta_{\red{\mu}}:= \mu \circ (\Id, \tau^{-1}),
\end{equation*}
where $\Id: \SS \to \SS$ is the identity map. In other words, $\theta_{\red{\mu}} \in \MM_1(\SS \times \SS)$ is the probability distribution of $(\eta_0, \tau(\eta_0))$, given $\eta_0 \sim \mu$.

Let $\nu \in \II$ and consider the coupled S6V model with initial distribution $\theta_\nu$.
 Since $\MM_1 (\SS \times \SS)$ is compact, any sequence of probability measures on $\SS \times \SS$ has a limit point. In particular, \red{for every fixed $\nu \in \MM_1(\SS)$}, as $n \to \infty$, a limit point of the Ces\`{a}ro average $\frac{1}{n} \sum_{i = 0}^{n-1} \theta_{\red{\nu}} \widetilde{S}_i$ exists (in the sense of weak convergence). 
\begin{prop}\label{prop:coupled}
\red{Fix $\nu \in \II$}. Let $\lambda$ be a limit point of 
$\frac{1}{n} \sum_{i = 0}^{n-1} \theta_{\red{\nu}} \widetilde{S}_i$
as $n \to \infty$.  
Then $\lambda \in \IIt$ and
\begin{equation}\label{eq:nodiscrepancy}
\lambda\Big(\{(\eta, \xi): \eta > \xi\} \cup \{(\eta, \xi): \eta = \xi\} \cup \{(\eta, \xi): \eta < \xi\}\Big) = 1.
\end{equation}
\end{prop}
\red{
\begin{cor}\label{cor:shift}
Fix $\nu \in \mathcal{I}_e$. Then we have either  $\nu > \nu \circ \tau^{-1}$, $\nu < \nu \circ \tau^{-1}$ or $\nu = \nu \circ \tau^{-1}$.
\end{cor}
\begin{proof}[Proof of Corollary \ref{cor:shift}]
Since $\nu \in \II$ and the S6V model is translation invariant, we know that $\nu \circ \tau^{-1} \in \II$. This implies that the marginals of $\lambda$ in Proposition \ref{prop:coupled} \red{are} given by $\nu$ and $\nu \circ \tau^{-1}$. \red{S}ince $\nu, \nu \circ \tau^{-1} \in \II_e$, by Proposition \ref{prop:coupled} and Lemma \ref{lem:mono}, we know that either $\nu > \nu \circ \tau^{-1}$, $\nu < \nu \circ \tau^{-1}$ or $\nu = \nu \circ \tau^{-1}$.	
\end{proof}
}

\red{The rest of the section is devoted to the proof of Proposition \ref{prop:coupled}. }

\red{Recall that we say that there is a discrepancy at  $x$ at time $t$ if $\eta_t (x) \neq \xi_t (x)$. We call it \red{an} $\eta$-type discrepancy if $\eta_t (x) > \xi_t (x)$ and a $\xi$-type discrepancy if $\eta_t(x) < \xi_t (x)$. Referring to the coupling in Definition \ref{def:coupling}, the discrepancies travel as second-class particles.  Note that the relation $\eta_t (x) > \xi_t (x)$ corresponds to an $\eta$-type second-class particle in $\widetilde{\eta}_t$ at location $x$ (since $\widetilde{\eta}_t (x) = 2$.). Similarly, the relation $\eta_t (x) < \xi_t (x)$ corresponds to a $\xi$-type second-class particle in $\widetilde{\xi}_t$ at location $x$.
When \red{an} $\eta$-type discrepancy and a $\xi$-type discrepancy arrive at the location at some time $T$, they disappear (since as second-class particles of different types, they become coupled).} 
 
\red{We define}
\begin{equation*}
\red{\phi^\mu (t) := \mathbb{P}^{\theta_\mu} \big(\eta_t (0) \neq \xi_t (0)\big).}
\end{equation*}  
\begin{lemma}\label{lem:phi}
\red{Recall that we denote $\mu_\rho$ to be the product Bernoulli measure with density $\rho$.} We consider the coupled S6V model $(\eta_t, \xi_t)$ that starts from the initial distribution $\theta_{\red{\mu_\rho}} = \mu_{\rho} \circ (\Id, \tau^{-1})$. 
Then $\lim_{t \to \infty} \phi^{\red{\mu_\rho}} (t) = 0$ uniformly for $\rho \in [0, 1]$.
\end{lemma}
\begin{proof}
Since $\mu_\rho$ is ergodic for the shift operator $\tau$, hence $\red{\theta}_{\red{\mu_\rho}}$ is ergodic under $\widetilde{\tau}$. Recall the higher rank coupling defined in Definition \ref{def:coupling}, for all time $t \in 
\mathbb{Z}_{\geq 0}$, $(\eta_t, \xi_t)$ can be expressed as a function of $(\eta_0, \xi_0)$ and $\{\chi^i_s (x), j^i_s (x)\}_{s \in \{0, \dots, t-1\},  i \in \{1, 2\}, x \in \mathbb{Z}}$. Since $\{\chi_s^i(x)\}$ and $\{j^i_s (x)\}$ are two groups of i.i.d. random variables, the probability distribution of $(\eta_t, \xi_t)$ is ergodic for the operator $\widetilde{\tau}$.
By \red{Birkhoff} ergodic theorem \red{\cite[Theorem 6.2.1]{durrett2019probability}}, for any $t \in \mathbb{Z}_{\geq 0}$,
\begin{equation}\label{eq:phi}
\phi^{\red{\mu_\rho}} (t) = \lim_{n \to \infty} \frac{1}{2n+1} \sum_{i = -n}^n \mathbbm{1}_{\{\eta_t (i) \neq \xi_t (i)\}}, \qquad \mathbb{P}^{\red{\theta}_{\red{\mu_\rho}}} \text{ a.e.}.
\end{equation}
Since the location of discrepancies $\{i:\eta_t (i) \neq \xi_t (i)\}$ travel as second-class particles in the two-class S6V model $\widetilde{\eta}_t$ and $\widetilde{\xi}_t$ during the update from time $t$ to $t+1$, there can be at most two discrepancies (one $\eta$-type and one $\xi$-type) entering $[-n, n]$. Hence, we have, for all $n$ and $t \in \mathbb{Z}_{\geq 0}$,
\begin{equation}\label{eq:temp1}
\sum_{i = -n}^{\red{n}} \mathbbm{1}_{\{\eta_{t+1} (i) \neq \xi_{t+1} (i)\}} \leq \sum_{i = -n}^{\red{n}} \mathbbm{1}_{\{\eta_{t} (i) \neq \xi_{t} (i)\}} + 2. 
\end{equation}
Using this together with \eqref{eq:phi}, we know that $\phi^\rho (t)$ is non-increasing in $t$. 

We proceed to conclude the lemma. Let $\lambda_\rho$ be a limit point of the Ces\`{a}ro average $\frac{1}{n} \sum_{i = 0}^{n-1} \red{\theta_{\mu_\rho}} \widetilde{S}_i$ as $n \to \infty$. \red{This means that we can find an increasing sequence of positive integers $\{n_k\}_{k = 1}^\infty$ such that $\lambda_\rho  = \lim_{k \to \infty} \frac{1}{n_k} \sum_{i = 0}^{n_k - 1} \theta_{\red{\mu_\rho}} \widetilde{S}_i$. This implies that $\lambda_\rho \widetilde{S}_1 = \lim_{k \to \infty} \frac{1}{n_k} \sum_{i = 1}^{n_k} \theta_{\red{\mu_\rho}} \widetilde{S}_i$ and hence 
$$\lambda_\rho \widetilde{S}_1 - \lambda_\rho = \lim_{k \to \infty }\frac{1}{n_k} (\theta_{\red{\mu_\rho}} \widetilde{S}_{n_k} - \theta_{\red{\mu_\rho}}) = 0.$$   
This shows that $\lambda_\rho \in \widetilde{\II}$. In addition, the probability distribution $\theta_{\red{\mu_\rho}}$ is translation invariant under $\widetilde{\tau}$ and so is the coupled S6V model. This implies that $\theta_{\red{\mu_\rho}} \widetilde{S}_i$ is translation invariant under $\widetilde{\tau}$ and so is $\lambda_\rho$. Therefore, we have $\lambda_\rho \in \widetilde{\II} \cap \widetilde{\mathcal{P}}$.}  
By Proposition \ref{prop:coramol}, we know that $\lambda_\rho \big(\{(\eta, \xi): \eta < \xi\} \cup \{(\eta, \xi): \eta = \xi\} \cup \{(\eta, \xi): \eta > \xi\}\big) = 1$ and both of the marginal distributions of $\lambda_\rho$ are $\mu_\rho$. Using this 
we conclude that $\lambda_\rho\big(\{(\eta, \xi): \eta = \xi\}\big) = 1$.
As a consequence, $\frac{1}{n} \sum_{i = 0}^{n-1} \phi^{\red{\mu_\rho}} (i)$ converges to $0$ as $n \to \infty$. Since $\phi^{\red{\mu_\rho}} (i)$ is \red{non-increasing} in $i$, we have $\lim_{i \to \infty}\phi^{\red{\mu_\rho}} (i) = 0$. Furthermore, note that $\phi^{\red{\mu_\rho}} (i)$ is continuous in $\rho$ for every fixed $i \geq 0$, thus the uniform convergence $\lim_{t \to \infty} \phi^{\red{\mu_\rho}} (t) = 0$ in $\rho$ follows from Dini's theorem. 
\end{proof}

We use $\#A$ to denote the number of elements in a set $A$. 
\begin{lemma}\label{lem:discrepanciesnb1}
Fix $\nu \in \II$. 
Let $(\eta_t, \xi_t)_{t \in \mathbb{Z}_{\geq 0}}$ be the coupled S6V model with initial distribution $\theta_{\red{\nu}}$. 
For any fixed $M > \red{0}$, we have
\begin{equation}\label{eq:temp}
\lim_{n \to \infty} \frac{1}{n} \mathbb{E}^{\theta_{\red{\nu}}}\Big[\#\{x \in [-M n, Mn] \cap \mathbb{Z}: \eta_{\lfloor \sqrt{n} \rfloor} (x) \neq \xi_{\lfloor \sqrt{n} \rfloor} (x)\}\Big] = 0.
\end{equation}
\red{In other words, the expectation of the  total number of discrepancies that lie in $[-Mn, Mn]$ at time $\lfloor \sqrt{n} \rfloor$ for the coupled S6V model with initial probability distribution $\theta_\nu$ is $o(n)$.  }
\end{lemma}
\begin{proof}
\red{\red{We take $M = 1$ for simplicity and the proof for other value of $M > 0$ is rather similar. We argue by contradiction. If $\eqref{eq:temp}$ does not hold, then that there exists $\e > 0$ and an increasing sequence of integers $\{n_k\}_{k = 1}^\infty$ satisfying $\lim_{k \to \infty} n_k = \infty$ and} 
\begin{equation}\label{eq:assumption}
\red{\liminf_{k \to \infty} \frac{1}{n_k} \mathbb{E}^{\theta_{\red{\nu}}}\Big[\#\{x \in [-n_k, n_k] \cap \mathbb{Z}: \eta_{\lfloor \sqrt{n_k} \rfloor} (x) \neq \xi_{\lfloor \sqrt{n_k} \rfloor} (x)\}\Big] \geq \e.}
\end{equation} 
In the following, we will derive a contradiction. 
} 

\red{L}et $\mu$ be \red{an arbitrary} limit point of $\frac{1}{2n_k +1} \sum_{i = -n_k}^{n_k} \nu \circ \tau^{-i}$ as $k \to \infty$. \red{Without loss of generality, we can assume that $\mu = \lim_{k \to \infty} \frac{1}{2n_k + 1} \sum_{i = -n_k}^{n_k} \nu \circ \tau^{-i}$ (if not, just take one more subsequence, and the proof won't change).} \red{Since $\nu \in \II$, one readily checks that} $\mu \in \II \cap \mathcal{P}$. 
By Choquet's representation \cite[Section 3]{phelps2001lectures},  every element of $\II \cap \PP$ has an integral representations in terms of el\red{e}ments of $(\II \cap \PP)_e = \{\mu_\rho\}_{\rho \in [0, 1]}$. Hence, 
there exists a probability measure $\zeta$ on $[0, 1]$ \red{such that $\mu = \int_0^1 \mu_\rho \zeta(d\rho)$.}
Hence, \red{$\theta_{\mu} = \int_0^1 \theta_{\red{\mu_\rho}} \zeta(d\rho)$} and we have  
\begin{equation*}
\phi^\mu (t) = \int_0^1 \phi^{\red{\mu_\rho}} (t) \zeta(d\rho).
\end{equation*}
\red{By Lemma \ref{lem:phi}, we know that $\lim_{t \to \infty} \phi^{\red{\mu_\rho}} (t) = 0$ uniformly for $\rho \in [0,1 ]$. Using this together with the displayed equation above,} 
we have $\lim_{t \to \infty} \phi^\mu(t) = 0$. 
\red{Since $\mu = \lim_{k \to \infty} \frac{1}{2n_k + 1} \sum_{i = -n_k}^{n_k} \nu \circ \tau^{-i}$, we know that} 
\begin{equation*}
\red{\phi^\mu(t) = \lim_{n \to \infty} \frac{1}{2n_k+1} \sum_{i = -n_k}^{n_k} \mathbb{P}^{\theta_{\nu \circ \tau^{-i}}} \big(\eta_t (0) \neq \xi_t (0)\big) =   
\lim_{k \to \infty} \frac{1}{2n_k+1} \mathbb{E}^{\theta_{\red{\nu}}}\Big[\#\{x \in [-n_k, n_k] \cap \mathbb{Z}: \eta_{t} (x) \neq \xi_{t} (x)\}\Big]. }
\end{equation*}
\red{Since $\lim_{t \to \infty} \phi^\mu (t) = 0$. There exists a $t_0$ such that} 
\begin{equation*}
\red{\lim_{k \to \infty} \frac{1}{n_k} \mathbb{E}^{\theta_{\red{\nu}}}\Big[\#\{x \in [-n_k, n_k] \cap \mathbb{Z}: \eta_{t_0} (x) \neq \xi_{t_0} (x)\}\Big] = \phi^{\mu} (t_0) \leq \frac{\e}{4}.}
\end{equation*} 
\red{Hence, there exists $N$ such that for all $n_k > N$,}
\begin{equation}\label{eq:temp2}
\red{\mathbb{E}^{\theta_{\red{\nu}}}\Big[\#\{x \in [-n_k, n_k] \cap \mathbb{Z}: \eta_{t_0} (x) \neq \xi_{t_0} (x)\}\Big] \leq \frac{\e}{2} n_k.} 
\end{equation}
Note that by repeatedly applying the inequality \eqref{eq:temp1} with $t = t_0, \dots, \lfloor\sqrt{n_k} \rfloor - 1$ and setting $n = n_k$ therein, we know that  
\begin{equation*}
\red{\sum_{i = -n_k}^{n_k} \mathbbm{1}_{\{\eta_{\lfloor \sqrt{n_k} \rfloor} (i) \neq \xi_{\lfloor \sqrt{n_k} \rfloor} (i)\}} \leq \sum_{i = -n_k}^{n_k} \mathbbm{1}_{\{\eta_{t_0} (i) \neq \xi_{t_0} (i)\}} + 2(\lfloor \sqrt{n_k} \rfloor -  t_0).} 
\end{equation*}
We take the expectation of both sides with respect to the initial probability distribution $\theta_\nu$ and then upper bound the first expectation on the right hand side using \eqref{eq:temp2} and finally divide both sides by $n_k$, this yields 
\begin{equation*}
\red{\frac{1}{n_k} \mathbb{E}^{\theta_{\red{\nu}}}\Big[\#\{x \in [-n_k, n_k] \cap \mathbb{Z}: \eta_{\lfloor \sqrt{n_k} \rfloor} (x) \neq \xi_{\lfloor \sqrt{n_k} \rfloor} (x)\}\Big] \leq \frac{\e}{2} + \frac{2( \lfloor \sqrt{n_k} \rfloor - t_0)}{n_k}.}
\end{equation*}
\red{By taking $n_k$ large enough so that the right hand side above is upper bounded by $\frac{3}{4} \e$, we see that this is in contradiction to \eqref{eq:assumption}.}
\end{proof} 
The following lemma is from \cite[Lemma 2.10]{aggarwal2020limit}, which controls the speed of the second-class particles. 
\begin{lemma}\label{lem:secondclass}
For the two-class S6V model and $k \in [-M_2, N_2]$, we have for $r \in \mathbb{Z}_{\geq 1}$,
\begin{equation*}
\mathbb{P}\Big(p_{t+1}^{2}(k) - p_t^{2} (k) \geq r\Big) \leq \max(b_1, b_2)^{r-1}
\end{equation*}	
\end{lemma}
\begin{proof}
Let us assume at time $t$, there are $m$ first-class particles lying in the interval $(p_t^{2} (k), p_t^{2} (k) + r)$.
If $p^2_{t+1}(k) \geq p_t^2(k) + r$ holds, then these first-class particles must stay during the update from time $t$ to $t+1$. The probability is upper bounded by $b_1^m$. Moreover, according to Definition \ref{def:2classs6v}, we require that $j_t^2(p_t^2(k)) \geq r - m$. Since the previous events are independent, we have 
\begin{equation*}
\mathbb{P}\big(p_{t+1}^2 (k) - p_t^2(k) \geq r\big) \leq b_1^m \mathbb{P}\big(j_t^2 (p_t^2(k)) \geq r-m\big) \leq \max(b_1, b_2)^{r-1}.  \qedhere
\end{equation*}
\end{proof}
Fix $n \in \mathbb{Z}_{\geq 1}$ and $y \in \mathbb{Z}$. Denote $\N_y(n)$ to be the number of discrepancies in the coupled S6V model 
that \red{start} from $(-\infty, y]$ \red{at time $0$} and reach the interval $[-n, n]$ \red{at some time between} $[\lfloor \sqrt{n} \rfloor, n]$.
\begin{lemma}\label{lem:discrepanciesnb2}
Fix $M  > 1 + \frac{1}{1 - \max(b_1, b_2)}$ and $\nu \in \II$.  
We have $\mathbb{E}^{\theta_{\red{\nu}}}[\N_{-Mn}(n)] = o(n)$. 
\end{lemma}
\begin{proof}
For the coupled S6V model $(\eta_t, \xi_t)_{t \in \mathbb{Z}_{\geq 0}}$, during its update from time $t$ to $t+1$, there can be at most two discrepancies (one $\eta$-type and one $\xi$-type) moving across $-Mn$ from left to right. Therefore, we have $\red{\mathsf{N}}_{-Mn}(n) \leq 2\red{(n - \lfloor \sqrt{n} \rfloor)} \leq 2n$. It suffices to show that there exists $C = C(M)$ such that 
\begin{equation}\label{eq:expdecay}
\mathbb{P}^{\theta_{\red{\nu}}}\big(\red{\mathsf{N}}_{-Mn}(n) \geq 1\big) \leq C e^{-\frac{n}{C}}.
\end{equation}
To show \eqref{eq:expdecay}, let $Z_t$ be the location of the rightmost \red{$\eta$-type} discrepancy 
in $(\eta_t, \xi_t)$
that originally lies in $(-\infty, -Mn]$ \red{at time 0}. Note that $Z_t$ travels as a second-class particle. By Lemma \ref{lem:secondclass}, We have 
\begin{equation*}
\mathbb{P}^{\theta_{\red{\nu}}} (Z_t - Z_{t-1} \geq r) \leq \max(b_1, b_2)^{r-1}.
\end{equation*}
\red{Similarly, let $Z_t'$ be the location of the rightmost $\xi$-type discrepancy in in $(\eta_t, \xi_t)$ that originally lies in $(-\infty, -Mn]$ at time $0$. We have $\mathbb{P}^{\theta_{\red{\nu}}} (Z'_t - Z'_{t-1} \geq r) \leq \max(b_1, b_2)^{r-1}.$
Hence, we can find i.i.d. geometric random variables \red{$\{X_t\}_{t \in \mathbb{Z}_{\geq 1}}$, $\{X'_t\}_{t \in \mathbb{Z}_{\geq 1}}$} with parameter $1 - \max(b_1, b_2)$ such that $Z_t - Z_{t-1}$ and $Z'_t - Z'_{t-1}$ are respectively stochastically dominated by $X_t$ and $X_t'$ for all $t \in \mathbb{Z}_{\geq 1}$.}  
By a standard large deviations bound, we have 
\begin{align*}
\mathbb{P}^{\theta_{\red{\nu}}}(\N_{-Mn} (n) \geq 1) &\leq 
\mathbb{P}^{\theta_{\red{\nu}}}(Z_n - Z_0 \geq (M-1)n) + \mathbb{P}^{\theta_{\red{\nu}}} (Z'_n - Z'_0 \geq (M-1)n)\\ 
&\leq \mathbb{P}^{\theta_{\red{\nu}}}\Big(\red{\sum_{i = 1}^n X_i \geq (M-1) n}\Big) + \mathbb{P}^{\theta_{\red{\nu}}}\Big(\red{\sum_{i = 1}^n X'_i \geq (M-1) n}\Big) \leq C e^{-\frac{1}{C} n}.
\end{align*}
This concludes \eqref{eq:expdecay}. 
\end{proof}
\begin{remark}
From the proof \red{above,} we see that Lemma \ref{lem:discrepanciesnb2} is \red{still} valid \red{if we replace $\theta_\nu$ with any probability measure} $\theta \in \MM_1(\SS \times \SS)$.
\end{remark}
Using Lemma \ref{lem:discrepanciesnb1} and \ref{lem:discrepanciesnb2}, we establish the following result.
\begin{prop}\label{prop:liggettineq}
\red{Fix $\nu \in \II$.} Consider the coupled S6V model with initial distribution $\theta_{\red{\nu}}$. Let $w_{\sqrt{n}, n}$ be the number of discrepancies that visited $[-n, n]$ \red{at some time between} $[\lfloor \sqrt{n} \rfloor, n]$. We have $\lim_{n \to \infty} \frac{1}{n} \mathbb{E}^{\theta_{\red{\nu}}} [w_{\sqrt{n}, n}] = 0$.
\end{prop}
\begin{proof}
Fix $M > 1 + \frac{1}{1 - \max(b_1, b_2)}$, we have 
\begin{equation*}
w_{\sqrt{n}, n} \leq \red{\mathsf{N}}_{-Mn} (n) + \#\Big\{x \in [-M n, Mn] \cap \mathbb{Z}: \eta_{\lfloor \sqrt{n} \rfloor} (x) \neq \xi_{\lfloor \sqrt{n} \rfloor} (x)\Big\}.
\end{equation*} 
Take the expectations of both sides above with respect to the initial probability distribution $\theta_\nu$, divide them by $n$ and apply Lemma \ref{lem:discrepanciesnb1} and \ref{lem:discrepanciesnb2} to the resulting right hand side, we conclude the proposition.
\end{proof}
Recall the definition of  $A_{x,y}$ from  \eqref{eq:Axy} and $a_{x, y}$ from \eqref{eq:axy}.
We use Proposition \ref{prop:liggettineq} and  Lemma \ref{lem:axy} to show Proposition \ref{prop:coupled}.
\begin{proof}[Proof of Proposition \ref{prop:coupled}]
Since $\lambda$ is a limit point of $\frac{1}{n} \sum_{i = 0}^{n-1} \theta_{\red{\nu}} \widetilde{S}_{\red{i}}$ as $n \to \infty$, it is clear that $\lambda \wtil{S}_{\red{1}} = \lambda$. Hence, we have $\lambda \in \IIt$. We proceed to prove \eqref{eq:nodiscrepancy}.
\red{As mentioned after Proposition \ref{prop:coupled}, the discrepancies can be identified as second-class particles, and two discrepancies of different types would disappear if they arrive at the same location for some time $T$. This is equivalent to saying that two second-class particles of different types become coupled after they arrive at the same location. By \eqref{eq:axy}, given that the configuration $(\eta_i, \xi_i) \in A_{x, y}$, at time $i+1$, with probability at least $a_{x, y}$, a discrepancy would disappear. Note that the discrepancies can not be created, we know that between time $[\lfloor \sqrt{n} \rfloor, n ]$, the expected number of discrepancies that disappear is at least $a_{x, y } \mathbb{E}^{\theta_{\red{\nu}}}[\sum_{i = \lfloor \sqrt{n}\rfloor}^{n-1} \mathbf{1}_{\{(\eta_i, \xi_i) \in A_{x, y}\}}]$. This number should be smaller than the expected number of discrepancies that visited $[-n, n]$  between time $[\lfloor \sqrt{n} \rfloor, n ]$ (we can take $n$ large enough so that $[x, y] \subseteq [-n, n]$). Therefore,}
\begin{equation*}
\mathbb{E}^{\theta_{\red{\nu}}} [w_{\sqrt{n}, n}] \geq
a_{x, y } \mathbb{E}^{\theta_{\red{\nu}}}\Big[\sum_{i = \lfloor \sqrt{n}\rfloor}^{n-1} \mathbf{1}_{\{(\eta_i, \xi_i) \in A_{x, y}\}}\Big]. 
\end{equation*}    
Divide both sides by $n$, take $n \to \infty$ and then apply Proposition \ref{prop:liggettineq} and Lemma \ref{lem:axy}, we know that $$\lim_{n \to 
\infty} \frac{1}{n}\mathbb{E}^{\theta_{\red{\nu}}}\Big[\sum_{i = \lfloor \sqrt{n} \rfloor}^{n-1} \mathbf{1}_{\{(\eta_i, \xi_i) \in A_{x, y}\}}\Big] = 0.$$
\red{This implies that}
\begin{equation*}
\red{\lim_{n \to 
\infty} \frac{1}{n}\mathbb{E}^{\theta_\nu}\Big[\sum_{i = 0}^{n-1} \mathbf{1}_{\{(\eta_i, \xi_i) \in A_{x, y}\}}\Big] = 0.}
\end{equation*} 
Since $\lambda$ is a limit point of the Ces\`{a}ro average 
$\frac{1}{n} \sum_{i = 0}^{n-1} \theta_{\red{\nu}} \widetilde{S}_i$, \red{the quantity on the left hand side above is equal to $\lambda(A_{x, y})$.} We conclude that $\lambda(A_{x, y}) = 0$ for any $x \neq y$. \red{Note that the complement of $\cup_{x < y} A_{x, y}$ is equal to union of sets on the left hand side of \eqref{eq:nodiscrepancy}, we conclude Proposition \ref{prop:coupled}.}
\end{proof}
\section{Current analysis and proof of Theorem \ref{thm:st}}
\label{sec:current}
\subsection{Monotonicity of current}
\red{
We define the \emph{current} function $K_y$ to be the number of particles that move across the location $y$ when we start the S6V model at time $0$ and update it to time $1$. Note that $K_y$ is either zero or one, and it indicates whether the horizontal edge of $(y, 0)$ -- $(y+1, 0)$ is occupied in the tiling interpretation of the S6V model. Mathematically, we define}
\begin{equation}\label{eq:current}
K_{y} := \mathbbm{1}_{\{\exists\, k \in \mathbb{Z}, p_0 (k) \leq y < p_{1} (k)\}},
\end{equation}
where we recall the definition of $p$ from Definition \ref{def:s6v}. 
We use $\mathbb{E}^{\eta} [K_y]$ to denote the expectation of the current of the S6V model with initial data $\eta$, where $\eta(x) = \mathbbm{1}_{\{\exists\, i,  p_0 (i) = x\}}$ for all $x \in \mathbb{Z}$.  
\begin{lemma}\label{lem:currentrec}
\red{Fix arbitrary $\eta \in \mathcal{S}$ and $y \in \mathbb{Z}$}, we have 
\begin{align*}
&\mathbb{E}^\eta [K_{y} ] = \mathbb{E}^\eta [K_{y-1}] + (1 - \mathbb{E}^{\eta}[K_{y-1}]) (1 - b_1), \qquad \text{  if } \eta (y) = 1,\\
&\mathbb{E}^{\eta} [K_{y}]  = b_2 \mathbb{E}^{\eta} [K_{y-1}], \hspace{12.1em} \text{  if } \eta (y) = 0.
\end{align*}
\end{lemma}
\begin{proof}
Given a particle configuration $\eta \in \SS$ and consider the update of it from time $0$ to time $1$. During the update, assume that a particle has moved across $y-1$ and arrived at $y$. If the location $y$ is already occupied by another particle, then the former particle stops and the latter particle jumps to its right with probability $1$. If instead $y$ is not occupied, then with probability $b_2$, the former particle continues to move across $y$. 
Hence, 
we have 
\begin{align}\label{eq:recur1}
\mathbb{P}^{\eta}(
K_y = 1) = b_2^{1 - \eta (y)}  \qquad\quad  \text{ if } 
K_{y-1} = 1.
\end{align}
Under a similar reasoning, we have
\begin{align}\label{eq:recur2}
\mathbb{P}^{\eta}(K_y = 1) = 1 - b_1^{\eta (y)} \quad \,\quad\   \text{if } K_{y-1} = 0. 
\end{align}
Combine \eqref{eq:recur1}-\eqref{eq:recur2}, we conclude the lemma. 
\end{proof}
The following proposition establishes the monotonicity of the \red{expectation of} the current.
\begin{lemma}[Monotonicity]\label{lem:current}
Fix arbitrary $\eta, \xi \in \SS$. 
If $\eta(z) \geq \xi(z)$ for all $z \leq y$, then $\mathbb{E}^{\eta}[K_{y}]  \geq \mathbb{E}^\xi [K_{y}]$. If additionally there exist $x \leq y$ such that $\eta (x) > \xi (x)$, then $\mathbb{E}^{\eta} [K_{y}] > \mathbb{E}^\xi [K_{y}]$. 
\end{lemma}
\begin{proof}
We prove the first part of the lemma. 
\red{We couple the update of S6V models from $\eta_0 = \eta$ to $\eta_1$ and from $\xi_0 = \xi$ to $\xi_1$ using the higher rank coupling stated in Definition \ref{def:coupling}.}
Since $\eta_0 (z) \geq \xi_0 (z)$ for all $z \leq y$, 
$\widetilde{\xi}_0 (z)$ is equal to either $0$ or $1$ for all $z \leq y$. In other words, on $(-\infty, y]$, $\widetilde{\xi}_0$ only contains first-class particles. 
$\weta_0$ contains first-class particles at the same positions and might have a few more second-class particles. By Definition \ref{def:coupling}, the first-class particles in $\weta_0$ and $\wxi_0$ \red{will} travel together. \red{Let $p'_t$ record the location of particles in $\xi_t$ so that $\xi_t (x) = \mathbbm{1}_{\{\exists i, p'_t (i) = x\}}$.
By the previous argument, $\mathbbm{1}_{\{\exists k \in \mathbb{Z}, p'_0 (k) \leq y < p'_{1} (k)\}} = 1$ implies that $\mathbbm{1}_{\{\exists k \in \mathbb{Z}, p_0 (k) \leq y < p_{1} (k)\}} = 1$. Therefore,  when $\eta(z) \geq \xi(z)$ for all $z \leq y$, under higher rank coupling, we have} \begin{equation*}
\red{\mathbbm{1}_{\{\exists k \in \mathbb{Z}, p_0 (k) \leq y < p_{1} (k)\}} \geq \mathbbm{1}_{\{\exists k \in \mathbb{Z}, p'_0 (k) \leq y < p'_{1} (k)\}}.}
\end{equation*}  
\red{Taking the expectation of both sides implies that  $\mathbb{E}^{\eta}[K_y]  \geq \mathbb{E}^\xi [K_y]$.}

We proceed to prove the second part of the lemma. \red{If additionally there exists some $x$ satisfying $x \leq y$ and $\eta (x) > \xi (x)$}, \red{we can define} $m := \max \{z: \eta_0 (z) > \xi_0 (z), z \leq y\}$. By definition, we have $\eta (m) = 1$, $\xi (m) = 0$ and $\eta (z) = \xi (z)$ for all $z \in [m+1, y] \cap \mathbb{Z}$. By Lemma \ref{lem:currentrec}, we have 
\begin{equation}\label{eq:recur}
\begin{split}
&\red{\mathbb{E}^\eta [K_{m}]  = \mathbb{E}^\eta [K_{m-1}] + (1 - \mathbb{E}^\eta [K_{m-1}]) (1 - b_1)},\\
&\red{\mathbb{E}^\xi [K_{m}] = \mathbb{E}^\xi [K_{m-1}]  b_2.}
\end{split}
\end{equation}
Using the first part of the lemma, we have  $\mathbb{E}^\eta [K_{m-1}] \geq \mathbb{E}^\xi [K_{m-1}]$. Using this together with \eqref{eq:recur}, we have 
\begin{align*}
\mathbb{E}^\eta [K_m] &=  \mathbb{E}^{\eta} [K_{m-1}] + (1 - \mathbb{E}^\eta [K_{m-1}]) (1 - b_1) \\
&\geq  \mathbb{E}^\xi [K_{m-1}] + (1 - \mathbb{E}^\xi [K_{m-1}]) (1 - b_1) \\
&> \mathbb{E}^\xi [K_{m-1}] b_2 = \mathbb{E}^\xi [K_m].
\end{align*}
This implies that $\mathbb{E}^\eta [K_m] > \mathbb{E}^\xi [K_m]$. Since $\eta_0 (z) = \xi_0 (z)$ for all $z \in [m+1, y]$, this implies that $\mathbb{E}^\eta [K_z]$ and $\mathbb{E}^\xi [K_z]$ follow the same recursion for $z \in [m+1, y]$ in Lemma \ref{lem:currentrec}. Since the recursions in Lemma \ref{lem:currentrec} preserve the monotonicity, by $\mathbb{E}^\eta [K_m] > \mathbb{E}^\xi [K_m]$,  
we conclude that $\mathbb{E}^\eta [K_{y}] > \mathbb{E}^\xi [K_{y}]$.
\end{proof}
Since the vertex of the S6V model is conservative in lines, we readily have the following lemma. 
\begin{lemma}\label{lem:currentconserve}
We have 
\begin{equation*}
\red{\mathbbm{1}_{\{\exists k \in \mathbb{Z}, p_0 (k) \leq x-1 < p_{1} (k)\}}} + \sum_{k = x}^y \eta_0 (k) = \red{\mathbbm{1}_{\{\exists k \in \mathbb{Z}, p_0 (k) \leq y < p_{1} (k)\}}} + \sum_{k = x}^{y} \eta_1 (k).
\end{equation*}
\end{lemma}
\begin{proof}[Proof of Theorem \ref{thm:st}]
By Lemma \ref{lem:prodstat}, we know that $\{\mu_\rho\}_{\rho \in [0, 1]} \subseteq \II$. To conclude the proof of Theorem \ref{thm:st}, it suffices to show that for all $\nu \in \II_e$, we have $\nu \in \{\mu_{\rho}\}_{\rho \in [0, 1]}$. 
Since $\nu \in \II_e$, by Corollary \ref{cor:shift}, we know that either $\nu > \nu \circ \tau^{-1}$\red{,} $\nu < \nu \circ \tau^{-1}$ or $\nu = \nu \circ \tau^{-1}$. 
We claim that only the last case is  possible.

To prove this claim, we argue by contradiction. Let $\mathbb{E}^\nu[K_y] = \int \mathbb{E}^\eta [K_y] \nu(d\eta)$.
If $\nu < \nu \circ \tau^{-1}$, by Lemma \ref{lem:current}, there exists $y \in \mathbb{Z}$ such that 
\begin{equation}\label{eq:currentineq}
\mathbb{E}^{\nu}[K_{y\red{-}1}] < \mathbb{E}^{\nu \circ \tau^{-1}}[K_{y-1}] = \mathbb{E}^{\nu}[K_{y}]. 
\end{equation}
In addition, since the vertex of the S6V model is conservative in lines, we have 
\begin{equation}\label{eq:conservative}
\eta_0 (y) + K_{y-1}   = \eta_{1} (y) + K_{y}.
\end{equation}
Using \eqref{eq:currentineq} and \eqref{eq:conservative}, we have $\mathbb{E}^{\nu} [\eta (y)] > \mathbb{E}^{\nu} [\eta_{1} (y)]$, which contradicts 
the assumption that $\nu \in \II$. We can also rule out $\nu > \nu \circ\tau^{-1}$ using a similar argument. 
Consequently, we conclude $\nu = \nu \circ \tau^{-1}$, which implies $\nu \in \mathcal{P}$. Hence, we know that $\nu \in \II_e$ implies $\nu \in \II_e \cap \PP \subseteq (\II \cap \PP)_e$. By Corollary \ref{cor:agg}, we have $\nu \in \{\mu_\rho\}_{\rho \in [0, 1]}$. This concludes the theorem.
\end{proof}

\section{The blocking measure $\mu_*$ for the S6V model under a moving frame}
\label{sec:blk} 
In this section, we look at the shifted S6V model $(\eta_t')_{t \in \ZZ_{\geq 0}} = (\tau_t (\eta_t))_{t \in \ZZ_{\geq 0}}$. 
Recall that we use $\mathfrak{I}$ to denote the set of its stationary distribution\red{s}. 
We have proved that $\{\mu_\rho\}_{\rho \in [0, 1]} \subseteq \II \cap \PP$, 
using this together with the fact that the S6V model is translation invariant implies that $\{\mu_\rho\}_{\rho \in [0, 1]} \subseteq \mathfrak{I}$. 
Recall the definition of the blocking measure $\mu_*$ from \eqref{eq:prodinhomo}. The main task of the section is to show that  $\mu_* \in \mathfrak{I}$. 
The key ingredient is the following lemma, which shows a shifted version of the well-known local stationarity of the S6V model (see Lemma \ref{lem:lcs}). 
\begin{lemma}[Local pseudo-stationarity]
	\label{lem:bs}
Consider a vertex configuration with stochastic weights in Figure \ref{fig:vertexconfig}. Recall that $q = \frac{b_1}{b_2}$. Let $v, h, v', h'$ denote the number of lines (either zero or one)  on the bottom, left, top and right of the vertex. 
Fix $\rho, \zeta \in [0, 1]$ such that 
\begin{equation}\label{eq:bs}
\red{q \rho (1-\zeta) = \zeta (1-\rho).}
\end{equation}
Assume that $(v, h) \sim \text{Ber}(\rho) \otimes \text{Ber}(\zeta)$, 
we sample $v'$ and $h'$ according to the probability measure determined by the stochastic weights in Figure \ref{fig:vertexconfig}, then we have $(v', h') \sim \text{Ber}(\zeta) \otimes \text{Ber}(\rho)$. See Figure \ref{fig:blocking} for visualization.
\end{lemma}
\begin{figure}[ht]
\centering
\begin{tikzpicture}
\draw[thick] (-1, 0) -- (1, 0);
\draw[thick] (0, -1) -- (0, 1);
\node at (0, -1.3) {$v$};
\node at (0, 1.3) {$v'$};
\node at (-1.6, 0) {$h$};
\node at (1.6, 0) {$h'$};
\begin{scope}[xshift = 6cm]
\draw[thick] (-1, 0) -- (1, 0);
\draw[thick] (0, -1) -- (0, 1);
\node at (0, -1.3) {$\Ber(\rho)$};
\node at (0, 1.3) {$\Ber(\zeta)$};
\node at (-1.6, 0) {$\Ber(\zeta)$};
\node at (1.6, 0) {$\Ber(\rho)$};
\end{scope}
\end{tikzpicture}
\caption{A visualization of Lemma \ref{lem:bs}. 
}
\label{fig:blocking}
\end{figure}
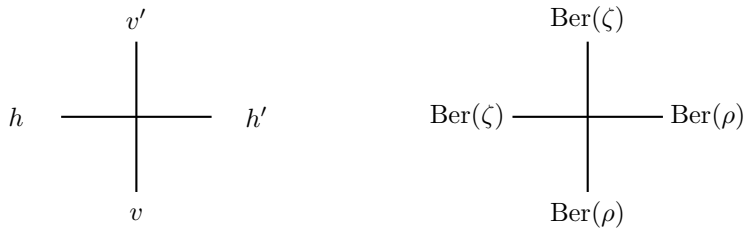
\begin{remark}
This lemma should be compared with Lemma \ref{lem:lcs}, which says if $\rho$ and $\zeta$ instead satisfy the relation in \eqref{eq:lcs}, then $(v', h') \sim \Ber(\rho) \times \Ber(\zeta)$. Lemma \ref{lem:bs} relates to the stationarity of the shifted S6V model while Lemma \ref{lem:lcs} relates to the stationarity of the S6V model without shifting. 
\end{remark}
\begin{proof}[\red{Proof of Lemma \ref{lem:bs}}]
	The proof is straightforward. We can compute 
	\begin{align*}
	\mathbb{P}\big(v' = 1\big) &= \mathbb{P}\big(v = 1, h = 1\big) + \mathbb{P}\big(v = 1, h = 0\big) b_1 + \mathbb{P}\big(v = 0, h = 1\big)(1 - b_2)\\ 
	&= \rho\zeta + \rho(1-  \zeta)b_1 + \zeta(1 -\rho) (1-b_2).
	\end{align*}
	Similarly, we have
	\begin{align*}
	\mathbb{P}\big(h' = 1\big) 
	= \rho\zeta + \zeta(1-\rho) b_2 + \rho(1 - \zeta)(1 - b_1).
	\end{align*}
	Using $q = \frac{b_1}{b_2}$ and 
	$\red{\eqref{eq:bs}}$, we have $\mathbb{P}(v' = 1) = \zeta$ and $\mathbb{P}(h' = 1) = \rho$. Moreover, we have 
	\begin{equation*}
	\mathbb{P}\big(v' = 1, h' = 1\big) = \mathbb{P}\big(v = 1, h = 1\big) = \rho \zeta.
	\end{equation*}
	This concludes the lemma.  
\end{proof}

Note that for arbitrary $t \in \mathbb{Z}_{\geq 0}$, the map $\tau_t: \mathcal{S} \to \mathcal{S}$ is bijective and $\tau_t = \tau_{-t}^{-1}$. For arbitrary $\mu \in \MM_1 (\SS)$, we define $\mu \circ \tau_t \in \MM_1 (\SS)$ such that $(\mu\circ \tau_t)(A) = \mu (\tau_t A)$ for any measurable set $A$ in $\SS$. Note that the definition of $\mu \circ \tau_t$ is consistent the definition of $\mu \circ \tau^{-1}_{-t}$ in the second paragraph of Section \ref{sec:translationinvariant}.    
\begin{lemma}\label{lem:currentdis}	
Suppose that we have $\red{\eta} \sim (\mu_* \circ \tau_t) = \otimes_{k \in \ZZ} \Ber(\frac{q^{-k+t}}{1 + q^{-k+t}})$, then for arbitrary $y \in \mathbb{Z}$, we have that $\red{K_y} \sim \text{Ber}\big(\frac{q^{-y+t}}{1 + q^{-y+t}})$.
\end{lemma}
\begin{proof}
	It suffices to prove that $\mathbb{E}^{\mu_* \circ \tau_t}[K_y] := \int \mathbb{E}^{\eta} [K_y]\, (\mu_* \circ\tau_t)(d\eta) = \frac{q^{-y+t}}{1 + q^{-y+t}}$.
	Using Lemma \ref{lem:currentrec}, we have 
\begin{align*}
		\mathbb{E}^{\eta}[K_y] 
		&= \mathbbm{1}_{\{\eta(y) = 1\}}  \Big(\mathbb{E}^{\eta}[K_{y-1}] +  \big(1 - \mathbb{E}^{\eta}[K_{y-1}]\big) (1 - b_1)\Big) + \mathbbm{1}_{\{\eta (y) = 0\}} \mathbb{E}^{\eta}[K_{y-1}] b_2.
		\end{align*}
\red{Integrating both sides above (as a function $\eta$) with respect to the probability measure $\mu_* \circ \tau_t$. Note that $\mathbb{E}^\eta[K_{y-1}]$ only depends on $\{\eta(z)\}_{z \leq y-1}$, this implies that when $\eta \sim \mu_* \circ \tau_t$, $\mathbb{E}^{\eta}[K_{y-1}]$ is independent with $\eta(y)$. We have}  
\begin{align*}
&\mathbb{E}^{\mu_* \circ \tau_t}[K_y]\\ 
&= (\mu_* \circ \tau_t) (\eta(y) = 1)  \Big(\mathbb{E}^{\mu_* \circ \tau_t}[K_{y-1}] +  \big(1 - \mathbb{E}^{\mu_* \circ \tau_t} [K_{y-1}]\big) \big(1 - b_1\big)\Big) + (\mu_* \circ \tau_t) (\eta (y) = 0) \mathbb{E}^{\mu_* \circ \tau_t}[K_{y-1}] b_2
\end{align*}
We have $(\mu_* \circ \tau_t)(\eta(y) = 1) = \mu_* (\eta(y - t) = 1) = \frac{q^{-y+t}}{1 + q^{-y+t}}$ and $(\mu_* \circ \tau_t)(\eta(y) = 0) = \frac{1}{1 + q^{-y+t}}$. Therefore, 
\begin{align*}
\mathbb{E}^{\mu_* \circ \tau_t}[K_y]
&=\frac{ q^{-y+t}}{1 + q^{-y+t}} \Big(\mathbb{E}^{\mu_* \circ \tau_t}[K_{y-1}] + \big(1 - \mathbb{E}^{\mu_* \circ \tau_t}[K_{y-1}]\big) \big(1 - b_1\big)\Big) + \frac{1}{1 + q^{-y+t}} \mathbb{E}^{\mu_* \circ \tau_t}[K_{y-1}] b_2\\
&= \frac{(q^{-y+1+t} + 1) b_2}{1 + q^{-y+t}}\mathbb{E}^{\mu_* \circ \tau_t}[K_{y-1}] + \frac{(1 -qb_2) q^{-y+t}}{1 + q^{-y+t}}.
\end{align*}
This implies that 
	\begin{equation}\label{eq:currentrecur}
	\mathbb{E}^{\mu_* \circ \tau_t}[K_y] - \frac{q^{-y+t}}{1+  q^{-y+t}} = \frac{( q^{-y+1+t} + 1) b_2 }{1 + q^{-y+t} }\Big(\mathbb{E}^{\mu_* \circ \tau_t}[K_{y-1}] - \frac{q^{-y+1+t}}{1 + q^{-y+1+t}}\Big).
	\end{equation}
	Iterating \eqref{eq:currentrecur} yields that for any $x < y$, we have 
	\begin{equation}\label{eq:iteration}
	\mathbb{E}^{\mu_* \circ \tau_t} [K_y] - \frac{ q^{-y+t}}{1+  q^{-y+t}} = \prod_{i = x+1}^y \frac{( q^{-i+1+t} + 1) b_2}{1 +  q^{-i+t} } \Big(\mathbb{E}^{\mu_* \circ \tau_t}[K_x] - \frac{ q^{-x+t}}{1 +  q^{-x+t}}\Big).
	\end{equation}
Recall that $q = \frac{b_1}{b_2}$, it is straightforward \red{verify that}
	\begin{equation*}
\sup_{i \in \mathbb{Z}}	\frac{( q^{-i+1+t} + 1) b_2}{1 +  q^{-i+t} } = \max(b_1, b_2) < 1. 
	\end{equation*}
Fix arbitrary $y$ and let $x \to \red{-}\infty$ in \eqref{eq:iteration}, the product \red{therein} tends to zero. Hence, $\mathbb{E}^{\mu_* \circ \tau_t}[K_y] =  \frac{ q^{-y+t}}{1 +  q^{-y+t}}$ for arbitrary $y \in \mathbb{Z}$. Thus, we conclude the lemma.  
\end{proof}
\begin{prop}\label{prop:mu*mfI}
	We have
	$\mu_* \in \mathfrak{I}.$  As a consequence, for arbitrary $n \in \mathbb{Z}$, $\mu_{(n)} \in \mathfrak{I}$.
\end{prop}
\begin{proof}
We need to show that $\tau_t(\eta_t) \sim \mu_*$ implies $\tau_{t+1}(\eta_{t+1}) \sim \mu_*$.	
It suffices to show  that if $\eta_t \sim \mu_* \circ \tau_{t}$, then we have $\eta_{t+1} \sim \mu_* \circ \tau_{t+1}$. Fix arbitrary $y \in \mathbb{Z}$. Given $\eta_t \sim \mu_* \circ \tau_t$,
using Lemma \ref{lem:currentdis} \red{with $\eta =  \eta_t$}, we know that $K_{y-1} \sim \text{Ber}(\frac{\red{q^{-y+1+t}}}{1 +  q^{-y+1 + t}})$. Note that for all $y \in \mathbb{Z}$, $K_y$ is given by \eqref{eq:current} with $\eta = \eta_t$.
	
Since $K_{y-1}$ only depends on $\eta_{t} (z), z \leq y-1$, we have that $K_{y-1}$ and $\eta_t (y)$ are independent. 
Set $h = K_{y-1}$, $v = \eta_t (y)$, $h' = K_y$, $v' = \eta_{t+1}(y)$ and $\rho = \frac{ q^{-y+t}}{1 +  q^{-y+t}}$, $\zeta = \frac{q^{-y+1+t}}{1 +  q^{-y+1+t}}$, we apply Lemma \ref{lem:bs} and get $K_y \sim \text{Ber}(\frac{ q^{-y+t}}{1 + q^{-y+t}})$, $\eta_{t+1} (y) \sim \text{Ber}(\frac{ q^{-y+1+t}}{1 +  q^{-y+1+t}})$. In addition, $K_y$ and $\eta_{t+1}(y)$ are independent. We keep applying Lemam 6.1 to location $y+1, y+2, \dots$, and we  conclude that the random variables $\{\eta_{t+1}(z)\}_{z \geq y}$ are independent with $\eta_{t+1}(z) \sim \text{Ber}(\frac{ q^{-z+1+t}}{1 +  q^{-z+1+t}})$. Since $y$ is arbitrary, we conclude that
$$\eta_{t+1} \sim \bigotimes_{k \in \mathbb{Z}} \text{Ber}(\frac{ q^{-k+1+t}}{1 +  q^{-k+1+t}}) = \mu_* \circ \tau_{t+1}.$$ 
This implies that $\mu_* \in \mfI$.
	
We proceed to show that $\mu_{(n)} \in \mfI$. Without loss of generality, we assume that $q > 1$. Recall that $\mu_*$ is supported on the set $A$ defined in \eqref{eq:setA} and recall the sets $A_n$ from \eqref{eq:setAn}. 	
	Note that $A = \cup_{n = -\infty}^\infty A_n$, moreover, it is straightforward to check that $A_n$ is an irreducible subspace of the shifted S6V model $(\tau_t \circ \eta_t)_{t \in \mathbb{Z}_{\geq 0}}$. Since we have proved that $\mu_* \in \mfI$, we conclude that $\mu_{(n)} \in \mfI$ for every $n \in \mathbb{Z}$. 
\end{proof}
\section{Uniqueness of blocking measures and Proof of Theorem \ref{thm:s6vmoving}}\label{sec:classify}
In this section, we show that there is no extremal stationary distribution for the shifted S6V model other than $\{\mu_\rho\}_{\rho \in [0, 1]}$ and $\{\mu_{(n)}\}_{n \in \mathbb{Z}}$ and conclude Theorem \ref{thm:s6vmoving}. Note that we can couple two shifted S6V models in a similar way as Definition \ref{def:coupling} by looking at the coupled S6V model under a moving frame. We use $(\wtil{\mathsf{S}}_{\red{t}})_{\red{t \in \mathbb{Z}_{\geq 0}}}$ to denote the semigroup of the coupled shifted S6V model. \red{As stated in Theorem \ref{thm:s6vmoving}, we assume $q = \frac{b_1}{b_2} \neq 1$ throughout the section.} 
\begin{lemma}\label{lem:blocking}
	If $\nu > \nu \circ \tau^{-1}$, then we have 
	\begin{equation}\label{eq:deninfty}
	\lim_{n \to -\infty} \nu \big(\eta(n) = 1\big) = 1, \qquad \lim_{n \to \infty} \nu\big(\eta(n) = 1\big) = 0.
	\end{equation} 
	If $\nu < \nu \circ \tau^{-1}$, then we have 
	\begin{equation*}
	  \lim_{n \to -\infty} \nu\big(\eta(n) = 1\big) = 0, \qquad  \lim_{n \to \infty} \nu\big(\eta(n) = 1\big) = 1.
	\end{equation*}
\end{lemma}
\begin{proof}
	It suffices to explain the proof when  $\nu > \nu \circ \tau^{-1}$. We have 
	\begin{align*}
	1 \leq \liminf_{n \to \infty}\mathbb{E}^{\nu}\Big[\sum_{m = -n}^n \eta(m)\Big] - \mathbb{E}^{\nu \circ \tau^{-1}}\Big[\sum_{m = -n}^n \eta(m)\Big]
	& = \liminf_{n \to \infty}\mathbb{E}^{\nu}\Big[\sum_{m = -n}^n \eta(m)\Big] - \mathbb{E}^{\nu}\Big[\sum_{m = -n}^n \eta(m+1)\Big]\\
	&= \liminf_{n \to \infty} \nu\big(\eta(-n) = 1\big) - \nu\big(\eta(n+1) = 1\big) \leq 1. 
	\end{align*}
The inequality on the left hand side is due to $\nu > \nu \circ \tau^{-1}.$	
	As a consequence, we know that $$ \liminf_{n \to \infty} \nu\big(\eta(-n) = 1\big) - \nu\big(\eta(n+1) = 1\big) = 1.$$
Since the probability measure takes value between $0$ and $1$, we conclude \eqref{eq:deninfty}. The proof for the situation $\nu < \nu \circ \tau^{-1}$ is similar.
\end{proof}

\red{
\begin{lemma}\label{lem:temp3}
Let $\nu \in \mfI_e$. We have either $\nu > \nu \circ \tau^{-1}$, $\nu < \nu \circ \tau^{-1}$ or $\nu = \nu \circ \tau^{-1}$. 
\end{lemma}
\begin{proof}
The proof of lemma is rather similar to what we have done for proving Corollary \ref{cor:shift} and Proposition \ref{prop:coupled}. The argument that works for the S6V model also applies similarly for the shifted S6V model. Hence the proof is omitted. 
\end{proof}
}

\red{Let $\mathcal{P}^c$ be the set of non-translation invariant distributions, which complements $\mathcal{P}$.}
\begin{lemma}\label{lem:order}
\red{Fix $q \neq 1$} and $\nu \in \mathfrak{I}_e \cap \mathcal{P}^c$. If $q > 1$, we have $\nu > \nu \circ \tau^{-1}$.
	If $q < 1$, we have $\nu < \nu \circ \tau^{-1}$. 
\end{lemma}
\begin{proof}
\red{By Lemma \ref{lem:temp3}}, we know that for $\nu \in \mfI_e \cap \PP^c$, we have either $\nu > \nu \circ \tau^{-1}$ or $\nu < \nu \circ \tau^{-1}$. 
we only prove \red{that $q > 1 \Rightarrow \nu > \nu \circ \tau^{-1}$. The proof for $q < 1 \Rightarrow \nu < \nu \circ \tau^{-1}$ is similar. We argue by contradiction.} 

If  \red{$q > 1$ and} $\nu < \nu \circ \tau^{-1}$, by Lemma \ref{lem:blocking}, we have 
	\begin{equation*}
	\lim_{n \to \infty} \nu\big(\eta(n) = 1\big) = 1.
	\end{equation*}
	This shows that the density of particles tends to $1$ as we \red{move to the right}, thus $\lim_{y \to \infty} \mathbb{P}^{\nu} \big(K_y = 1\big) = 1$. Since $\nu$ is stationary for $\tau_t \circ \eta_t$,  
	we have for any $x < y$, 
	\begin{equation*}
	\mathbb{E}^{\nu}\Big[\sum_{k = x}^{y} \eta_0 (k)\Big] = \mathbb{E}^{\nu}\Big[\sum_{k = x+1}^{y+1} \eta_1 (k)\Big].
	\end{equation*}
This implies that 
\begin{equation*}
\mathbb{E}^{\nu}\Big[\sum_{k = x}^{y} \eta_{\red{1}} (k)\Big] - \mathbb{E}^{\nu}\Big[\sum_{k = x}^{y} \eta_{\red{0}} (k)\Big] = \red{\mathbb{E}^\nu[ \eta_1 (x) - \eta_1 (y+1)]}. 
\end{equation*}
Using Lemma \ref{lem:currentconserve}, we have
	\begin{equation}\label{eq:shiftstat}
	\mathbb{E}^{\nu}\big[K_{x-1}  - K_{y} \big] =  \mathbb{E}^{\nu} \big[\eta_1 (x) - \eta_1 (y+1)\big].
	\end{equation}
Since we have shown that $\lim_{y \to \infty}\mathbb{P}^{\nu}\big(\eta_1 (y+1) = 1\big) = \lim_{y \to \infty}\mathbb{P}^{\nu}\big(K_y  = 1\big) = 1$, by letting $y \to \infty$ in \eqref{eq:shiftstat}, we have $\mathbb{E}^{\nu} \big[\eta_1 (x)\big] = \mathbb{E}^{\nu}\big[K_{x-1}\big]$, which is equivalent to 
\begin{equation}\label{eq:blocking}
	\mathbb{P}^{\nu}\big(\eta_1 (x) = 1\big) = \mathbb{P}^{\nu} \big(K_{x-1} = 1\big).
\end{equation} 
Since $\eta_0 (x), K_{x-1}, \eta_1 (x), K_x$ are the number of lines on the bottom, left, top and right of the vertex at $(x, 0)$, we have 
	\begin{align*}
\red{\mathbb{P}^{\nu} \big(K_{x-1} = 1\big)} &= 	\mathbb{P}^{\nu}\big(\eta_1 (x) =1\big)\\ 
&= \mathbb{P}^{\nu}\big(K_{x-1} = 1, \eta_0 (x) = 1\big) + b_1 \mathbb{P}^{\nu}\big(K_{x-1} = 0, \eta_0 (x) = 1\big) + (1 - b_2) \mathbb{P}^{\red{\nu}}\big(K_{x-1} = 1, \eta_0 (x) = 0\big).
	\end{align*}
This yields	that $b_2 \mathbb{P}^{\nu} \big(K_{x-1} = 1, \eta_0 (x) = 0\big) = b_1 \mathbb{P}^{\nu} \big(K_{x-1} = 0, \eta_0 (x) = 1\big)$. Since $q > 1$, we have $b_1 > b_2$ and thus 
\begin{equation*}
\mathbb{P}^{\nu}\big(K_{x-1} = 1, \eta_0 (x) = 0\big) > \mathbb{P}^{\nu}\big(K_{x-1}(\eta_0) = 0, \eta_0 (x) = 1\big).
\end{equation*} 
\red{Adding $\mathbb{P}^\nu \big(K_{x-1} = \eta_0 (x) = 1\big)$ to both sides above}, we have \begin{equation*}
\mathbb{P}^{\nu}\big(K_{x-1} = 1\big) > \mathbb{P}^{\nu} \big(\eta_0 (x) = 1\big).
\end{equation*} 
By \eqref{eq:blocking} and the stationarity of $\nu$ for the shifted S6V model $(\eta_t')_{t \in \mathbb{Z}_{\geq 0}} = (\tau_t (\eta_t))_{t \in \mathbb{Z}_{\geq 0}}$, we know that $\mathbb{P}^\nu \big(\eta_0 (x-1) \red{\  =1\big)} = \mathbb{P}^\nu \big(\eta_1 (x)   =1\big) =  \mathbb{P}^{\nu}\big(K_{x-1} = 1\big)$, thus $\mathbb{P}^\nu \big(\eta_0 (x-1) \red{\  = 1\big)} > \mathbb{P}^{\nu} \big(\eta_0 (x) = 1\big)$\red{.}  However, this contradicts 
	the assumption that $\nu < \nu \circ \tau^{-1}$. Hence $q > 1$ implies that $\nu > \nu \circ \tau^{-1}$.
\end{proof}
\red{
\begin{lemma}\label{lem:ordering}
For two different probability measures $\nu_1, \nu_2 \in \mfI_e \cap \mathcal{P}^c$, we have either $\nu_1 > \nu_2$ or $\nu_1 < \nu_2$. 
\end{lemma}
}
\begin{proof}
	The proof follows the spirit of the proof of Proposition \ref{prop:coupled}. 
\red{We assume that $q > 1$ without loss of generality.  
By Lemma \ref{lem:order}, we know that for each $i \in \{1, 2\}$, $\nu_i > \nu_i \circ \tau^{-1}$.} Using this and Lemma \ref{eq:deninfty}, we know that
\begin{equation}\label{eq:asympdensity}
\red{\lim_{n \to -\infty} \nu_i\big(\eta(n) = 1\big) = 1, \qquad \lim_{n \to \infty} \nu_i\big(\eta(n) = 1\big) = 0.}
\end{equation}

Now we consider the coupled shifted S6V model starting from the probability distribution $\red{\theta} := \nu_1 \otimes \nu_2$. \red{Recall that the semigroup of the coupled shifted S6V model $(\eta'_t,\xi'_t)_{t \in \mathbb{Z}_{\geq 0}}$ is given by $(\wtil{\mathsf{S}}_{\red{t}})_{\red{t \in \mathbb{Z}_{\geq 0}}}$.} Let $\lambda$ be a limit point of $\frac{1}{n}\sum_{i = 0}^{n-1} \red{\theta} \wtil{\mathsf{S}}_i$ as $n \to \infty$. Since $\nu_1, \nu_2 \in \mfI$, the marginals of $\lambda$ are respectively $\nu_1$ and $\nu_2$. By \eqref{eq:asympdensity}, the particle densities of $\nu_1$ and $\nu_2$ have the same $n \to \pm\infty$ asymptotic, this implies that 
\begin{equation*}
\red{\lim_{n \to \infty} \frac{1}{n} \mathbb{E}^{\theta}\Big[\#\{x \in [-M n, Mn] \cap \mathbb{Z}:  \eta'_{0} (x) \neq \xi'_{0} (x)\}\Big] = 0.}
\end{equation*}
\red{Since the number of discrepancies that enter $[-Mn, Mn]$ some time between $[0, \lfloor \sqrt{n} \rfloor]$ is at most $\mathcal{O}(\sqrt{n})$, we have}
\begin{equation*}
\lim_{n \to \infty} \frac{1}{n} \mathbb{E}^{\theta}\Big[\#\{x \in [-M n, Mn] \cap \mathbb{Z}:  \eta'_{\lfloor \sqrt{n} \rfloor} (x) \neq \xi'_{\lfloor \sqrt{n} \rfloor} (x)\}\Big] = 0.
\end{equation*}
On the other hand, it is straightforward to verify that the analogue of Lemma \ref{lem:discrepanciesnb2} holds for the shifted S6V model, i.e. fix $M  > 1 + \frac{1}{1 - \max(b_1, b_2)}$, we have $\mathbb{E}^{\theta}[\mathsf{N}_{-M n} (n)] = o(n)$, where $\N_y(n)$ is be the number of discrepancies in the coupled shifted S6V model 
that start from $(-\infty, y]$ at time $0$ and reach $[-n, n]$ between time $[\lfloor \sqrt{n} \rfloor, n]$. As in Proposition \ref{prop:liggettineq}, we let $w_{\sqrt{n}, n}$ be the number of discrepancies of the coupled shifted S6V models that visited $[-n, n]$ between time intervals $[\lfloor \sqrt{n} \rfloor, n]$. We have 
\begin{equation*}
w_{\sqrt{n}, n} \leq \red{\mathsf{N}}_{-Mn} (n) + \#\Big\{x \in [-M n, Mn] \cap \mathbb{Z}: \eta_{\lfloor \sqrt{n} \rfloor} (x) \neq \xi_{\lfloor \sqrt{n} \rfloor} (x)\Big\}.
\end{equation*} 
Therefore,
we have $\lim_{n \to \infty} \frac{1}{n} \mathbb{E}^{\theta} [w_{\sqrt{n}, n}] = 0$. 
Then the same argument as in the proof of Proposition \ref{prop:coupled} yields that $\lambda$ satisfies \eqref{eq:nodiscrepancy}. Since $\nu_1, \nu_2$ are different elements of $\mfI_e$, we conclude that either $\nu_1 > \nu_2$ or $\nu_1 < \nu_2$.
\end{proof}

\begin{lemma}\label{lem:mfIshiftinvariance}
We have $(\mfI \cap \PP)_e = \{\mu_\rho\}_{\rho \in [0, 1]}$\red{.}
\end{lemma}
\begin{proof}
Note that we have $\mfI \cap \PP = \II \cap \PP$. The result then follows from Corollary \ref{cor:agg}.
\end{proof}
\red{We proceed to prove Theorem \ref{thm:s6vmoving}.}
\begin{proof}[Proof of Theorem \ref{thm:s6vmoving}]
	For arbitrary fixed $n \in \mathbb{Z} \cup \{-\infty, \infty\}$, let $\delta_n$ \red{be} a delta probability measure which is supported on the  particle configuration $\eta$, where $\eta(x) = \mathbbm{1}_{\{x \leq n\}}$ for every $x \in \mathbb{Z}$. In particular, $\delta_\infty$ denotes the particle configuration where there is a particle everywhere and $\delta_{-\infty}$ denotes the particle configuration where there is no particle.

Let us show an ordering for the blocking measures. Recall that $\mu_{(n)}$ is the stationary distribution for the shifted S6V model $(\eta'_t)_{t \in \mathbb{Z}_{\geq 0}} = (\tau_t(\eta_t))_{t \in \mathbb{Z}_{\geq 0}}$ on the 
countable state space $A_n$. Since $\eta'_t$ is irreducible on $A_n$, $\eta'_t$ with initial distribution $\mu_{(n)}$ is positive recurrent.  
This implies 
that if we start the S6V model with initial distribution $\delta_n$ and let time $t$ goes to infinity, the distribution of $\eta_t$ converges to $\mu_{(n)}$. In other words, for every $n \in \mathbb{Z}$, $\lim_{t \to \infty} \delta_n \mathsf{S}_t = \mu_{(n)}$, recall that  $(\mathsf{S}_t)_{t \in \mathbb{Z}_{\geq 0}}$ is the semigroup of the shifted S6V model. Since $\delta_n > \delta_{n-1}$, using Lemma \ref{lem:mono}, we have $\mu_{(n)} > \mu_{(n-1)}$.   Moreover, \red{one can verify that}
\begin{equation*}
\lim_{n \to \infty}\mu_{(n)} 
= \delta_\infty, \qquad \lim_{n \to -\infty}\mu_{(n)} = \delta_{-\infty}.
\end{equation*}
Using the above discussion and Lemma \ref{lem:ordering}, if $\nu \in \mfI_e \cap \mathcal{P}^c$ 
and $\nu \notin \{\mu_{(n)}\}_{n \in \mathbb{Z}}$, there exists $k \in \mathbb{Z}$ such that $\mu_{(k-1)} <  \nu < \mu_{(k)}$. This implies that
\begin{align*}
&\liminf_{N \to 
	\infty}\mathbb{E}^{\red{\nu}}\Big[\sum_{i  = -N}^N \eta(i)\Big] - \mathbb{E}^{\red{\mu_{(k-1)}}}\Big[\sum_{i  = -N}^N \eta(i)\Big] \geq 1,\\
&\liminf_{\red{N} \to 
	\infty}\mathbb{E}^{\red{\mu_{(k)}}}\Big[\sum_{i  = -N}^N \eta(i)\Big] - \mathbb{E}^{\red{\nu}} \Big[\sum_{i  = -N}^N \eta(i)\Big] \geq 1.
\end{align*}
\red{Combine the two inequalities above and replace $N$ with $n$}, we obtain that 
\begin{equation}
\label{eq:diff}
\liminf_{n \to \infty} \mathbb{E}^{\mu_{\red{(k)}}}\Big[\sum_{i = -n}^n\eta(i)\Big] - \mathbb{E}^{\mu_{(k\red{-1})}}\Big[\sum_{i = -n}^n\eta(i)\Big] \geq 2. 
\end{equation}
Since $\delta_k \circ \tau^{-1} = \delta_{k-1}$ and $\lim_{t \to \infty} \delta_k \mathsf{S}_t = \mu_{(k)}$, we have 
\begin{equation}\label{eq:shifting}
\mu_{(k)} \red{\circ \tau^{-1}} = \mu_{(k-1)}.
\end{equation} 
Therefore, we have
\begin{equation*}
\limsup_{n \to \infty} \mathbb{E}^{\mu_{\red{(k)}}}\Big[\sum_{i = -n}^n\eta(i)\Big] - \mathbb{E}^{\mu_{(k\red{- 1})}}\Big[\sum_{i = -n}^n\eta(i)\Big] = \limsup_{n \to \infty} \mu_{(k)}\big(\eta(\red{-n}) = 1\big) -  \mu_{(k)}\big(\eta(\red{n+1}) = 1\big) \leq 1.
\end{equation*}  
This yields a contradiction with \eqref{eq:diff}. Hence, there exists $n \in \mathbb{Z}$ such that $\mu_{(n)} = \nu$. This implies that $\mfI_e \cap \PP^c \subseteq \{\mu_{(n)}\}_{n \in \mathbb{Z}}$. Moreover, by Lemma \ref{lem:mfIshiftinvariance}, we have $\mfI_e \cap \PP  \subseteq (\mfI \cap \PP)_e = \{\mu_\rho\}_{\rho \in [0, 1]}$. Therefore, we get 
\begin{equation*}
\mfI_e  = (\mfI_e \cap \PP) \cup (\mfI_e \cap \PP^c) \subseteq \{\mu_\rho\}_{\rho \in [0, 1]} \cup \{\mu_{(n)}\}_{n \in \ZZ}.  
\end{equation*}
On the other hand, by Lemma \ref{lem:prodstat} and Proposition \ref{prop:mu*mfI}, we know that $\{\mu_\rho\}_{\rho \in [0, 1]} \cup \{\mu_{(n)}\}_{n \in \ZZ} \subseteq \mfI$. It is straightforward to see that every element of $\{\mu_\rho\}_{\rho \in [0, 1]} \cup \{\mu_{(n)}\}_{n \in \ZZ}$ \red{cannot} be written as a linear combination of the others. Hence, we conclude that $\mfI_e = \{\mu_\rho\}_{\rho \in [0, 1]} \cup \{\mu_{(n)}\}_{n \in \ZZ}$.
\end{proof}

\section{Fusion and proof of Theorem \ref{thm:st1} and \ref{thm:blockingHS}}\label{sec:sketch}
In this section, we prove Theorem \ref{thm:st1} and \ref{thm:blockingHS}. Recall the SHS6V model from Section \ref{sec:extenHS}, we call $I/2$ the \emph{vertical spin}, $J/2$ the \emph{horizontal spin} and $\alpha$ the \emph{spectral parameter} of the SHS6V model. 
Recall that the $\alpha, q$ parameters of the SHS6V model satisfied either condition \eqref{item:condition1} or \eqref{item:condition2} in Section \ref{sec:extenHS}.

Given an integer $a$ and a positive integer $b$, we set $a_{[b]} := a \text{ mod } b$. Set $\beta(x, t) := x_{[I]} + t_{[J]} $ and define
\begin{equation}\label{eq:b1b2}
\begin{split}
b_1 (x, t) &:= \L_{\alpha q^{\beta(x, t)}}^{1, 1} (1, 0; 1, 0) =  \frac{1+\alpha q^{\beta(x, t) +1}}{1 + \alpha q^{\beta(x, t)}}, \\
b_2 (x, t) &:= \L_{\alpha q^{\beta(x, t)}}^{1, 1} (\red{0}, \red{1}; \red{0}, \red{1}) \red{\ = \ }  \frac{\alpha q^{\beta(x, t)} + q^{-1}}{1 + \alpha q^{\beta(x, t)}}.
\end{split}
\end{equation}
We see that  $b_1(x, t), b_2(x, t)$ are periodic in $x, t$ with $b_i (x+I, t) = b_i (x, t+J) =  b_i (x, t)$ for $i = 1, 2$. In addition, $b_1 (x, t), b_2 (x, t) \in (0, 1)$.
It turns out the SHS6V model in Definition \ref{def:fused} can be identified as a fused version of an inhomogeneous version of the S6V model that is called \emph{the unfused SHS6V model}. 
\begin{defin}[Unfused SHS6V model with \red{finitely many} particles]
\label{def:unfused}
The unfused SHS6V model with \red{finitely many} particles $\overline{\mathbf{p}}_t = (\overline{p}_t (-M) < \dots < \overline{p}_t (N))$ is defined in the same way as Definition \ref{def:s6v}, except that we replace the probability distribution \eqref{eq:updaterule} of independent random variables $\{\chi_t (x), j_t (x): x\in \mathbb{Z}, t \in \mathbb{Z}_{\geq 0}\}$ with
\begin{align*}
&\mathbb{P}(\chi_t(x) = 1) = b_1(x, t), \qquad \mathbb{P}(\chi_t (x) = 0) = 1 - b_1(x, t), \\ &\mathbb{P}(j_t (x) = n) = (1 - b_2(t, x+n)) \prod_{k = 1}^{n-1} b_2(t, x+k), \qquad \forall \, n \in \mathbb{Z}_{\geq 1}.
\end{align*} 
\end{defin}
We can define the unfused SHS6V model with \red{infinitely many} particles using a similar argument appeared in Section \ref{sec:infinite}. Note that at any time $t \in \mathbb{Z}_{\geq 0}$, the independent random variables $\{\chi_t (x)\}_{x \in \mathbb{Z}}$ have probability distribution given by $\chi_t (x) \sim \text{Ber}(b_1 (x, t))$ for every $x \in \mathbb{Z}$. By \eqref{eq:b1b2}, we know that $b_i (x+I, t) = b_i (x, t+J) = b_i (x, t)$ for all $i \in \{1, 2\}, t \in \mathbb{Z}_{\geq 0}$ and $x \in \mathbb{Z}$. As a consequence, we know that  $\sup_{i \in \{1, 2 \}, t \in \mathbb{Z}_{\geq 0}, x \in \mathbb{Z}} b_i (x, t) < 1$. Note that $\mathbb{P}(\chi_t (x) = 0) = 1 - b_1 (x, t)$,
by Borel-Cantelli lemma and the fact that $1 - b_1 (x, t)$ is uniformly lower bounded away from $0$, we know that with probability $1$, we can find a random sequence of integers $\dots <d_{-1} < d_0 < d_1 < \cdots$  such that $\chi_t (\overline{p}_t (d_i)) = 0$ for every $i \in \mathbb{Z}$. The rest of the argument follows as in Section \ref{sec:infinite}. 

Define the unfused SHS6V model occupation process $\overline{\eta}_t \in \SS$ such that for every $t \in \mathbb{Z}_{\geq 0}$ and $x \in \mathbb{Z}$, 
\begin{equation*}
\overline{\eta}_t (x) := \mathbbm{1}_{\{\exists i, \overline{p}_t(i) = x\}}. 
\end{equation*}
Clearly, $(\overline{\eta}_t)_{t \in \mathbb{Z}_{\geq 0}}$ is a time-inhomogeneous Markov process. We use $(\overline{S}_t)_{t \in \mathbb{Z}_{\geq 0}}$ to denote its semigroup. 

Fix arbitrary $n \in \mathbb{Z}_{\geq 1}$ and define a map $f_n: \mathbb{R}^n \to \mathbb{R}$ such that $f_n (v_1, \dots, v_n) = \sum_{k = 1}^n v_k$.
Define a matrix $\Lambda_n$ with row and column indexed by $v \in \{0, 1, \dots, n\}$ and $(v_1, \dots, v_n) \in \{0, 1\}^n$ with entries 
\begin{equation*}
\Lambda_n \big(v; (v_1, \dots, v_n)\big) = Z_{v, n}^{-1} \mathbbm{1}_{\{f_n (v_1, \dots, v_n) = v\}} q^{\sum_{r = 1}^n (r-1) v_r}.
\end{equation*}
where $Z_{v, n}$ is the normalizing constant 
\begin{equation*}
Z_{v, n} := \sum_{(v_1, \dots, v_n) \in \{0, 1\}^n} \mathbbm{1}_{\{f_n(v_1, \dots, v_n) = v\}} q^{\sum_{r = 1}^n (r-1) v_r}. 
\end{equation*} 
Note that $\Lambda_n$ is a stochastic matrix. 
By slightly abusing the notation, we can also view $\Lambda_n$ as a map from $\MM_1(\{0, 1, \dots, n\})$ to $\MM_1 (\{0, 1\}^n)$. For $\mathsf{u} \in \MM_1(\{0, 1, \dots, n\})$, we define $\Lambda_n \mathsf{u} \in \MM_1 (\{0, 1\}^n)$ such that for every $(v_1, \dots
, v_n) \in \{0, 1\}^n$,
\begin{equation*}
(\Lambda_n \mathsf{u})\big((v_1, \dots, v_n)\big) := \sum_{v = 0}^n \mathsf{u}(v) \Lambda_n \big(v; (v_1, \dots, v_n)\big).
\end{equation*}
Note that there is only one non-zero term in the above summation with $v = f_n(v_1, \dots, v_n)$. 

We also define a stochastic matrix $\Lambda'_n$ such that for all $v \in \{0, 1, \dots, n\}$ and $(v_1, \dots, v_n) \in \{0, 1\}^n$,
\begin{equation*}
\Lambda'_n \big(v; (v_n, \dots, v_1)\big) = \Lambda_n \big(v; (v_1, \dots, v_n)\big).
\end{equation*}
Similarly, we can also view $\Lambda_n'$ as a map from $\MM_1 (\{0, 1, \dots, n\})$ to $\MM_1 (\{0, 1\}^n)$ such that for $\mathsf{u}' \in 
\MM_1 (\{0, 1, \dots, n\})$, $v \in \{0, 1, \dots, n\}$ and $(v_1, \dots, v_n) \in \{0, 1\}^n$,
\begin{equation*}
\Lambda_n' \mathsf{u}'\big((v_1, \dots, v_n)\big) = \sum_{v = 0}^n \mathsf{u}'(v) \Lambda'_n \big(v; (v_1, \dots, v_n)\big). 
\end{equation*}
We proceed to define a map $\boldsymbol{\Lambda}'_n$ from $\MM_1 (\{0, 1, \dots, n\}^\ZZ)$ to $\MM_1(\{0, 1\}^\mathbb{Z})$ such that for any $\mu \in \MM_1 (\{0, 1, \dots, n\}^\ZZ)$, we have $\Lambda'_n \mu \in \MM_1 (\{0, 1\}^{\ZZ})$ satisfying for any integers $a < b$ and $v_i \in \{0, 1\}$, $i = an, an+1, \dots, bn - 1$, we have
\begin{align*}
&\boldsymbol{\Lambda}_n' \mu\big(\{\eta \in \{0, 1\}^\ZZ: \eta(i) = v_i \text{ for all } i \in [an, bn-1]\}\big)\\ 
&= \sum_{\substack{u_k \in \{0, 1, \dots, n\}\\ k \in [a, b-1]}}\mu(\{g \in \{0, 1, \dots, n\}^\ZZ: g(k) = u_k \text{ for all } k \in [a, b-1]\}) \prod_{k = a}^{b-1} \Lambda'_n \big(u_k; (v_{kn}, v_{kn+1}, \dots, v_{(k+1)n-1})\big).
\end{align*} 
We also define a map $\mathbf{f}_{n} : \{0, 1\}^{\ZZ} \to \{0, 1, \dots, n\}^{\ZZ}$ such that for all $\eta \in \{0, 1\}^\ZZ$ we have $g = \mathbf{f}_n (\eta)$ with $g(k) = \sum_{i = kn}^{(k+1)n - 1} \eta(i)$ for all $k \in \ZZ$. In words, the map $\mathbf{f}_n$ collapses every $n$ neighboring lattice points, collects the particles at these $n$ lattice points and put them at a single lattice point.
 
The following proposition shows that the unfused SHS6V model has the same probablity distribution as the SHS6V model, after we collapse every $I$ steps of space and $J$ steps of time into one. 
\begin{prop}[fusion]
	\label{prop:fusion}
	
	We consider the SHS6V model $(g_t)_{t \in \mathbb{Z}_{\geq 0}}$ defined in Definition \ref{def:fused} and the unfused SHS6V model $(\overline{\eta}_t)_{t \in \mathbb{Z}_{\geq 0}}$ defined in Definition \ref{def:unfused}. Assume that $g_0$ has \red{finitely many} particles and has probability distribution $\mu$. If we have $\overline{\eta}_0 \sim \bLambda'_I \mu$, then  
	\begin{equation*}
	(g_t, t \geq 0) \overset{d}{=} (\mathbf{f}_I (\overline{\eta}_{Jt}), t \geq 0).
	\end{equation*} 
\end{prop}   

According to the paragraph after Definition \ref{def:unfused}, the unfused SHS6V model with infinitely many particles is well-defined. By Proposition \ref{prop:fusion}, the definition of the SHS6V model in Definition \ref{def:fused} can also be extended to the situation when there are infinitely many particles.

Although the result of Proposition \ref{prop:fusion} might be well-known by experts, we do not find a proof of it in the literature. The proof for the situation $I = 1$ can be obtained by the argument in \cite[Section 3]{corwin2016stochastic}. In fact, the argument in \cite[Section 3]{corwin2016stochastic} is related to unfusing a vertex with vertical spin $I/2$ and horizontal spin $J/2$ to a column of $J$ vertices with vertical spin $I/2$ and horizontal spin $1/2$, see Figure \ref{fig:column}. Proposition \ref{prop:fusion} indicates that one can further  unfuse each vertex in Figure \ref{fig:column} into a row of $I$ vertices with vertical and horizontal spin $1/2$, see Figure \ref{fig:fusion}. This is consequence of reflection invariance of the $\L$-matrix in \eqref{eq:reflection}.  

We proceed to prove Proposition \ref{prop:fusion}. We start with some preparation. We call a probability measure $P$ on $\{0, 1\}^n$ \emph{$q$-exchangable} if there exists a probability measure $\mathsf{u}$ on $\{0, 1, \dots, n\}$ such that $P = \Lambda_n \mathsf{u}$. 
The name $q$-exchangable comes from the fact that switching a neighboring pair of $01$ to $10$ would multiple the probability of $(x_1, \dots, x_n) \in \{0, 1\}^n$ by $q$. Similarly, we call a probability measure $P'$ on $\{0, 1\}^n$ \emph{$q^{-1}$-exchangable} if there exists a probability measure $\mathsf{u}'$ on $\{0, 1, \dots, n\}$ such that $P' = \Lambda'_n \mathsf{u}'$.  

Attach vertically $J$ vertices, assign the $i$-th vertex counting from bottom the stochastic matrix $\L_{\alpha q^{i-1}}^{I, 1}$. Fix $v$ number of lines coming from the bottom and $v'$ number of lines going out to top, the column of vertex configurations maps the probability distribution $P_1$ on the incoming lines $(h_1, \dots, h_J)$ to a probability distribution $P_2$ on the outgoing lines $(h_1', \dots, h_J')$, see Figure \ref{fig:column}. 
\begin{figure}[ht]
\centering
\begin{tikzpicture}[scale = 1.5]
\begin{scope}[xshift = 0cm]
\draw[thick] (0, -0.2) -- (2, -0.2);
\draw[thick] (1, -1) -- (1, 1);
\draw[thick] (0, 0.5) -- (2, 0.5); 
\draw[thick] (0, -.5) -- (2, -.5);
\node at (1, -1.3) {$v$};
\node at (1, 1.3) {$v'$};
\node at (0.5, 0.2) {$\vdots$};
\node at (-0.2, -.5) {$h_1$};
\node at (-0.2, -.2) {$h_2$};
\node at (-0.2, .5) {$h_J$};
\node at (1.5, 0.2) {$\vdots$};
\node at (-0.2, .5) {$h_J$};
\node at (2.2, -.5) {$h'_1$};
\node at (2.2, -0.2) {$h'_2$};
\node at (2.2, .5) {$h'_J$};
\node at (1.11, -0.7) {\small{$\alpha$}};
\node at (1.15, -0.35) {\small{$\alpha q$}};
\node at (1.32, 0.35) {\small{$\alpha q^{J-1}$}};
\end{scope}
\end{tikzpicture}
\caption{Attach $J$ vertices and assgin the $i$-th vertex counting from the bottom with stochastic matrix $\L_{\alpha q^{i-1}}^{I, 1}$. The column of vertex configurations maps the probability distribution $P_1$ on $(h_1, \dots, h_J)$ to a probability distribution $P_2$ on $(h'_1, \dots, h'_J)$.}
\label{fig:column}
\end{figure}
\begin{lemma}
\label{lem:qexchange}
Fix arbitrary $I, J \in \mathbb{Z}_{\geq 1}$. If the probability distribution $P_1 \in \MM_1 (\{0, 1\}^n)$ is $q$-exchangable, so is $P_2$. Moreover, we have 
\begin{equation*}
\sum_{\substack{h_i, h_i' \in \{0, 1\}, i = 1,\dots, J \\ v_i \in \{0, 1\}, i = 1, \dots, J-1}} \Lambda_J \big(h; (h_1, \dots, h_J)\big) \prod_{i = 1}^J \L^{I, 1}_{\alpha q^{i-1}} (v_{i-1}, h_i; v_i, h_i') \mathbbm{1}_{\{h' = f_J(h_1', \dots, h_J')\}} = \L^{I, J}_\alpha (v, h; v', h'), 
\end{equation*}
where we set $v_0 = v$ and $v_J = v'$.
\end{lemma}
\begin{proof}
The claim that $P_2$ is $q$-exchangable is proved in \cite[Proposition 5.4]{BP18}. The equality above is proved in \cite[Section 3]{corwin2016stochastic}.
\end{proof}
By Proposition 5.1 of \cite{corwin2016stochastic}, we know that 
\begin{equation}\label{eq:reflection}
\L^{I, J}(i_1, j_1; i_2, j_2\, |\, \alpha, q) = \L^{J, I}(j_1, i_1; j_2, i_2 \,|\, \alpha^{-1}, q^{-1}).
\end{equation}
Note that we write $\L^{I, J}_{\alpha}(\cdot) = \L^{I, J} (\cdot \,|\, \alpha, q)$ to emphasize its dependence on $q$. The action $(i_1, j_1; i_2, j_2) \to (j_1, i_1; j_2, i_2)$ corresponds to a reflection in the diagonal of a vertex configuration. Using this together with Lemma \ref{lem:qexchange}, one obtains the following corollary. 
\begin{cor}\label{cor:qexchangable}
Fix arbitrary $I, J \in \mathbb{Z}_{\geq 1}$. Attach horizontally $I$ vertices, assign the $i$-th vertex counting from left the stochastic matrix $\L^{1, J}_{\alpha q^{i-1}}$. Fix $h$ number of lines coming from the left and $h'$ number of lines going to the right. The row of vertex configurations maps the probability distribution $P_1$ on the incoming lines $(v_1, \dots, v_I)$ to a probability distribution $P_2$ on the outgoing lines $(v_1', \dots, v_I')$, see Figure \ref{fig:row}. If $P_1 \in \MM_1(\{0, 1\}^n)$ is $q^{-1}$-exchangable, then $P_2 \in \MM_1(\{0, 1\}^n)$ is also $q^{-1}$-exchangable. Moreover, we have 
\begin{equation*}
\sum_{\substack{v_i, v_i' \in \{0, 1\}, i = 1,\dots, I \\ h_i \in \{0, 1\}, i = 1, \dots, I-1}} \Lambda'_I \big(v; (v_1, \dots, v_I)\big) \prod_{i = 1}^I \L^{1, J}_{\alpha q^{i-1}} (v_{i}, h_{i-1}; v'_i, h_i) \mathbbm{1}_{\{v' = f_I(v_1', \dots, v_I')\}} = \L^{I, J}_\alpha (v, h; v', h'), 
\end{equation*}
where we set $h_0 = h$ and $h_I = h'$.  
\end{cor}
\begin{figure}[ht]
	\centering
	\begin{tikzpicture}[rotate = -90, scale = 1.5]
	\begin{scope}[xshift = 0cm]
	\draw[thick] (0, -0.2) -- (2, -0.2);
	\draw[thick] (1, -1) -- (1, 1);
	\draw[thick] (0, 0.5) -- (2, 0.5); 
	\draw[thick] (0, -.5) -- (2, -.5);
	\node at (1, -1.3) {$h$};
	\node at (1, 1.3) {$h'$};
	\node at (0.5, 0.2) {$\dots$};
	\node at (-0.2, -.5) {$v'_1$};
	\node at (-0.2, -.2) {$v'_2$};
	\node at (-0.2, .5) {$v'_I$};
	\node at (1.5, 0.2) {$\dots$};
	\node at (2.2, -.5) {$v_1$};
	\node at (2.2, -0.2) {$v_2$};
	\node at (2.2, .5) {$v_I$};
	\node at (1.2, -0.7) {\small{$\alpha$}};
	\node at (1.22, -0.35) {\small{$\alpha q$}};
	\node at (1.18, 0.35) {\small{$\alpha q^{I-1}$}};
	\end{scope}
	\end{tikzpicture}
	\caption{Attach $I$ vertices and assign the $i$-th vertex counting from the bottom with stochastic matrix $\L_{\alpha q^{i-1}}^{I, 1}$. The row of vertex configurations maps the probability distribution $P_1$ on $(v_1, \dots, v_I)$ to a probability distribution $P_2$ on $(v'_1, \dots, v'_I)$.}
	\label{fig:row}
\end{figure}
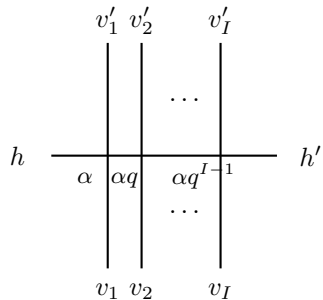
The following proposition says that the weight of a vertex configuration indexed by $(v, h, v', h')$ with vertical spin $I/2$ and horizontal spin $J/2$ and is equal to the sum of the product of the weights of vertex configurations with vertical spin $1/2$ and horizontal spin $1/2$ in a $I \times J$ grid. Note that the spectral parameter of $i$-th and $j$-th the vertex (counting respectively from bottom to top and left to right) is given by $\alpha q^{i+j - 2}$ for all $(i, j) \in \{1, \dots, I\} \times \{1, \dots, J\}$. The probability distributions of the number of lines on the bottom and left of the grid are specified by the matrix $\Lambda_I'$ and $\Lambda_J$, see Figure \ref{fig:fusion}. 
\begin{prop}\label{prop:unfused}
We have for all $(v, h), (v', h') \in \{0, 1 \dots, I\} \times \{0, 1, \dots, J\}$,
\begin{align}
\notag
&\mathsf{L}_{\alpha}^{I, J}(v, h; v', h')\\ 
\notag
&= 
\sum_{\substack{v_{i, j} \in \{0, 1\}\\ i = 1, \dots, I \\ j = 0, \dots, J}} \sum_{\substack{h_{i, j} \in \{0, 1\}\\ i = 0, \dots, I \\ j = 1, \dots, J}}   
\Lambda'_I \big(v; (v_{1, 0}  \dots, v_{I, 0})\big) \Lambda_J \big(h; (h_{0, 1} \dots, h_{0, J})\big) \mathbbm{1}_{\{v' = f_I (v_{1, J}, \dots, v_{I, J})\}}  
\\
\label{eq:key}
&\hspace{9em} \times  \mathbbm{1}_{\{h' = f_J (h_{I, 1}, \dots, h_{I, J})\}}  
 \prod_{i = 1}^{I} \prod_{j = 1}^{J} \L_{\alpha q^{i+j - 2}}^{1, 1} (v_{i, j-1}, h_{i-1, j}; v_{i, j}, h_{i, j}). 
\end{align}
\end{prop}
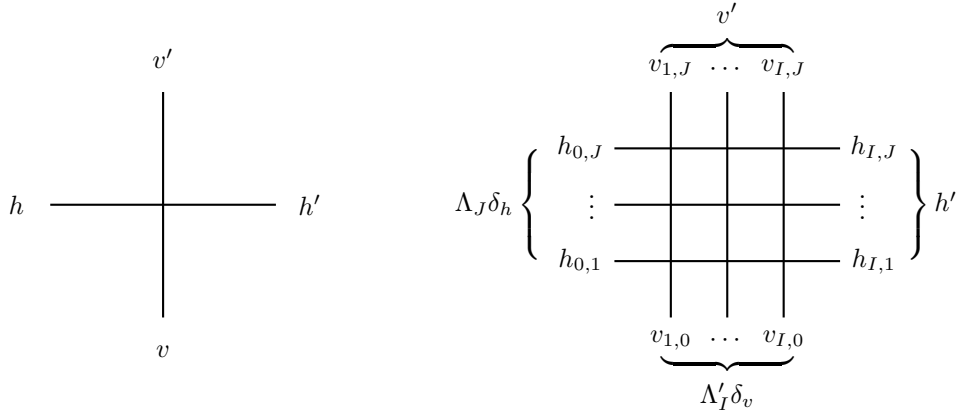
\begin{figure}[ht]
\centering
\begin{tikzpicture}[scale = 1.5]
\begin{scope}[xshift = 0cm]
\draw[thick] (0, 0) -- (2, 0);
\draw[thick] (1, -1) -- (1, 1);
\node at (-0.3, 0) {$h$};
\node at (1, -1.3) {$v$};
\node at (1, 1.3) {$v'$};
\node at (2.3, 0) {$h'$};
\end{scope}
\begin{scope}[xshift = 5cm]
\draw[thick] (0, 0) -- (2, 0);
\draw[thick] (0.5, -1) -- (0.5, 1);
\draw[thick] (1.5, -1) -- (1.5, 1);
\draw[thick] (1, -1) -- (1, 1);
\draw[thick] (0, 0.5) -- (2, 0.5); 
\draw[thick] (0, -.5) -- (2, -.5);
\node at (0.5, -1.2) {$v_{1, 0}$};
\node at (1, -1.2) {$\cdots$};
\node at (1.5, -1.2) {$v_{I, 0}$};
\node at (1, -1.4) {$\underbrace{\hspace{5em}}$};
\node at (1, -1.7) {$\Lambda'_I \delta_v$};
\node at (0.5, 1.2) {$v_{1, J}$};
\node at (1, 1.2) {$\cdots$};
\node at (1.5, 1.2) {$v_{I, J}$};
\node at (1, 1.4) {$\overbrace{\hspace{5em}}$};
\node at (1, 1.7) {$v'$};
\node at (-0.9, 0) {$\Lambda_J \delta_h
\begin{cases} \\ \\ \\ \end{cases}$};
\node at (-0.2, 0.05) {$\vdots$};
\node at (-0.3, -.5) {$h_{0, 1}$};
\node at (-0.3, .5) {$h_{0, J}$};
\node at (2.8, 0) {$\begin{rcases} \\ \\ \\ \end{rcases} h'$};
\node at (2.2, 0.05) {$\vdots$};
\node at (2.3, -.5) {$h_{I, 1}$};
\node at (2.3, .5) {$h_{I, J}$};
\end{scope}
\end{tikzpicture}
\caption{The weight of the vertex configuration $\L_{\alpha}^{I, J}(v, h; v', h')$ on the left is equal to the sum of the product of the weights of $I \times J$ vertex configurations with horizontal spin $1/2$, vertical spin $1/2$ and  geometrically progressing spectral parameters on the right hand side. The probability distributions of the number of the lines on the bottom and left of the grid are specified by the probability distribution $\Lambda_I' \delta_v, \Lambda_J \delta_h$. We also require that $\sum_{i = 1}^I v_{i, J} = v'$ and $\sum_{i = 1}^J h_{I, i} = h'$.}  
\label{fig:fusion}
\end{figure}
\begin{proof}
Applying Corollary \ref{cor:qexchangable} with $J = 1$, we know that for any function $u: \{0, 1\}^I \to \mathbb{R}$,
\begin{align*}
&\sum_{\substack{v_{i, 0}, v_{i, 1} \in \{0, 1\}, i = 1, \dots, I \\ h_{i, 1} \in \{0, 1\}, i = 1,\dots, I-1}} \Lambda'_I \big(v_0, (v_{1, 0}, \dots, v_{I, 0})\big) \prod_{i = 1}^I \L^{1, 1}_{\alpha q^{i-1}} (v_{i, 0}, h_{i-1, 1},; v_{i, 1}, h_{i, 1}) u(v_{1, 1}, \dots, v_{I, 1})\\ 
&= \sum_{v_1 = 0}^I \L^{I, 1}_\alpha (v_0, h_{0, 1}; v_1, h_{I, 1}) \sum_{\substack{v_{i, 1} \in \{0, 1\}\\ i = 1, \dots, I}} \Lambda_I' \big(v_1; (v_{1, 1}, \dots, v_{I, 1})\big) u(v_{1, 1}, \dots, v_{I, 1}).
\end{align*}
We repeatedly apply the above equality and propagate from the bottom row to the top row for the right panel of Figure \ref{fig:fusion}, this significantly simplifies the right hand side of \eqref{eq:key}. We have that the right hand side of \eqref{eq:key} is equal to 
\begin{equation*}
\sum_{\substack{h_{0, i}, h_{I, i} \in \{0, 1\}, i \in \{1, \dots, J\}\\v_i \in \{0, \dots, I\}, i = 1, \dots, J-1}} \Lambda_J \big(h; (h_{0, 1}, \dots, h_{0, J})\big) \prod_{j = 1}^J \L^{I, 1}_{\alpha q^{j-1}} (v_{j-1}, h_{0, j}; v_{j}, h_{I, j}) \mathbbm{1}_{\{h' = f_J (h_{I, 1}, \dots, h_{I, j})\}}.
\end{equation*}
In the above summation,  we set $v_0 = v$ and $v_J = v'$. Note that the $v_j$ in the above summation plays the role of $\sum_{i = 1}^I v_{i, j}$ that appeared on the right hand side of \eqref{eq:key}.  By the displayed equation in Lemma \ref{lem:qexchange}, the above summation is equal to the left hand side of \eqref{eq:key}. This concludes the proposition.
\end{proof} 
\begin{proof}[Proof of Proposition \ref{prop:fusion}]
It suffices to show that the tiling of the SHS6V model specified by $(g_t)_{t \in \mathbb{Z}_{\geq 0}}$ and the tiling of the unfused SHS6V model specified by $(\mathbf{f}_I (\overline{\eta}_{Jt}))_{t \in \mathbb{Z}_{\geq 0}}$ have the same total weight.
Using Proposition \ref{prop:unfused}, each vertex of the SHS6V model can be unfused into a $I \times J$ rectangle of vertices. More precisely, the vertex located at coordinate $(i, j) \in \ZZ \times \ZZ_{\geq 0}$ can be unfused into a $I \times J$ rectangle of vertices that lie on the integer lattice points of the grid $[Ii, I(i+1) - 1] \times [Jj, J(j+1) - 1]$. 
By Proposition \ref{prop:unfused}, each vertex in the rectangle has horizontal spin $1/2$ and vertical spin $1/2$. Moreover, recall that $a_{[b]} := a \text{ mod } b$, we assign the stochastic matrix $\L^{1, 1}_{\alpha q^{x_{[I]} + t_{[J]}}}$ to the vertex at $(x, t) \in [Ii, I(i+1) - 1] \times [Jj, J(j+1) - 1]$. Therefore, the weight of the vertex configuration at $(x, t)$ is  given by Figure \ref{fig:vertexconfig} with parameter $b_1 = b_1 (x, t)$ and $b_2 = b_2 (x, t)$ defined in \eqref{eq:b1b2}. This is exactly the tiling of the unfused SHS6V model defined in Definition \ref{def:unfused}. Using this together with the boundary condition $\widetilde{\eta}_0 \sim \Lambda_I' \mu$ (which is $q^{-1}$-exchangable) and the preservation of the $q$ (or $q^{-1}$)-exchangablility obtained in Lemma \ref{lem:qexchange} and Corollary \ref{cor:qexchangable}, we conclude Proposition \ref{prop:fusion}. 
\end{proof}

Proposition \ref{prop:fusion} indicates that in order to study the set of stationary distributions $\IIf$ of the SHS6V model, it suffices to study the set of distributions that are invariant under $J$-step evolution for the unfused SHS6V model. This motivates us to consider the set of stationary distributions of the time-homogeneous Markov process $(\overline{\eta}_{Jt})_{t \in \mathbb{Z}_{\geq 0}}$. We denote this set to be $\overline{\II}$.

Recall that $\SS = \{0, 1\}^{\ZZ}$, $\GG = \{0, 1, \dots, I\}^{\ZZ}$. The following lemma provides the relation between $\IIf$ and $\Ibar$.
\begin{lemma}\label{lem:statrelation}
Let
$\Muf = \{\bLambda'_I \mu: \mu \in \MM_1 (\GG)\}$.
We have $\IIf = \{\overline{\nu} \circ \mathbf{f}_I^{-1} : \overline{\nu} \in \overline{\II} \cap \overline{\MM}_1 (\SS)\}$.
\end{lemma}
\begin{proof}
By Proposition \ref{prop:fusion}, $\mu \in \II^{\HS}$ is equivalent to $\bLambda'_I \mu \in \Ibar$. Using the surjectivity of $\bLambda'_I: \MM_1 (\GG) \to \overline{\MM}_1 (\SS)$ and the identity $(\bLambda_{I}' \mu) \circ \mathbf{f}_I^{-1} = \mu$, we conclude that $\mu \in \IIf$ iff there exists $\overline{\nu} \in \overline{\II} \cap \overline{\MM}_1 (\SS)$ (which equals $\boldsymbol{\Lambda}'_I \mu$) such that $\mu = \overline{\nu} \circ \mathbf{f}_I^{-1}$.
\end{proof}
We proceed to study the stationary distributions of the time-homogeneous Markov process $(\overline{\eta}_{Jt})_{t \in \mathbb{Z}_{\geq 0}}$. Fix arbitrary $\rho \in [0, 1]$, we define the inhomogeneous Bernoulli product measures 
\begin{equation}\label{eq:inhomoprod}
\mubar_{\rho} := \bigotimes_{k \in \mathbb{Z}} \text{Ber}\bigg(\frac{\rho q^{\mod{(I-1-k)}{I}}}{1 - \rho+ \rho q^{\mod{(I-1-k)}{I}}}\bigg).
\end{equation}
In Appendix \ref{sec:productstat}, we show that the product Bernoulli measures are stationary for the S6V model using the property of local stationarity in Lemma \ref{lem:lcs}. A similar argument will yield that for all $\rho \in [0, 1]$, $\mubar_\rho$ is stationary for the unfused SHS6V model $(\overline{\eta}_t)_{t \in \mathbb{Z}_{\geq 0}}$, see Remark \ref{rmk:ilcs}. 
The key for proving the stationarity is observing that for a vertex at $(x, t) \in \ZZ \times \ZZ_{\geq 0}$ with $v$ $h$, $v'$ and $h'$ being the number of lines on its bottom, left, top and right, if we know that 
$(v, h) \sim \text{Ber}(\rho(x)) \otimes \text{Ber}(\zeta(t))$ where 
$$\rho(x) := \frac{\rho q^{\mod{(I-1-x)}{I}}}{1 - \rho + \rho q^{\mod{(I-1-x)}{I}}}, \qquad \zeta(t) := 
\frac{-\alpha \rho q^{I} q^{\mod{t}{J}}}{1  - \rho -  \alpha \rho q^{I} q^{\mod{t}{J}}}.$$ 
(note that $\rho(x)$ is exactly the density of $\overline{\mu}_{\rho}$ at location $x$) and sample $v', h'$ according to the stochastic weights in Figure \ref{fig:vertexconfig} with $b_1 = b_1 (x, t)$ and $b_2 = b_2 (x, t)$, one can readily check that $(v', h') \sim \text{Ber}(\rho(x)) \otimes \text{Ber}(\zeta(t))$ via a direct computation as in Lemma \ref{lem:lcs}. Note that we have
\begin{equation}\label{eq:ilcs}
\rho(x) (1 - b_1(x, t)) (1 - \zeta(t)) = \zeta(t)  (1  -b_2(x, t)) (1 - \rho(x)),   
\end{equation}
which is a analogue of the relation stated in \eqref{eq:lcs}.

By the explanation in the previous paragraph, we know that $\overline{\mu}_\rho$ is stationary for the Markov process $(\overline{\eta}_t)_{{t \in \mathbb{Z}}_{\geq 0}}$. Therefore, $\overline{\mu}_{\rho}$ is also stationary for the Markov process $(\overline{\eta}_{Jt})_{{t \in \mathbb{Z}}_{\geq 0}}$, hence we have  $\{\overline{\mu}_{\rho}\}_{\rho \in [0, 1]} \subseteq \Ibar$.
\begin{prop}\label{prop:extremeinhomo}
$\Ibar_e = \{\mubar_\rho\}_{\rho \in [0, 1]}$. 
\end{prop}
\begin{proof}
The coupling for the S6V model in Definition \ref{def:coupling} can be generalized to the unfused SHS6V model. The only difference is that we need the random variables $\{\chi^i_t (x)\}_{x \in \mathbb{Z}, t\in \ZZ_{\geq 0}, i = 1, 2}$ and $\{j^i_t (x)\}_{x \in \mathbb{Z}, t \in \ZZ_{\geq 0}, i = 1, 2}$ that appeared in Definition \ref{def:coupling} to be independent and satisfy
\begin{align*}
&\mathbb{P}(\chi^i_t(x) = 1) = b_1(x, t), \qquad \mathbb{P}(\chi^i_t (x) = 0) = 1 - b_1(x, t), \\ &\mathbb{P}(j^i_t (x) = n) = (1 - b_2(t, x+n)) \prod_{i = 1}^{n-1} b_2(t, x+i), \qquad \forall \, n \in \mathbb{Z}_{\geq 1}.
\end{align*} 
Since the parameters $b_1$ and $b_2$ in \eqref{eq:b1b2} are periodic in $x$ with period $I$, the unfused SHS6V model is invariant if we shift the space by $I$. This motivates us to consider the set of $I$-step translation invariant probability distributions to be $\Pbar := \{\nu \circ \tau_I^{-1} = \nu\}$. Using the coupling described in the last paragraph and a similar argument as in Section \ref{sec:translationinvariant}, one can show that $(\Ibar \cap \Pbar)_e = \{\mubar_\rho\}_{\rho \in [0, 1]}$. Moreover, for $\nu \in \Ibar$, we consider the coupled unfused SHS6V models with initial distribution $\theta'_{\nu} = \nu \circ (\Id, \tau_I^{-1})$. Let $(\mathfrak{S}_t)_{t \in \mathbb{Z}_{\geq 0}}$ be the semigroup of the coupled unfused SHS6V models and let $\lambda$ be an arbitrary limit point of the Ces\`{a}ro sum $\frac{1}{n}\sum_{i = 0}^{n-1} \theta_\nu \mathfrak{S}_i$ as $n \to \infty$. Using a similar argument as the proof of Proposition \ref{prop:coupled}, we conclude that 
\begin{equation*}
\lambda\Big(\{(\eta, \xi): \eta > \xi\} \cup \{(\eta, \xi): \eta = \xi\} \cup \{(\eta, \xi): \eta < \xi\}\Big) = 1.
\end{equation*}  
Using an argument which is similar to the proof of Corollary \ref{cor:shift}, we know that if $\nu \in \Ibar_e$, then   
only one of the following holds: $\nu > \nu \circ \tau_{I}^{-1}$ or  $\nu < \nu \circ \tau_I^{-1}$ or $\nu =  \nu \circ \tau_I^{-1}$. Note that the unfused SHS6V model is a space-time inhomogeneous version of the S6V model, one readily sees that the argument in Section \ref{sec:current} also applies to the unfused SHS6V model. In particular, an argument which is similar to the proof of Theorem \ref{thm:st} implies that both $\nu > \nu \circ \tau_{I}^{-1}$ and $\nu < \nu \circ \tau_{I}^{-1}$ are impossible. Hence, $\nu \in \Pbar$. This implies that $\Ibar_e \subseteq \Ibar_e \cap \Pbar \subseteq (\II \cap \Pbar)_e = \{\mubar_\rho\}_{\rho \in [0, 1]}$. On the other hand, the inhomogeneous product measures $\{\mu_\rho\}_{\rho 
\in [0, 1]}$ are ergodic with repsect to the map $\tau_I$, this implies that any element of $\{\mu_\rho\}$ is not a linear combination of others. This together with $\overline{\II}_e \subseteq \{\overline{\mu}_{\rho}\}_{\rho \in [0, 1]} \subseteq \Ibar$ imply that $\Ibar_e = \{\overline{\mu}_{\rho}\}_{\rho \in [0, 1]}$.  
\end{proof}
The proof of Theorem \ref{thm:st1} follows once we state the following lemma. 
\begin{lemma}[{\cite[Proposition 2.3]{imamura2020stationary}}]
	\label{lem:IMS}
Fix arbitrary $n \in \mathbb{Z}_{\geq 1}, \gamma > 0$ and consider independent Bernoulli random variables $Y_1, \dots, Y_n$ with $Y_i \sim \text{Ber}(\frac{q^{i-1}\gamma}{1 + q^{i-1}\gamma})$. 
Define $X = Y_1 + \dots + Y_n$, we have $X \sim q\text{NB}(q^{-n}, -q^{n} \gamma)$ (recall the definition of the $q$-negative Binomial distribution from \eqref{eq:qNB}).
\end{lemma}
\begin{proof}[Proof of Theorem \ref{thm:st1}]
By Lemma \ref{lem:statrelation}, we have $\mathcal{I}^{\HS}  = \{\mu \circ \mathbf{f}_I^{-1}: \mu \in \Ibar \cap \Muf\}$. We proceed to show that $\overline{\II} \subseteq \MM_1(\SS)$, which would imply $\Ibar \cap \MM_1(\SS) = \Ibar$.
Note that one readily checks that $\overline{\MM_1}(\SS)$ is a convex set and $\mubar_\rho \in \overline{\MM}_1(\SS)$ for any $\rho \in [0, 1]$. Using this together with Proposition \ref{prop:extremeinhomo}, we know that $\Ibar \subseteq \MM_1 (\SS)$. Hence, we have $\mathcal{I}^{\HS}  = \{\mu \circ \mathbf{f}_I^{-1}: \mu \in \Ibar\}$ and $\mathcal{I}^{\HS}_e = \{\mu \circ \mathbf{f}_I^{-1}: \mu \in \overline{\II}_e\} = \{\overline{\mu}_\rho \circ \mathbf{f}_{I}^{-1}: \rho \in [0, 1]\}$.
By Lemma \ref{lem:IMS}, we know that $\overline{\mu}_\rho \circ \mathbf{f}_{I}^{-1} = \mu_{\rho}^{\text{HS}}$ for all $\rho \in [0, 1]$. This concludes Theorem \ref{thm:st1}.
\end{proof}
We proceed to prove Theorem \ref{thm:blockingHS}. We assume that $I = J$. Recall that $(\overline{S}_{t})_{t \in \mathbb{Z}_{\geq 0}}$ is the semigroup of the unfused SHS6V model . Let $\mfIuf := \{\nu \in \MM_1 (\SS): \nu \Sbar_I = \nu \circ \tau_I^{-1}\}$. Similar to Lemma \ref{lem:statrelation}, we know that
\begin{equation}\label{eq:statrelation}
\mfIHS = \{\nu \circ \mathbf{f}_I^{-1} : \nu \in \mfIuf \cap \overline{\MM}_1 (\SS)\}.
\end{equation}
\begin{proof}[Proof of Theorem \ref{thm:blockingHS}]
Since $\mfIHS$ is convex, it suffices to show that $ \{\mu^\HS_\rho\}_{\rho \in [0, 1]} \cup \{\mu^\HS_{(n)}\}_{n \in \ZZ} \subseteq \mfIHS$ (recall their definition from \eqref{eq:qnegative} and \eqref{eq:mu*HS}). 
By Proposition \ref{prop:extremeinhomo}, we know that for all $\rho \in [0, 1]$, $\overline{\mu}_\rho \in \Ibar$. This implies that $\overline{\mu}_\rho \overline{S}_I = \overline{\mu}_\rho$. Moreover, it is clear from \eqref{eq:inhomoprod} that we have $\overline{\mu}_\rho = \overline{\mu}_\rho \circ \tau_I^{-1}$.
By Lemma \ref{lem:IMS} and \eqref{eq:statrelation}, we have $\overline{\mu}_\rho \in \mfIuf$ for all $\rho \in [0, 1]$. We proceed to show that $\mu_{*} \in \mfIuf$. To prove this, we show that $\mu_* \overline{S}_t = \mu_* \circ \tau_t^{-1}$ for all $t \in \mathbb{Z}_{\geq 0}$. Note that the unfused SHS6V model is a version of the space-time inhomogeneous S6V model and the proof for $\mu_* \overline{S}_t = \mu_* \circ \tau_t^{-1}$ is similar to the argument for the S6V model in Section \ref{sec:blk}. The key property that we have used in Section \ref{sec:blk} is that $q = \frac{b_1}{b_2}$ at every $(x, t) \in \mathbb{Z} \times \mathbb{Z}_{\geq 0}$. This is certainly true since the parameters $b_1, b_2$ of the S6V model do not depend on the space and time coordinate of the vertex. For the unfused SHS6V model, one can verify that $q = \frac{b_1(x, t)}{b_2(x, t)}$ for every $(x, t) \in \mathbb{Z \times \mathbb{Z}}_{\geq 0}$ using \eqref{eq:b1b2}. Therefore, a similar argument as in Section \ref{sec:blk} also applies here.  

The argument in the previous paragraph shows that $\{\overline{\mu}_\rho\}_{\rho \in [0, 1]} \cup \{\mu_*\} \subseteq \mfIuf$.  One also readily verifies that $\{\overline{\mu}_\rho\}_{\rho \in [0, 1]} \cup \{\mu_*\} \subseteq \overline{\MM}_1 (\SS)$. Using this together with \eqref{eq:statrelation} and Lemma \ref{lem:IMS}, we conclude that $ \{\mu_{\rho}^{\text{HS}} \}_{\rho \in [0, 1]} \cup \{\mu^{\text{HS}}_*\} \subseteq \mfIHS$ (one can check that $\mu_{*}^{\text{HS}} = \mu_* \circ \mathbf{f}_{I}^{-1}$). 
As explained after \eqref{eq:mu*HS}, $\{\blkHS{n}\}_{n \in \mathbb{Z}}$ are the projection of the probability measure $\mu^{\text{HS}}_*$ on the irreducible subspace $G_n$ of the SHS6V model. Hence, $\mu_*^{\text{HS}} \in \mfIHS$ implies that $\{\blkHS{n}\}_{n \in \mathbb{Z}} \subseteq \mfIHS$. This concludes Theorem \ref{thm:blockingHS}. 
\end{proof}

\appendix 

\section{The product Bernoulli measures are stationary for the S6V model}\label{sec:productstat}
In this section, we briefly explain 
why the product Bernoulli measures $\{\mu_\rho\}_{\rho \in [0, 1]}$ are stationary  for the S6V model. It is not easy to verify $\mu_\rho S_1 = \mu_\rho$ via a direct computation, due to complicated dynamics of the S6V model. However, we have the following property of local stationarity from \cite[Appendix A]{aggarwal2018current}.  
\begin{lemma}[Local stationarity]\label{lem:lcs}
For a vertex, let $v$, $h$, $v'$ and $h'$ be the number of lines on the bottom, left, top and right of it. We assign the vertex configuration with the stochastic weights as in Figure \ref{fig:vertexconfig}. \red{Fix $\rho, \zeta \in [0, 1]$.} If we have $(v, h) \sim \text{Ber}(\rho) \otimes \text{Ber}(\zeta)$ and  
\begin{equation}\label{eq:lcs}
(1 - b_1) \rho \red{(1 - \zeta)} = \zeta \red{(1-\rho)} (1 - b_2),
\end{equation}
then $(v', h') \sim \text{Ber}(\rho) \otimes \text{Ber}(\zeta)$. See Figure \ref{fig:lcs} for visualization. 
\end{lemma}
\begin{figure}[ht]
\centering
\begin{tikzpicture}
\draw[thick] (-1, 0) -- (1, 0);
\draw[thick] (0, -1) -- (0, 1);
\node at (0, -1.3) {$v$};
\node at (0, 1.3) {$v'$};
\node at (-1.6, 0) {$h$};
\node at (1.6, 0) {$h'$};
\begin{scope}[xshift = 6cm]
\draw[thick] (-1, 0) -- (1, 0);
\draw[thick] (0, -1) -- (0, 1);
\node at (0, -1.3) {$\Ber(\rho)$};
\node at (0, 1.3) {$\Ber(\rho)$};
\node at (-1.6, 0) {$\Ber(\zeta)$};
\node at (1.6, 0) {$\Ber(\zeta)$};
\end{scope}
\end{tikzpicture}
\caption{A visualization of the above lemma.
}
\label{fig:lcs}
\end{figure}
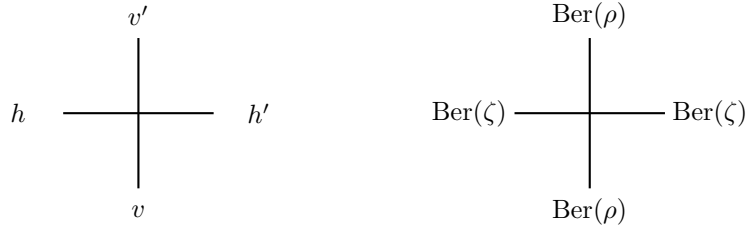

\begin{lemma}\label{lem:prodstat}
Fix arbitrary $\rho \in [0, 1]$, we have $\mu_\rho \in \II$.
\end{lemma}
\begin{proof}
If we have $\eta_0 \sim \mu_\rho$, by iterating Lemma \ref{lem:currentrec}, one can check that $\red{\mathbb{E}^{\eta_0} [K_y] = \zeta := \frac{(1 - b_1)\rho}{(1 - b_2)(1-\rho) + (1 - b_1)\rho}}$, hence the current $K_y  \sim \text{Ber}(\zeta)$. 
Moreover, $K_y$ is independent of $\{\eta_{0} (z)\}_{z > y}$. 
Note that  $\zeta$ and $\rho$ satisfy the relation \eqref{eq:lcs}. We apply Lemma \ref{lem:lcs} for the vertex at $(y+1, 0)$ with $(v, h, v', h) = (\eta_0 (y+1), K_y, \eta_1 (y+1), K_{y+1})$, this
yields $\red{K_{y+1}} \sim \Ber(\zeta)$ and $\eta_1 (y+1) \sim \Ber(\rho)$. Moreover, the random variables $K_{y+1}$ and $\eta_1 (y+1)$ are independent. Keep applying Lemma \ref{lem:lcs} for the vertex at $(i+1, 0)$ with $\red{(v, h, v', h') = (\eta_0 (i+1), K_i, \eta_1(i+1), K_{i+1})}$, where $i = y+1, y+2, \dots$, we obtain that $\{(\eta_{1} (z)\}_{z > y}$ follows the product Bernoulli distribution \red{with density $\rho$}. Since $y \in \mathbb{Z}$ is arbitrary, we conclude that $\eta_1 \sim \mu_\rho$. 
\end{proof}
\begin{remark}\label{rmk:ilcs}
By \eqref{eq:ilcs}, a similar argument as above would also imply that the probability measures $\{\mubar_\rho\}_{\rho \in [0, 1]}$ defined in \eqref{eq:inhomoprod} are stationary distributions for the unfused SHS6V model $(\overline{\eta}_t)_{t \in \mathbb{Z}_{\geq 0}}$.  
\end{remark}


\bibliographystyle{alpha}
\bibliography{ref}

\end{document}